\documentclass[reqno,10pt]{amsart}
\usepackage[left=1in,right=1in,top=.95in,bottom=.95in]{geometry}
\usepackage{amsmath,amssymb,amsthm,mathrsfs,mathtools}
\usepackage{esint}
\usepackage{microtype}
\usepackage{xcolor}
\usepackage[colorlinks=true,linkcolor=blue,citecolor=red,urlcolor=blue]{hyperref}
\usepackage{enumitem}
\numberwithin{equation}{section}
\allowdisplaybreaks[2]

\theoremstyle{plain}
\newtheorem{theorem}{Theorem}[section]
\newtheorem{lemma}[theorem]{Lemma}
\newtheorem{proposition}[theorem]{Proposition}
\newcommand{\tr}{\operatorname{tr}}
\newcommand{\eps}{\epsilon}
\newcommand{\Hess}{\nabla^2}
\newcommand{\Ric}{\operatorname{Ric}}
\newcommand{\osc}{\operatorname*{osc}}
\newcommand{\avg}{\fint}
\newcommand{\dd}{\mathop{}\!\mathrm d}

\begin{document}
\title{Green Function Asymptotics for the $\sigma_2$-Yamabe Problem}

\author{Bin Deng}
\address{Department of Mathematics and Statistics, Wuhan University,
Wuhan 430072, Hubei, P.R. China}
\email{dbmath@whu.edu.cn}
\author{Han Lu}
\address{School of Mathematics and Statistics, Henan University,
Kaifeng 475004, China}
\email{hlu@henu.edu.cn}

\date{August 18, 2026}

\begin{abstract}
We establish, at every pole, the asymptotic expansion of the normalized
$\Gamma_2$-Green function with a finite pole set in every dimension $n\ge5$.
In dimensions $5\le n\le7$ the first correction is a unique nonnegative
constant, whereas dimension $8$ exhibits a Weyl-driven $\sqrt{\log}$ term.
For $n>8$ we construct the finite local curvature parametrix through the first
positive indicial resonance and identify the first genuinely global
coefficient by a relative Newton flux.  
\end{abstract}

\maketitle

\enlargethispage{3pt}
\tableofcontents

\section{Introduction}

The asymptotic expansion of a Green function near one of its poles is one of the basic
local-to-global bridges in the Yamabe problem.  For the conformal Laplacian,
this expansion not only shows the local geometry near that pole, but also
produces the constant term which, in the appropriate range of dimensions and
gauges, is identified with the mass of the associated asymptotically flat
metric.  The purpose of this paper is to develop an analogous expansion theory
at every pole of the $\Gamma_2$-Green function.
A central feature of the nonlinear problem is that the form of the expansion
changes sharply with the dimension.  In particular, dimensions $5\le n\le 7$,
the borderline dimension $n=8$, and dimensions $n>8$ exhibit three genuinely
different asymptotic regimes.

\subsection{The classical Yamabe Green function}

Let $(M^n,g)$ be a smooth compact Riemannian manifold with positive Yamabe
constant, and let $G_p$ denote the Green function of the conformal Laplacian
with pole at $p\in M$.  We normalize $G_p$ so that its leading singular
coefficient is one; then
\[
  G_p(x)\sim r^{2-n},\qquad r=d_g(x,p)\to0.
\]
A particularly useful formulation of the refined expansion is given by
Lee and Parker \cite{LeeParker1987}.  In conformal normal coordinates at the
pole, they construct an asymptotic parametrix for $G_p$ by solving successively
for homogeneous correction terms.  The resulting expansion already displays
a dimension-dependent structure.  In dimensions $n=3,4,5$ (and, in every
dimension, when the metric is locally conformally flat near the pole), the
first two terms take the familiar form
\begin{equation}\label{eq:yamabe-low}
  G_p(x)=r^{2-n}+A+O(r).
\end{equation}
In dimension $n=6$, a Weyl-curvature-driven logarithmic term occurs at the
bounded scale, while in higher dimensions local curvature terms enter before
the constant term.  More schematically, write
\begin{equation}\label{eq:yamabe-LP-schematic}
  G_p(x)=
  \underbrace{r^{2-n}}_{\text{principal singularity}}
  +\underbrace{\mathcal P^{\mathrm{loc}}_p(x)}_{\text{finite local curvature parametrix}}
  +\underbrace{\mathcal G^{\mathrm{glob}}_p(x)}_{\text{first global/free Green term}}
  +o(s_p),
\end{equation}
Here $\mathcal P^{\mathrm{loc}}_p$ is determined by a finite curvature jet,
whereas $\mathcal G^{\mathrm{glob}}_p$ is the first coefficient or profile
left free by the local recursion; $s_p$ is its scale.  In the finite-charge
range, $\mathcal G^{\mathrm{glob}}_p=A_p$.  We use this local--global
distinction for the $\Gamma_2$-Green function below.

\subsection{The \texorpdfstring{$\Gamma_2$}{Gamma-2}-Green function}

Let
\begin{align}
    A_g=\frac1{n-2}
  \left(\operatorname{Ric}_g-\frac{R_g}{2(n-1)}g\right)
\end{align}
be the Schouten tensor, and write $\lambda(A_g)$ for its eigenvalues with
respect to $g$.  We assume throughout that
\begin{equation}\label{eq:admissible-background}
  \lambda(A_g)\in\Gamma_2 := \{\lambda\in\mathbb R^n\ : \ \sigma_1(\lambda)>0, \sigma_2(\lambda)>0\}
\end{equation}
where $\Gamma_2$ is the G\aa rding cone.
For a positive function $v$, set
\begin{align}
    g_v=v^{\frac4{n-2}}g.
\end{align}
Let $S=\{p_1,\ldots,p_L\}$ be a nonempty finite subset of $M$.
A normalized $\Gamma_2$-Green function with pole set $S$ is a positive
function $v$ on $M\setminus S$ satisfying, in the viscosity sense,
\begin{equation}\label{eq:green-equation}
  \lambda(A_{g_v})\in\partial\Gamma_2
  \quad\text{on }M\setminus S,
  \qquad
  \lim_{x\to p_i}d_g(x,p_i)^{n-2}v(x)=1
  \quad(1\le i\le L).
\end{equation}

Earlier Green-function results were established in the spherical model.
When $(M,g)$ is conformally equivalent to the standard sphere and there is a
single pole, Chang--Gursky--Yang \cite{ChangGurskyYang2002Annals} treated the
case $n=4$ and $\Gamma=\Gamma_2$ under $C^{1,1}$ regularity.  For general
conformally invariant cones, the corresponding one-pole theory was proved in
the Lipschitz class by Li \cite{LiDegenerate2007,LiGradient2009} and in the
continuous viscosity class by Li--Nguyen--Wang \cite{LiNguyenWang2018}.

The broader motivation comes from the $\sigma_k$-Yamabe problem, initiated by
Viaclovsky \cite{Viaclovsky2000}.  Compactness is known in several important
settings, including $k=2$ in dimension four
\cite{ChangGurskyYang2002Apriori}, the locally conformally flat case
\cite{LiLi2003}, the range $k>n/2$ \cite{GurskyViaclovsky2007}, and $k=n/2$
under a lower Ricci curvature bound \cite{LiNguyenCompactness2014}.  By
contrast, for $k=2$ and $n\ge5$ the general compactness theory remains
incomplete.  Since rescaled blow-up sequences are expected to converge to
$\Gamma_2$-Green functions, refined asymptotics at their poles are directly
relevant to the blow-up analysis of the $\sigma_2$-Yamabe equation.

Li and Nguyen \cite{LiNguyen2023} established existence and uniqueness for
every finite pole set and every prescribed collection of positive strengths,
for conformally invariant cones in the relevant range.  For $\Gamma_2$ one has
\begin{align}
    \mu^+_{\Gamma_2}=\frac{n-2}{2},
\end{align}
so that $\mu^+_{\Gamma_2}>1$ precisely when $n\ge 5$.  Consequently, under
\eqref{eq:admissible-background}, the normalized $\Gamma_2$-Green function in
\eqref{eq:green-equation} exists and is unique for every $n\ge5$; moreover it
belongs to $C^{1,1}_{\mathrm{loc}}(M\setminus S)$.  Their construction also
shows that the conformal metric $g_v$ has an asymptotically flat end at each
pole and,
importantly for the analysis below, provides smooth strict sub- and
super-solutions arising from a natural elliptic regularization of the
degenerate equation.

This theory gives the leading singularity and global well-posedness, but not
the next asymptotic terms.  A linear parametrix does not apply: the equation
lies on $\partial\Gamma_2$ and the background curvature interacts with the
singular conformal factor at the scales to be detected.

There is also a substantial literature on refined asymptotics for isolated
singular solutions of the nondegenerate $\sigma_k$-Yamabe equation; see, for
example, Han--Li--Teixeira \cite{HanLiTeixeira2010} and the higher-order
expansions of Han--Li--Li \cite{HanLiLi2021}.  Those results provide an
important conceptual comparison, but the Green equation
\eqref{eq:green-equation} is of a different type: its curvature vector lies on
the boundary of the admissible cone, and the corresponding linearized theory
is intrinsically degenerate.

\subsection{The natural Green factor and the dimensional threshold}

It is convenient to replace $v$ by the conformal factor adapted to the
$\sigma_2$ scaling,
\begin{equation}\label{eq:U-def}
  U:=v^{\frac{n-4}{2(n-2)}},
  \qquad
  \alpha:=\frac{n-4}{2}.
\end{equation}
Then
\begin{equation}\label{eq:metric-U}
  g_v=U^{\frac8{n-4}}g,
  \qquad
  U(x)\sim d_g(x,p_i)^{-\alpha}
  \quad\text{as }x\to p_i,
  \qquad 1\le i\le L.
\end{equation}
The exponent $\alpha$ makes the dimensional splitting transparent.  In
conformal normal gauge the first curvature contribution to the inverted end
occurs at relative degree $2$.  Hence
\begin{align}
    \alpha<2 \quad (5\le n\le7),\qquad
  \alpha=2 \quad (n=8),\qquad
  \alpha>2 \quad (n>8).
\end{align}
When $\alpha<2$, the curvature correction is of lower order than the bounded scale and a scalar constant can be isolated directly.  At $n=8$ the Weyl term lands exactly at the critical scale and produces a nonlinear resonance.  For $n>8$, local curvature corrections necessarily appear before the bounded scale, so the correct analogue of the Lee--Parker expansion is a nonlinear local parametrix rather than a two-term formula.

The following theorem gives the expansion in the three dimensional
regimes.  In high dimensions, the nonflat branch is carried through the first
positive indicial root and separates the local resonant term from the first
genuinely global coefficient.

\begin{theorem}\label{thm:main-working}
Let $(M^n,g)$ be a smooth compact manifold, $n\ge5$, satisfying
\eqref{eq:admissible-background}.  Let $S\subset M$ be a nonempty finite set,
let $v$ be the normalized $\Gamma_2$-Green function with pole set $S$ defined
by \eqref{eq:green-equation}, and let $U$ and $\alpha$ be defined by
\eqref{eq:U-def}.  Fix $p\in S$ and set $r=d_g(x,p)$.  Then the following
conclusions hold as $x\to p$.

\begin{enumerate}[label=\textup{(\roman*)},leftmargin=2.2em]
\item \textbf{Dimensions $5\le n\le7$.}
There exists a unique constant $A_p\ge0$ such that
\begin{equation}\label{eq:main-low}
  U(x)=r^{-\alpha}+A_p+o(1),
  \qquad x\to p.
\end{equation}

\item \textbf{The borderline dimension $n=8$.}
Choose a conformal normal representative near $p$, write its coordinate radius
as $\rho$, and denote the corresponding Green factor by $U_{cn}$.
If $|W_g(p)|\ne0$, then
\begin{equation}\label{eq:main-n8-generic}
  U_{cn}(x)
  =\rho^{-2}
   +\frac{|W_g(p)|}{4\sqrt{210}}
      \sqrt{\log\frac1\rho}
   +o(1).
\end{equation}
If $W_g(p)=0$, there exists a unique constant $A_p\ge0$ such that
\begin{equation}\label{eq:main-n8-flat}
  U_{cn}(x)=\rho^{-2}+A_p+o(1).
\end{equation}

\item \textbf{Dimensions $n>8$: the nonflat branch.}
Fix a conformal normal representative and suppose that the leading
Kelvin--Schouten cell $C_0$ in \eqref{eq:C0-definition} is nonzero.  Set $m=n-6$.
Then there exist a positive function $\psi_*$ with
$\int_{\mathbb S^{n-1}}\psi_*=1$, locally determined smooth functions
$\Phi_2,\ldots,\Phi_m$, a locally determined number $\ell_p$, and a number
$A_p$ determined by the global $\Gamma_2$-Green function such that
\begin{equation}\label{eq:main-high}
\begin{split}
  U_{cn}(\rho,\theta)
  =\rho^{-\alpha}\Bigg[&1+\sum_{j=2}^{m-1}\rho^j\Phi_j(\theta)\\
   &+\rho^m\left(
       \ell_p\psi_*(\theta)\log\frac1\rho
       +\Phi_m(\theta)+A_p\psi_*(\theta)
     \right)+o(\rho^m)\Bigg].
\end{split}
\end{equation}
Here $m=n-6$ is the first positive indicial root, and $A_p$ is the first global coefficient.
\end{enumerate}
All local cells and their coefficients are determined by the conformal-normal
jet at the chosen pole $p$, whereas the free coefficient $A_p$ may depend on
the full Green function and hence on the entire pole set $S$.
\end{theorem}

\begin{theorem}
\label{thm:main-high-flatness}
Under the hypotheses and notation of
Theorem~\ref{thm:main-working}, assume $n>8$ and, at the fixed pole $p$,
$C_0=0$.
Let $q\ge1$ be the least index such that $C_q\not\equiv0$, and set
$j=q+2$.  Then:
\begin{enumerate}[label=\textup{(\alph*)},leftmargin=2.2em]
\item If $j<\alpha$, there is a unique admissible
$\phi_j\in C^{1,1}(\mathbb S^{n-1})$ such that
\begin{equation}\label{eq:main-high-flat-subcritical}
 U_{cn}(\rho,\theta)
 =\rho^{-\alpha}
  +\frac\alpha2\phi_j(\theta)\rho^{j-\alpha}
  +o(\rho^{j-\alpha}).
\end{equation}
No expansion beyond this first nonzero cell is asserted without
nondegeneracy of its limiting Newton tensor.

\item If $j=\alpha$, let $\kappa_q$ be the quadratic obstruction defined in
\eqref{hd:eq:kappaq-def}.  If $\kappa_q>0$, then for the canonical angular
function $\psi_0$,
\begin{equation}\label{eq:main-high-flat-critical-positive}
 U_{cn}(\rho,\theta)
 =\rho^{-\alpha}
 +\sqrt{\frac{\kappa_q}{2(n-1)}}\sqrt{\log\frac1\rho}
 +\frac\alpha2\psi_0(\theta)+o(1).
\end{equation}
If $\kappa_q=0$, there are a locally determined angular function
$\phi_{\rm rem}$ and a unique constant $A_p\ge0$ such that
\begin{equation}\label{eq:main-high-flat-critical-zero}
 U_{cn}(\rho,\theta)
 =\rho^{-\alpha}+\frac\alpha2\phi_{\rm rem}(\theta)+A_p+o(1).
\end{equation}

\item If $j>\alpha$, there is a unique $A_p\ge0$ such that
\begin{equation}\label{eq:main-high-flat-supercritical}
 U_{cn}(\rho,\theta)=\rho^{-\alpha}+A_p+o(1).
\end{equation}
The same conclusion holds if all higher Kelvin--Schouten cells vanish.
\end{enumerate}
\end{theorem}

\subsection{Organization}

Section~\ref{sec:green-end-preliminaries} sets up the Green end.
Sections~\ref{sec:subcritical}--\ref{sec:endpoint8} prove the low-dimensional
and borderline cases.  Sections~\ref{sec:high-leading-correction}--\ref{sec:high-moving-scale}
develop the high-dimensional parametrix, treat higher-order flatness, and
complete the proof of the main theorems.

\section{Green-end formulation and analytic preliminaries}
\label{sec:green-end-preliminaries}
\label{sec:normalization-kelvin}

\subsection{Tensor conventions and the natural Green factor}

For a symmetric matrix $B\in \mathbb R^{n\times n}$, we
write
\begin{equation}\label{eq:sigma2-T1-conventions}
  \sigma_1(B)=\tr B,
  \qquad
  \sigma_2(B)=\frac12\bigl((\tr B)^2-\tr(B^2)\bigr),
  \qquad
  T_1(B)=\sigma_1(B)I-B.
\end{equation}
The contraction of two symmetric matrices is denoted by
$B:C=\tr(BC)$.  Throughout the paper, $g_{\mathrm E}$ denotes the
Euclidean metric on $\mathbb R^n$, and $\delta_{ij}$ denotes its coordinate
components; the unindexed symbol $\delta$ is reserved for scalar parameters.
If $h$ is a Riemannian metric and $S$ is a symmetric
$(0,2)$-tensor, then $S^\sharp=h^{-1}S$ denotes the associated
$h$-self-adjoint endomorphism.  In particular,
$(\nabla_h^2f)^\sharp$ is the Hessian of $f$ viewed as an
$h$-self-adjoint endomorphism.

Temporarily write the original representative and Green function as
$(g_0,v_0)$.  Fix an arbitrary pole $p\in S$ and choose $r_p>0$ so that
$B_{2r_p}^{g_0}(p)\cap S=\{p\}$.  All constructions below are local to the
Green end associated with this pole; the other poles enter only through the
global coefficients.  We therefore suppress the pole index throughout the
proof.  Set
$U_0=v_0^{(n-4)/(2(n-2))}$ and $\alpha=(n-4)/2$.  Then
$g_{v_0}=U_0^{8/(n-4)}g_0$ and
$U_0(x)\sim d_{g_0}(x,p)^{-\alpha}$.

\subsection{Conformal-normal coordinates}

Fix an integer $N$ larger than all orders needed in the finite asymptotic
construction.  By the conformal-normal-coordinate construction of Lee--Parker
\cite{LeeParker1987}, there is a positive smooth function $\psi$, defined near
$p$ and, after shrinking the coordinate neighborhood, extended smoothly to
$M$ so that $\psi\equiv1$ near $S\setminus\{p\}$, such that
\begin{equation}\label{eq:psi-normalization}
  g:=\psi^{\frac4{n-2}}g_0,
  \qquad
  \psi(p)=1,
  \qquad
  d\psi(p)=0,
\end{equation}
and, in $g$-normal coordinates $x=(x^1,\ldots,x^n)$ centered at $p$, writing
$\rho:=|x|$,
\begin{equation}\label{eq:conformal-normal-determinant}
  \det(g_{ij}(x))=1+O(|x|^N).
\end{equation}
The same Green metric is represented in this gauge by
\begin{equation}\label{eq:green-factor-gauge-change}
  v:=\psi^{-1}v_0,
  \qquad
  U:=v^{\frac{n-4}{2(n-2)}}
    =\psi^{-\frac{\alpha}{n-2}}U_0,
\end{equation}
for then
\begin{equation}\label{eq:same-green-metric}
  v^{\frac4{n-2}}g=v_0^{\frac4{n-2}}g_0.
\end{equation}
Because of \eqref{eq:psi-normalization} and the choice of the extension, all
normalizations in \eqref{eq:green-equation} are preserved.  Throughout the local analysis we
work with the normalized triple $(g,v,U)$ and suppress any further decoration.
We return to the original representative only after the asymptotic expansion
has been established.

The determinant condition \eqref{eq:conformal-normal-determinant}, with
$N\ge3$, implies $\operatorname{Ric}_g(p)=0$ and hence
$\operatorname{Rm}_g(p)=W_g(p)$.  With a sign
$\varsigma\in\{+1,-1\}$ depending only on the convention for the Riemann
tensor, the normal-coordinate expansion is therefore
\begin{equation}\label{eq:cn-metric-quadratic-expansion}
  g_{ij}(x)
  =\delta_{ij}
   +\varsigma\frac13 W_{ikjl}(p)x^kx^l
   +Q_{ij}(x),
\end{equation}
where
\begin{equation}\label{eq:cn-metric-cubic-remainder}
  |\partial_x^\beta Q(x)|\le C_\beta\rho^{3-|\beta|},
  \qquad |\beta|\le3.
\end{equation}
Only the quadratic term in \eqref{eq:cn-metric-quadratic-expansion} is needed
to identify the leading curvature contribution at the inverted end.  Higher
finite-order expansions will be invoked later when the high-dimensional
parametrix is constructed.

\subsection{Kelvin inversion and normalized end variables}

Let
\begin{equation}\label{eq:kelvin-map}
  \iota(y):=\frac{y}{|y|^2},
  \qquad
  R:=|y|,
  \qquad
  \theta:=\frac{y}{R}.
\end{equation}
Thus $x=\iota(y)$, $\rho=R^{-1}$, and the punctured neighborhood of $p$
becomes an exterior region $\{R>R_0\}\subset\mathbb R^n$.  Define
\begin{equation}\label{eq:F-and-ginfty}
  F(y):=R^{-\alpha}U(\iota(y)),
  \qquad
  g_\infty:=R^4\iota^*g.
\end{equation}
The normalization in \eqref{eq:green-equation} gives
\begin{equation}\label{eq:F-to-one}
  F(y)\longrightarrow1
  \qquad\text{as }R\longrightarrow\infty.
\end{equation}
Moreover,
\begin{equation}\label{eq:inverted-green-metric-F}
  \iota^*g_v
  =F^{\frac8{n-4}}g_\infty.
\end{equation}
It is convenient to make one further change of unknown:
\begin{equation}\label{eq:W-definition}
  W:=F^{-\frac4{n-4}}=F^{-\frac2\alpha}.
\end{equation}
Then
\begin{equation}\label{eq:inverted-green-metric-W}
  \iota^*g_v=W^{-2}g_\infty,
  \qquad
  W\longrightarrow1,
  \qquad
  F=W^{-\alpha/2}.
\end{equation}
For a function or tensor field \(T\) on the exterior region
\(\{R=|y|\geq R_0\}\), we write
\[
  T=O_k(R^{-q})
\]
if, in the coordinates \(y\),
\[
  |\partial_y^\beta T(y)|
  \leq C_\beta R^{-q-|\beta|},
  \qquad |\beta|\leq k.
\]
For tensor fields, the estimate is understood componentwise; unless another
metric is indicated, the norms and contractions are taken with respect to
$g_{\mathrm E}$.

\begin{lemma}
\label{lem:inverted-background}
In the coordinates above,
\begin{equation}\label{eq:ginfty-leading-expansion}
  (g_\infty)_{ij}(y)
  =\delta_{ij}
   +\varsigma\frac13W_{ikjl}(p)\frac{y^ky^l}{R^4}
   +O_3(R^{-3}).
\end{equation}
In particular,
\begin{equation}\label{eq:ginfty-basic-AF}
  g_\infty-g_{\mathrm E}=O_3(R^{-2}),
  \qquad
  (A_{g_\infty})^\sharp=O_1(R^{-4}).
\end{equation}
More precisely,
\begin{equation}\label{eq:Aginfty-leading-Weyl}
  (A_{g_\infty})^\sharp
  =R^{-4}C_0(\theta)+O_1(R^{-5}),
\end{equation}
where
\begin{equation}\label{eq:C0-definition}
  (C_0)_{ij}(\theta)
  :=\varsigma\frac23W_{ikjl}(p)\theta^k\theta^l.
\end{equation}
The leading tensor satisfies
\begin{equation}\label{eq:C0-algebra-basic}
  \tr C_0=0,
  \qquad
  C_0(\theta)\theta=0.
\end{equation}
Its homogeneous extension
\begin{equation}\label{eq:Cinfty-definition}
  C_\infty(y):=R^{-4}C_0(\theta)
  =\varsigma\frac23W_{ikjl}(p)\frac{y^ky^l}{R^6}
\end{equation}
satisfies
\begin{equation}\label{eq:Cinfty-structure}
  \tr C_\infty=0,
  \qquad
  C_\infty(y)y=0,
  \qquad
  \operatorname{div}_{g_{\mathrm E}} C_\infty=0
  \quad\text{in }\mathbb R^n\setminus\{0\}.
\end{equation}
Consequently, if $W_g(p)=0$, then
\begin{equation}\label{eq:Aginfty-Weyl-flat-improvement}
  g_\infty-g_{\mathrm E}=O_3(R^{-3}),
  \qquad
  (A_{g_\infty})^\sharp=O_1(R^{-5}).
\end{equation}
\end{lemma}

\begin{proof}
The conformal-normal expansion \eqref{eq:cn-metric-quadratic-expansion}--\eqref{eq:cn-metric-cubic-remainder} is standard; see \cite{LeeParker1987}.  Using $D\iota=R^{-2}(I-2\theta\otimes\theta)$ and the Weyl symmetries gives \eqref{eq:ginfty-leading-expansion}, while the standard Ricci--Schouten expansion gives \eqref{eq:Aginfty-leading-Weyl}.  The remaining identities follow directly from the algebraic symmetries and trace-freeness of the Weyl tensor; if $W_g(p)=0$, the leading terms vanish.
\end{proof}

\subsection{The equation in the \texorpdfstring{$W$}{W}-variable}

If $h$ is a smooth background metric and $W>0$, the conformal change
$\widehat h=W^{-2}h$ satisfies
\begin{equation}\label{eq:Schouten-W-covariant}
  A_{\widehat h}
  =A_h+W^{-1}\nabla_h^2W
       -\frac12W^{-2}|\nabla W|_h^2h.
\end{equation}
Raising $A_{\widehat h}$ with $\widehat h$ (while all $\sharp$
operations on the right below are taken with $h$), define
\begin{equation}\label{eq:calA-h-W}
  \mathcal A_h[W]
  :=W^2(A_h)^\sharp
    +W(\nabla_h^2W)^\sharp
    -\frac12|\nabla W|_h^2I.
\end{equation}
The Green equation on the inverted end becomes
\begin{equation}\label{eq:W-end-equation}
  \lambda\bigl(\mathcal A_{g_\infty}[W]\bigr)
  \in\partial\Gamma_2
  \qquad\text{in }\{R>R_0\},
\end{equation}
with $W\to1$ at infinity.  This is the basic equation used throughout the
paper.

A perturbation $W=1+O(R^{-\mu})$ contributes at order $R^{-\mu-2}$,
so the leading curvature term in \eqref{eq:Aginfty-leading-Weyl} is balanced
at $\mu=2$.  Since
$U=\rho^{-\alpha}W^{-\alpha/2}$, the absolute bounded scale is
$\mu=\alpha$.  Their comparison gives the three regimes in
Theorem~\ref{thm:main-working}.

\subsection{\texorpdfstring{$\Gamma_2$}{Gamma-2}-Green-function theory at the Green end}
\label{subsec:gamma2-green-theory}

We conclude the common setup by isolating the $\Gamma_2$-Green-function properties used in the sequel.  We retain the notation of
Section~\ref{sec:normalization-kelvin}; in particular,
$g_{\mathrm{cn}}$, $v_{\mathrm{cn}}$, $U_{\mathrm{cn}}$, $F$, $W$,
and $g_\infty$ denote the conformal-normal and Kelvin-end quantities
introduced there.

\begin{proposition}
\label{prop:gamma2-green-theory}
Under the standing assumptions, the following statements hold.

\begin{enumerate}[label=\textup{(\roman*)},leftmargin=2.2em]
\item The normalized $\Gamma_2$-Green function with pole set $S$ is unique and
satisfies
\begin{equation}\label{eq:LN-C11}
  v_{\mathrm{cn}}
  \in C^{1,1}_{\mathrm{loc}}(M\setminus S).
\end{equation}
Moreover, after decreasing the decay exponent if necessary, there are
constants
\[
  0<\beta_0<\min\{2,\alpha\},
  \qquad C>0,
  \qquad R_0>0,
\]
such that
\begin{equation}\label{eq:LN-subcritical-decay}
  |F(y)-1|+|W(y)-1|
  \leq C|y|^{-\beta_0}
  \qquad\text{for }|y|\geq R_0.
\end{equation}

\item There exist sequences $r_j\downarrow0$, $\delta_j\downarrow0$,
and positive functions
\[
  v_{\mathrm{cn},j}
  \in C^\infty
  \left(M\setminus
    \bigcup_{q\in S}\overline{B_{r_j}(q)}\right)
\]
such that the metrics
\[
  g_j:=v_{\mathrm{cn},j}^{\frac4{n-2}}g_{\mathrm{cn}}
\]
satisfy
\begin{equation}\label{eq:LN-unscaled-regularization}
  \lambda(A_{g_j})\in\Gamma_2,
  \qquad
  \sigma_2\bigl(\lambda(A_{g_j})\bigr)^{1/2}=\delta_j
\end{equation}
on their domains, and
\begin{equation}\label{eq:LN-local-convergence}
  v_{\mathrm{cn},j}\longrightarrow v_{\mathrm{cn}}
  \quad\text{in }C^{1,\gamma}_{\mathrm{loc}}(M\setminus S)
  \quad\text{for every }\gamma\in(0,1).
\end{equation}

\item Define
\[
  U_{\mathrm{cn},j}
  :=v_{\mathrm{cn},j}^{\frac{n-4}{2(n-2)}},
  \qquad
  F_j(y):=|y|^{-\alpha}U_{\mathrm{cn},j}(\iota(y)),
  \qquad
  W_j:=F_j^{-2/\alpha}.
\]
For $R>0$, let
\[
  D_R(x):=Rx,
  \qquad
  g_R:=R^{-2}D_R^*g_\infty.
\]
Fix $K\Subset\mathbb R^n\setminus\{0\}$ and choose a normalization
$\varepsilon_R>0$.  Write
\[
  W(Rx)=1+\varepsilon_RP_R(x),
  \qquad
  W_j(Rx)=1+\varepsilon_RP_{R,j}(x).
\]
For all sufficiently large $j$, the functions $P_{R,j}$ are smooth on
$K$, $1+\varepsilon_RP_{R,j}>0$ there, and
\begin{equation}\label{eq:LN-scaled-convergence}
  P_{R,j}\longrightarrow P_R
  \quad\text{in }C^{1,\gamma}(K)
  \quad\text{for every }\gamma\in(0,1).
\end{equation}
Set $W_{R,j}:=1+\varepsilon_RP_{R,j}$.  The normalized augmented Hessian
\begin{equation}\label{eq:LN-scaled-M}
  \mathcal M_{R,j}
  :=\bigl(\nabla^2_{g_R}P_{R,j}\bigr)^\sharp
   -\frac{\varepsilon_R}{2W_{R,j}}
      |\nabla P_{R,j}|_{g_R}^2I
   +\frac{W_{R,j}}{\varepsilon_R}(A_{g_R})^\sharp
\end{equation}
satisfies
\begin{equation}\label{eq:LN-normalized-strict-equation}
  \lambda(\mathcal M_{R,j})\in\Gamma_2,
  \qquad
  \sigma_2(\mathcal M_{R,j})
  =\tau_{R,j}W_{R,j}^{-2},
  \qquad
  \tau_{R,j}:=\left(\frac{R^2\delta_j}{\varepsilon_R}\right)^2.
\end{equation}
In particular, for each fixed $R$,
\begin{equation}\label{eq:LN-tau-to-zero}
  \tau_{R,j}\longrightarrow0
  \qquad\text{as }j\to\infty.
\end{equation}
\end{enumerate}
\end{proposition}

\begin{proof}
Statements \textup{(i)}--\textup{(ii)} are precisely the finite-pole results
of \cite[Theorems~1.2 and~1.4, and Section~4]{LiNguyen2023}, rewritten at the
fixed pole in the conformal-normal representative; decreasing the decay
exponent gives \eqref{eq:LN-subcritical-decay}.  For \textup{(iii)}, Kelvin
inversion, dilation, and \eqref{eq:calA-h-W}, together with
$g_R=R^{-2}D_R^*g_\infty$, give
\[
 \mathcal A_{g_R}[W_{R,j}]
 =R^2D_R^*\mathcal A_{g_\infty}[W_j]
 =\varepsilon_RW_{R,j}\mathcal M_{R,j}.
\]
Here $D_R^*$ denotes the natural pullback of the $(1,1)$-endomorphism,
and the factor $R^2$ comes from the constant rescaling of the metric.
The transformed strict equation is
\[
 \lambda\bigl(\mathcal A_{g_R}[W_{R,j}]\bigr)\in\Gamma_2,
 \qquad
 \sigma_2\bigl(\mathcal A_{g_R}[W_{R,j}]\bigr)=R^4\delta_j^2.
\]
The homogeneity of $\sigma_2$ gives \eqref{eq:LN-normalized-strict-equation}, and \eqref{eq:LN-local-convergence} with $\delta_j\to0$ gives the remaining assertions.
\end{proof}

\subsection{Logical architecture of the proof}

We now explain how Sections~\ref{sec:subcritical}--\ref{sec:high-moving-scale}
fit together to prove Theorems~\ref{thm:main-working} and
\ref{thm:main-high-flatness}.  Two relative homogeneities organize the
argument.  A mode $R^{-\mu}\phi(\theta)$ in $W-1$ enters the Hessian term
$W(\nabla_{g_\infty}^2W)^\sharp$ in \eqref{eq:calA-h-W} at order
$R^{-\mu-2}$, whereas the background term
$W^2(A_{g_\infty})^\sharp$ begins at order $R^{-4}$; hence the first local
curvature balance is $\mu=2$.  On the other hand,
\[
 U=\rho^{-\alpha}W^{-\alpha/2},
 \qquad \alpha=\frac{n-4}{2},
\]
so \(\mu=\alpha\) is the scale of an absolute bounded correction to \(U\).
Thus \(\alpha<2\), \(\alpha=2\), and \(\alpha>2\) lead respectively to a
finite-charge term, a borderline resonance, and a local curvature expansion
before the first bounded term.

\smallskip
\noindent
\emph{Finite charge and the borderline dimension.}
Let \(j\ge2\) be the first relative curvature degree, so its contribution to
\(W\) has order \(R^{-j}\).  If \(j>\alpha\), in particular when
\(5\le n\le7\), comparison barriers give
\[
 |W-1|\le CR^{-\alpha},
 \qquad
 W(Rx)=1+R^{-\alpha}P_R(x).
\]
Compactness and the isolated-singularity classification imply that every
blow-down has a subsequence converging to \(-c|x|^{-\alpha}\).  To remove the
subsequence, average the radial flux of the Newton field
\(T_1((\nabla_{g_R}^2P_R)^\sharp)\nabla^{g_R}P_R\) and call the result
\(Q(R)\).  The null-Lagrangian identity gives
\[
 |Q(tR)-Q(R)|
 \le C\bigl(R^{-\alpha}+R^{-2}+R^{\alpha-j}\bigr),
 \qquad
 Q[-c|x|^{-\alpha}]
 =(n-1)|\mathbb S^{n-1}|\alpha^2c^2.
\]
for \(1\le t\le2\).  Since the drift is dyadically summable, \(Q(R)\)
converges and its model value forces all tangents to have the same \(c\).
Thus \(P_R\to-c|x|^{-\alpha}\) and \(A=\alpha c/2\).

When \(n=8\), \(\alpha=2\) and the leading Weyl curvature is resonant.  If
\(W_g(p)\ne0\), the identities
\(\sigma_2(C_0)=-|C_0|^2/2\) and
\(\operatorname{div}C_\infty=C_\infty x=0\) cancel the mixed
Hessian--curvature term and yield
\[
 Q_8(tR)-Q_8(R)=\kappa_p\log t+o(1),
 \qquad
 \kappa_p=\frac{|\mathbb S^7|}{120}|W_g(p)|^2.
\]
The charge is quadratic in the tangent amplitude, so
\(Q_8(R)\sim\kappa_p\log R\) produces the \(\sqrt{\log R}\) law, with
\(c=|W_g(p)|/(4\sqrt{210})\).  If \(W_g(p)=0\), then \(j>\alpha\) and the
finite-charge argument applies.

\smallskip
\noindent
\emph{The nonflat high-dimensional branch.}
Assume \(n>8\) and \(C_0\not\equiv0\).  The weighted \(C^2\)
degree-two ansatz for \(W\) gives the nonlinear spherical cell
\[
 \mathcal S_4:=\mathcal K_2[\phi_2]+C_0,
 \qquad
 \lambda(\mathcal S_4)\in\partial\Gamma_2,
 \qquad
 N:=T_1(\mathcal S_4)\ge c_0I.
\]
The positive Newton tensor makes the indicial pencil
\(\mathcal L_\mu\phi:=N:\mathcal K_\mu[\phi]\) elliptic, and its divergence
form gives
\[
 \mathcal L_\mu^*=\mathcal L_{n-6-\mu},
 \qquad
 \ker\mathcal L_\mu=\{0\}\quad(0<\mu<n-6).
\]
Hence \(m=n-6\) is the first positive root.  The equations
\[
 \mathcal L_k\phi_k
 =F_k(C_0,\ldots,C_{k-2};\phi_2,\ldots,\phi_{k-1}),
 \quad 3\le k<m,
\]
form a triangular recursion, and comparison barriers identify the resulting
parametrix with the true expansion.  Thus all coefficients below \(m\) are
local, determined by a finite curvature jet at \(p\).  At \(m\), let
\(\psi_*>0\) span \(\ker\mathcal L_m\), with
\(\int_{\mathbb S^{n-1}}\psi_*=1\).  Since
\(\ker\mathcal L_m^*=\operatorname{span}\{1\}\), the spherical mean fixes
the logarithmic coefficient:
\[
 a_{\log}
 =-\frac{\int_{\mathbb S^{n-1}}\mathcal G_m\,d\theta}{\tau_*},
 \qquad
 \tau_*:=\langle1,\mathcal L_m'\psi_*\rangle<0.
\]
The coefficient of \(R^{-m}\psi_*\), however, is global.

\smallskip
\noindent
\emph{Flux selection and the refined ABP estimate.}
Put \(q=\log W\), \(q_0=\log P_{\rm res}\), \(h=q-q_0\), and, along
\(q_t=q_0+th\), set
\(\mathscr A_t=e^{-2q_t}\mathcal A_{g_\infty}[e^{q_t}]\) and
\(J_h=(\int_0^1T_1(\mathscr A_t)\,dt)\nabla h\).  This path-averaged Newton
field satisfies
\[
 \operatorname{div}_{g_\infty}J_h
 =-\sigma_2(\mathscr A_0)
 +(n-4)\int_0^1T_1(\mathscr A_t)(\nabla q_t,\nabla h)\,dt.
\]
After averaging \(\langle J_h,\nabla r\rangle\) on \(R<r<2R\), this identity
and the critical Caccioppoli estimate give a limiting flux
\(\mathfrak F_\infty\), with
\[
 R^mh\longrightarrow B\psi_*
 \quad\Longrightarrow\quad
 \mathfrak F_\infty=\tau_*B,
 \qquad
 B_*=\frac{\mathfrak F_\infty}{\tau_*}.
\]
This selects a candidate but does not yet give uniform convergence.  If a
critical remainder survives after subtracting \(B_*R^{-m}\psi_*\), normalize
it on scales \(R_k\to\infty\).  Although uniform ellipticity is then available
only for contact jets of size \(O(\delta_k^{-1})\), the finite-jet
Pucci--ABP estimate gives
\[
 \|v_k\|_{L^\infty(K)}
 \le C\|v_k\|_{L^2(K')}^\gamma+C\delta_k^{\beta_*/2}.
\]
The weak limit solves
\(\operatorname{div}(N_\infty\nabla v_\infty)=0\) with zero conormal flux, so
the critical-cylinder Liouville theorem gives \(v_\infty=0\).  The displayed
estimate upgrades this \(L^2_{\rm loc}\)-convergence to local uniform
convergence, contradicting the normalization.  Hence \(B_*\) is the actual
and unique global coefficient.

\smallskip
\noindent
\emph{Higher-order flatness.}
If \(C_q\) is the first nonzero Kelvin--Schouten coefficient, set \(j=q+2\).
For \(j<\alpha\), the cell
\(\mathcal K_j[\phi_j]+C_q\in\partial\Gamma_2\) gives the first local
correction; without nondegeneracy of its Newton tensor no further expansion
is asserted.  At \(j=\alpha\), the zero-mean corrector defines
\[
 B^\circ=C_q+\mathcal K_\alpha[\psi_0^\circ],
 \qquad
 \kappa_q=-\fint_{\mathbb S^{n-1}}\sigma_2(B^\circ)\,d\theta\ge0.
\]
If \(\kappa_q>0\), the modified charge produces the
\(\sqrt{\log R}\) law.  If \(\kappa_q=0\), the critical cell is removable
and the remaining term is finite charge.  Finally, \(j>\alpha\) is already
covered by the finite-charge mechanism.  These alternatives complete the
proofs of Theorems~\ref{thm:main-working} and
\ref{thm:main-high-flatness}.


\section{The subcritical dimensions \texorpdfstring{$5\le n\le7$}{5 <= n <= 7}}
\label{sec:subcritical}

\subsection{Barrier package and critical decay}
\label{sec:low-barriers}

Throughout this section $5\le n\le7$ and
$\alpha=(n-4)/2\in(0,2)$.  We use the standard admissible comparison
principle on bounded annuli together with multiplicative compensation at
infinity.

For $g[h]=h^{8/(n-4)}\bar g=e^{2\omega}\bar g$, set
\[
 \mathcal B_{\bar g}[h]
 :=\bar g^{-1}\!\left(A_{\bar g}-\nabla^2\omega
 +d\omega\otimes d\omega-\frac12|\nabla\omega|^2\bar g\right).
\]
For a Euclidean radial profile, with $s=-r\omega'$, one has
\begin{equation}\label{eq:low-radial-eigenvalues}
 \lambda_r=\frac{s'}r-\frac{s(2-s)}{2r^2},\qquad
 \lambda_\theta=\frac{s(2-s)}{2r^2},\qquad
 \sigma_2=(n-1)\lambda_\theta\left(\frac{s'}r+\alpha\lambda_\theta\right).
\end{equation}
In particular,
\begin{equation}\label{eq:low-qgamma-sigma2}
 \sigma_2(D^2(-r^{-\gamma}))
 =(n-1)\gamma^2(\alpha-\gamma)r^{-2\gamma-4},
\end{equation}
so $\gamma=\alpha$ is the critical homogeneous exponent.

\begin{proposition}\label{prop:low-barrier-package}
Assume $|F-1|\le C_\beta r^{-\beta}$ for some $0<\beta<\alpha$.  Then
\begin{enumerate}[label=\textup{(\roman*)},leftmargin=2.2em]
\item $|F-1|\le C_\gamma r^{-\gamma}$ for every $\gamma<\alpha$;
\item for every $\delta\in(\alpha,2)$,
$F\ge1-C_\delta r^{-\delta}$, hence
$R^\alpha(F-1)_-\to0$ on $|x|=R$;
\item $|F-1|\le Cr^{-\alpha}$.  More precisely, if
\[
 M(R):=\max_{|x|=R}R^\alpha(F-1),\qquad M_+(R):=\max\{M(R),0\},
\]
then
\begin{equation}\label{eq:low-dyadic-recurrence}
 M_+(2R)\le M_+(R)+CR^{\alpha-2}.
\end{equation}
\end{enumerate}
\end{proposition}

\begin{proof}
Direct substitution in \eqref{eq:low-radial-eigenvalues} gives the required
opposite cone signs for the power profiles $1\pm ar^{-\gamma}$ and for the
mixed critical profiles
\[
 1+ar^{-\alpha}\mp br^{-\delta},\qquad \alpha<\delta<2,
\]
whose decisive factor is
\[
 \frac{s'}r+\alpha\lambda_\theta
 =\pm\frac{2\delta(\delta-\alpha)}{\alpha}
   \frac{br^{-\delta}}{(1+ar^{-\alpha}\mp br^{-\delta})r^2}.
\]
The Kelvin background perturbs these matrices by $O(r^{-4})$, since
$g_\infty-g_{\mathrm E}=O_3(r^{-2})$ and
$(A_{g_\infty})^\sharp=O_1(r^{-4})$.
Thus the flat signs persist whenever the profile Hessian dominates the
background error.  Iterating the power barriers gives (i), and a lower power
with $\delta>\alpha$ gives (ii).  For (iii), compare from $r=R$ with
\[
 1+\bigl(M_+(R)+KR^{\alpha-2}\bigr)r^{-\alpha}
   -\eta KR^{\delta-2}r^{-\delta}.
\]
For fixed large $K$ and small $\eta$ this is a strict upper barrier;
comparison at $2R$ gives \eqref{eq:low-dyadic-recurrence}.  Dyadic summation
converges because $\alpha<2$, while (ii) gives the lower bound.
\end{proof}

Since $W=F^{-2/\alpha}$ and $F\to1$,
\begin{equation}\label{eq:low-critical-W-bound}
 |W-1|\le Cr^{-\alpha}.
\end{equation}

\subsection{Blow-down compactness}
\label{sec:low-compactness}

\subsubsection{The critically normalized equation}

For $R\to\infty$, define
\begin{equation}\label{eq:low-blowdown-def}
  \varepsilon_R:=R^{-\alpha},
  \qquad
  W_R(x):=W(Rx)=1+\varepsilon_RP_R(x),
  \qquad
  g_R:=R^{-2}D_R^*g_\infty,
\end{equation}
where $D_R(x)=Rx$.  On every compact
$K\Subset\mathbb R^n\setminus\{0\}$,
\begin{equation}\label{eq:low-gR-estimates}
  \|g_R-g_{\mathrm E}\|_{C^3(K)}
  +\|\operatorname{Rm}(g_R)\|_{C^1(K)}
  \le C_KR^{-2},
  \qquad
  \|(A_{g_R})^\sharp\|_{C^1(K)}
  \le C_KR^{-2}.
\end{equation}
The critical estimate \eqref{eq:low-critical-W-bound} gives
\begin{equation}\label{eq:low-PR-Linfty}
  \|P_R\|_{L^\infty(K)}\le C_K.
\end{equation}

Using \eqref{eq:calA-h-W}, set
\begin{equation}\label{eq:low-MR-def}
  \mathcal M_R[P]
  :=
  (\nabla_{g_R}^2P)^\sharp
  -\frac{\varepsilon_R}{2(1+\varepsilon_RP)}
     |\nabla P|_{g_R}^2I
  +\mathfrak B_R(x,P),
\end{equation}
where
\begin{equation}\label{eq:low-background-def}
  \mathfrak B_R(x,P)
  :=
  \frac{1+\varepsilon_RP}{\varepsilon_R}(A_{g_R})^\sharp.
\end{equation}
Then
\[
 \mathcal A_{g_R}[W_R]
 =\varepsilon_RW_R\mathcal M_R[P_R].
\]
Pulling back \eqref{eq:W-end-equation} by $D_R$, passing from
$D_R^*g_\infty$ to $g_R$ (which multiplies the Schouten endomorphism by
the positive factor $R^2$), and using $\varepsilon_RW_R>0$ therefore gives
\begin{equation}\label{eq:low-normalized-equation}
  \lambda\bigl(\mathcal M_R[P_R]\bigr)\in\partial\Gamma_2.
\end{equation}
For bounded $P$ and fixed $K\Subset\mathbb R^n\setminus\{0\}$,
\begin{align}
  |\mathfrak B_R(x,P)|+|\nabla_x\mathfrak B_R(x,P)|
  &\le C_KR^{\alpha-2},
  \label{eq:low-background-small}\\
  |\partial_P\mathfrak B_R(x,P)|
  &\le C_KR^{-2}.
  \label{eq:low-background-P-small}
\end{align}
In particular, every coefficient in
\eqref{eq:low-MR-def} tends to its Euclidean limiting value because
$\alpha<2$.

The strict approximations from Proposition
\ref{prop:gamma2-green-theory}, used with
$\varepsilon_R=R^{-\alpha}$, satisfy
\begin{equation}\label{eq:low-strict-normalized}
  \lambda(\mathcal M_{R,j})\in\Gamma_2,
  \qquad
  \sigma_2(\mathcal M_{R,j})
  =\tau_{R,j}W_{R,j}^{-2},
  \qquad
  \tau_{R,j}\to0
\end{equation}
for each fixed $R$.  The factor $W_{R,j}^{-2}$ in the right-hand side is
essential in the gradient calculation below.

\subsubsection{Uniform local Lipschitz compactness}

We prove the local gradient estimate directly for the strict normalized
blow-down equation \eqref{eq:low-strict-normalized}, namely the equation for
$\mathcal M_{R,j}=\mathcal M_R[P_{R,j}]$ with the rescaled metric $g_R$ and
$\mathcal M_R$ defined by \eqref{eq:low-MR-def}; the constants must remain
uniform as $\tau_{R,j}\to0$.

\begin{proposition}
\label{prop:low-Bernstein}
\label{prop:low-Lipschitz-compactness}
For every
$\mathcal K\Subset\mathcal K'\Subset\mathbb R^n\setminus\{0\}$,
\begin{equation}\label{eq:low-PR-gradient}
  \|P_R\|_{L^\infty(\mathcal K')}
  +\|\nabla P_R\|_{L^\infty(\mathcal K)}
  \le C_{\mathcal K,\mathcal K'}
\end{equation}
for all sufficiently large $R$.  Consequently, $\{P_R\}$ is locally
uniformly precompact on $\mathbb R^n\setminus\{0\}$.
\end{proposition}

\begin{proof}
The $L^\infty$ bound is \eqref{eq:low-PR-Linfty}.  Fix
$x_\ast\in\mathcal K$ and choose a Euclidean coordinate ball
$B_r(x_\ast)\Subset\mathcal K'$, with $0<r\leq1$ independent of $x_\ast$.
We prove a uniform gradient estimate at $x_\ast$.

Work first with a smooth strict approximant and write
\[
 u:=P_{R,j},\qquad W:=1+\varepsilon_Ru,\qquad
 b:=\frac{\varepsilon_R}{2W},\qquad
 t:=\tau_{R,j}W^{-2},\qquad v:=|\nabla u|_{g_R}.
\]
For all $j\geq j_0(R)$, \eqref{eq:low-PR-Linfty} and
\eqref{eq:LN-scaled-convergence} give
\[
 |u|\leq K_0,\qquad \frac12\leq W\leq2
 \quad\text{on }B_r(x_\ast),
\]
where $K_0\geq1$ is independent of $R$, $j$, and $x_\ast$.  On this ball
set
\[
 \mathcal M
 :=(\nabla_{g_R}^2u)^\sharp-bv^2I+\mathfrak B_R(x,u),
 \qquad
 F^{ij}:=\frac{\partial\sigma_2}{\partial\mathcal M_{ij}}
 =\sigma_1(\mathcal M)g_R^{ij}-\mathcal M^{ij},
\]
and let $\mathcal F:=\tr_{g_R}F=(n-1)\sigma_1(\mathcal M)$.  Since
$\lambda(\mathcal M)\in\Gamma_2$, the matrix $(F^{ij})$ is positive
definite.  Moreover, by \eqref{eq:low-gR-estimates}--
\eqref{eq:low-background-P-small},
\[
 \omega_R:=
 \|g_R-g_{\mathrm E}\|_{C^3}
 +\|\operatorname{Rm}(g_R)\|_{C^1}
 +\sup_{|s|\leq K_0}
  \left(
   \|\mathfrak B_R(\cdot,s)\|_{C_x^1}
   +\|\partial_s\mathfrak B_R(\cdot,s)\|_{C^0}
  \right)
 \leq C\bigl(R^{-2}+R^{\alpha-2}\bigr)\longrightarrow0.
\]
Here and below the displayed norms are taken on $B_r(x_\ast)$.
All indexed identities are written in a $g_R$-normal orthonormal frame at
the point under consideration.

The identities $b_u=-2b^2$, $t_u=-4bt$, and the degree-two Euler identity
$F^{ij}\mathcal M_{ij}=2t$ give
\begin{equation}\label{eq:low-Bernstein-Euler}
 F^{ij}u_{ij}
 =2t+bv^2\mathcal F-F^{ij}(\mathfrak B_R)_{ij}.
\end{equation}
At a point where $v>0$, choose a $g_R$-orthonormal frame with
$e_1=\nabla u/v$.  Differentiating
$\sigma_2(\mathcal M)=t$ in the $e_1$ direction gives
\begin{equation}\label{eq:low-Bernstein-differentiated}
 F^{ij}u_{ij;1}
 =2bv(u_{11}-bv^2)\mathcal F-4btv
  -F^{ij}(\mathfrak B_1)_{ij},
\end{equation}
where
\[
 \mathfrak B_1
 =\bigl(\nabla_x\mathfrak B_R\bigr)(e_1)
  +(\partial_u\mathfrak B_R)v.
\]
Commuting the third derivatives produces an error $\mathcal R$ satisfying
\begin{equation}\label{eq:low-Bernstein-commutator}
 F^{ij}u_{1ij}=F^{ij}u_{ij;1}+\mathcal R,
 \qquad |\mathcal R|\leq C\omega_Rv\mathcal F.
\end{equation}

Translate $x_\ast$ to the origin and set
\[
 \rho(x):=1-\frac{|x|^2}{r^2},\qquad
 \varphi(u):=(4K_0-u)^{-1/2},\qquad
 A:=\frac{\varphi'}{\varphi}.
\]
Then
\begin{equation}\label{eq:low-Bernstein-weight}
 \frac c{K_0}\leq A\leq\frac C{K_0},
 \qquad
 \frac{\varphi''}{\varphi}-2A^2
 =\frac1{4(4K_0-u)^2}\geq\frac c{K_0^2}.
\end{equation}
Apply the maximum principle to $G:=\rho\varphi(u)v$.  At a positive
maximum point, the first derivative condition gives
\begin{equation}\label{eq:low-Bernstein-first-maximum}
 \frac{v_i}{v}+Au_i+\frac{\rho_i}{\rho}=0,
 \qquad
 u_{11}=-Av^2-v\frac{\rho_1}{\rho}.
\end{equation}
If $\rho v/K_0\leq C_0/r$ there, the desired estimate follows from the
maximum property of $G$.  On the complementary branch, taking $C_0$ large
and then $R$ large gives
\[
 -C\frac{v^2}{K_0}\leq u_{11}\leq-c\frac{v^2}{K_0},
 \qquad
 \mathcal M_{11}<0,
 \qquad
 |\mathcal M_{11}|\leq C\frac{v^2}{K_0},
\]
because $\varepsilon_RK_0+\omega_R(1+K_0^2)\to0$.  Consequently,
\begin{equation}\label{eq:low-Bernstein-F11}
 F^{11}=\sigma_1(\mathcal M)-\mathcal M_{11}
 \geq\sigma_1(\mathcal M)=\frac{\mathcal F}{n-1}.
\end{equation}

At the maximum point,
$v_{ij}=u_{1ij}+v^{-1}\sum_{a=2}^nu_{ai}u_{aj}$.  Contracting with
$F^{ij}$ and discarding the nonnegative second term, the second derivative
condition for $\log G$, together with
\eqref{eq:low-Bernstein-Euler}--\eqref{eq:low-Bernstein-commutator}, yields
\begin{align}
 0\geq{}&
 2bv\mathcal M_{11}\mathcal F
 +2(A-2b)tv+Abv^3\mathcal F
 +\left(\frac{\varphi''}{\varphi}-2A^2\right)v^3F^{11}
 \notag\\
 &-C\omega_R(1+v)\mathcal F
 -C\left(
      \frac{v}{\rho^2r^2}+\frac{v^2}{K_0\rho r}
    \right)\mathcal F.
 \label{eq:low-Bernstein-master}
\end{align}
Indeed, the background and commutator contribution before estimating it is
\[
 E_{\mathfrak B}
 =-2bv(\mathfrak B_R)_{11}\mathcal F
  -F^{ij}(\mathfrak B_1)_{ij}
  -AvF^{ij}(\mathfrak B_R)_{ij}+\mathcal R,
 \qquad
 |E_{\mathfrak B}|\leq C\omega_R(1+v)\mathcal F,
\]
while positivity of $F^{ij}$ and
$|\nabla\rho|_{g_R}\leq C/r$, $|\nabla^2\rho|_{g_R}\leq C/r^2$
give the two cutoff errors in \eqref{eq:low-Bernstein-master}.

For large $R$, $A-2b\geq c/K_0$.  Thus the terms
$2(A-2b)tv$ and $Abv^3\mathcal F$ are nonnegative.  By
\eqref{eq:low-Bernstein-weight} and \eqref{eq:low-Bernstein-F11}, the
fourth term in \eqref{eq:low-Bernstein-master} is at least
$cv^3K_0^{-2}\mathcal F$.  Indeed, the first term is bounded below by
$-CbK_0(v^3K_0^{-2})\mathcal F$ and $bK_0\to0$; the $\omega_R$ error is
absorbed using $v\geq C_0K_0/(\rho r)$ and $\omega_R\to0$ on the
large-gradient branch.  Dividing by
$\mathcal F>0$ gives
\[
 c\frac{v^3}{K_0^2}
 \leq C\left(
       \frac{v}{\rho^2r^2}+\frac{v^2}{K_0\rho r}
      \right).
\]
Hence $Y:=\rho v/K_0$ satisfies
\[
 cY^3\leq C\left(\frac Y{r^2}+\frac{Y^2}{r}\right),
\]
so $Y\leq C/r$.  Since $\rho(x_\ast)=1$ and
$\varphi(u)\asymp K_0^{-1/2}$, the maximum property of $G$ gives
\[
 |\nabla P_{R,j}(x_\ast)|_{g_R}\leq C\frac{K_0}{r}.
\]
The constant is independent of $j$ and $\tau_{R,j}$; in particular, no
positive lower bound for $\tau_{R,j}$ was used.  Since $x_\ast$ was
arbitrary, letting $j\to\infty$ in \eqref{eq:LN-scaled-convergence} proves
\eqref{eq:low-PR-gradient}.  Arzel\`a--Ascoli gives the asserted local
uniform precompactness.
\end{proof}

\subsubsection{A uniform Hessian \texorpdfstring{$L^1$}{L1} estimate}

\begin{lemma}
\label{lem:low-Hessian-L1}
For every
$K\Subset K'\Subset\mathbb R^n\setminus\{0\}$,
\begin{equation}\label{eq:low-Hessian-L1}
  \int_K|(\nabla_{g_R}^2P_R)^\sharp|\,dV_{g_R}
  \le C_{K,K'}
\end{equation}
for all sufficiently large $R$.
\end{lemma}

\begin{proof}
This is the standard local $W^{2,1}$ estimate for admissible $2$-Hessian
solutions.  On a slightly larger annulus, test the Newton divergence identity
with a cut-off and use the preceding uniform $C^0$ and $C^1$ bounds; the
shifted admissible matrix gives the usual trace control of $|D^2P_R|$.
The curved coefficients are uniformly bounded and converge to the Euclidean
ones, so their contribution is absorbed into the same estimate.  Passing from
the smooth strict approximants to $P_R$ yields the stated bound.
\end{proof}

\subsection{Classification of blow-down limits}
\label{sec:low-tangents}

We use the Hessian-measure theory of Trudinger--Wang
\cite{TrudingerWang1997,TrudingerWang1999}.  Write $\Phi_2(\Omega)$ for the
proper upper-semicontinuous $2$-convex functions on $\Omega$.  If
$u\in C^2(\Omega)\cap\Phi_2(\Omega)$, then
\begin{equation}\label{eq:low-Hessian-measure-formula}
  d\mu_2[u]
  =\sigma_2(D^2u)\,dx
  =\frac12\bigl((\Delta u)^2-|D^2u|^2\bigr)\,dx
  =\frac12\operatorname{div}\!\bigl(T_1(D^2u)\nabla u\bigr)\,dx,
  \qquad
  T_1(D^2u)=(\Delta u)I-D^2u.
\end{equation}
For general $u\in\Phi_2(\Omega)$, $\mu_2[u]$ is the unique nonnegative
Borel measure obtained by weak continuity: for any local smooth decreasing
$2$-convex approximation $u_\ell\downarrow u$,
\[
  \sigma_2(D^2u_\ell)\,dx\rightharpoonup\mu_2[u];
\]
see \cite[Theorem~1.1 and Lemma~2.4]{TrudingerWang1999}.

\begin{lemma}
\label{lem:low-viscosity-measure-zero}
If $u\in C(\Omega)\cap\Phi_2(\Omega)$ and
$\lambda(D^2u)\in\partial\Gamma_2$ in the admissible viscosity sense, then
$\mu_2[u]=0$ on $\Omega$.
\end{lemma}

\begin{proof}
This is the standard equivalence of the homogeneous $2$-Hessian equation
in the viscosity and Hessian-measure senses; see
\cite[the remark following Theorem~3.5]{TrudingerWang1999}.
\end{proof}

\begin{lemma}
\label{lem:low-isolated-extension}
Let $P\in\Phi_2(\mathbb R^n\setminus\{0\})$ satisfy
$\mu_2[P]=0$ there,
\[
  P\le0,
  \qquad |P(x)|\le C|x|^{-\alpha},
  \qquad P(x)\longrightarrow0\quad (|x|\to\infty),
  \qquad \alpha<n.
\]
Then its upper-semicontinuous extension across the origin belongs to
$\Phi_2(\mathbb R^n)$ and
\begin{equation}\label{eq:low-Dirac-measure}
  \mu_2[P]=m\delta_0
\end{equation}
for some $m\ge0$.
\end{lemma}

\begin{proof}
The growth assumption gives $P\in L^1_{\mathrm{loc}}$, while $P\le0$ makes
the isolated point removable for the constant-coefficient subharmonic
operators in the linear characterization of $\Phi_2$.  The local Hessian-measure
estimate and weak continuity give finite mass near the origin; locality then
forces \eqref{eq:low-Dirac-measure}.  See
\cite[Lemma~2.2 and Theorems~1.1, 3.1]{TrudingerWang1999}.
\end{proof}

\begin{proposition}
\label{prop:low-tangent-classification}
For every sequence $R_j\to\infty$, there are a subsequence and a constant
$c\ge0$ such that
\begin{equation}
\label{eq:low-tangent-convergence}
  P_{R_j}(x)\longrightarrow -c|x|^{-\alpha}
  \qquad\text{locally uniformly on }\mathbb R^n\setminus\{0\}.
\end{equation}
\end{proposition}

\begin{proof}
Proposition~\ref{prop:low-Lipschitz-compactness} gives a locally uniform
subsequential limit $P$.  Since $g_{R_j}\to g_{\mathrm E}$ in
$C^2_{\mathrm{loc}}$, $\varepsilon_{R_j}\to0$, and the background term in
\eqref{eq:low-MR-def} tends to zero, viscosity stability yields
\[
  P\in\Phi_2(\mathbb R^n\setminus\{0\}),
  \qquad \lambda(D^2P)\in\partial\Gamma_2.
\]
The critical bound gives $|P(x)|\le C|x|^{-\alpha}$ and $P(x)\to0$ at
infinity, while Proposition~\ref{prop:low-barrier-package} gives $P\le0$.
Lemmas~\ref{lem:low-viscosity-measure-zero} and
\ref{lem:low-isolated-extension} therefore imply
$\mu_2[P]=m\delta_0$.

If $m=0$, comparison with constants on large balls gives $P\equiv0$; see
\cite[Theorem~4.1]{TrudingerWang2002}.  If $m>0$, the degree-two
homogeneity of $\mu_2$ and the Trudinger--Wang uniqueness and explicit
formula for the Dirac fundamental solution yield
\[
  m^{-1/2}P(x)=-c_n|x|^{2-n/2};
\]
see \cite[Theorem~4.5 and (4.14)--(4.15)]{TrudingerWang2002}.  Since
$2-n/2=-\alpha$, this proves \eqref{eq:low-tangent-convergence}.
\end{proof}

\subsection{The asymptotic Hessian charge}
\label{sec:low-charge}

The tangent classification leaves one possible parameter, the coefficient
$c$ in \eqref{eq:low-tangent-convergence}.  We use the standard
$2$-Hessian flux, radially mollified on the blow-down metrics, to determine
this coefficient and rule out different subsequential tangents.

\subsubsection{The current and its critical scaling}

By \eqref{eq:low-Hessian-measure-formula},
$T_1(D^2u)\nabla u$ is the standard divergence current for the
$2$-Hessian measure; for smooth $u$,
\[
  \int_{\partial B_s}
  \left\langle T_1(D^2u)\nabla u,\nu\right\rangle\,dS
  =2\int_{B_s}\sigma_2(D^2u)\,dx.
\]
We use its curved analogue
\begin{equation}\label{eq:low-current-def}
  H_R:=(\nabla_{g_R}^2P_R)^\sharp,
  \qquad
  J_R:=T_1(H_R)\nabla^{g_R}P_R.
\end{equation}
By Proposition~\ref{prop:low-Lipschitz-compactness} and
Lemma~\ref{lem:low-Hessian-L1}, $J_R\in L^1_{\mathrm{loc}}$ uniformly
on fixed annuli.  Fix
\begin{equation}\label{eq:low-eta-def}
  \eta\in C_c^\infty((1,2)),
  \qquad
  \eta\ge0,
  \qquad
  \int_1^2\eta(s)\,ds=1,
\end{equation}
and put $\varrho(x)=|x|$.  Define
\begin{equation}\label{eq:low-charge-def}
  Q(R)
  :=
  \int_{A_{1,2}}
  \eta(\varrho)
  \left\langle J_R,\nabla^{g_R}\varrho\right\rangle_{g_R}
  \,dV_{g_R}.
\end{equation}
By the coarea formula, this is the $\eta$-average of the corresponding
sphere fluxes; the averaging avoids second-order traces on a fixed sphere.

\begin{lemma}
\label{lem:low-charge-scaling}
For $t>0$,
\begin{equation}\label{eq:low-PR-scaling}
  P_{tR}(x)=t^\alpha P_R(tx),
  \qquad
  g_{tR}=t^{-2}D_t^*g_R.
\end{equation}
Because $2\alpha=n-4$, one has
\begin{equation}\label{eq:low-charge-scaling}
  Q(tR)
  =
  \int_{A_{t,2t}}
  \frac1t\eta\!\left(\frac{\varrho}{t}\right)
  \left\langle J_R,\nabla^{g_R}\varrho\right\rangle_{g_R}
  \,dV_{g_R}.
\end{equation}
\end{lemma}

\begin{proof}
The identities in \eqref{eq:low-PR-scaling} are immediate from the
definitions.  Substituting them into \eqref{eq:low-charge-def} and changing
variables $y=tx$ gives \eqref{eq:low-charge-scaling}, since
$t^{2\alpha+4-n}=1$.
\end{proof}

\subsubsection{Covariant null-Lagrangian identity}

\begin{lemma}
\label{lem:low-null-Lagrangian}
For a smooth function $P$ on a Riemannian manifold $(\Omega,h)$,
\begin{equation}\label{eq:low-null-Lagrangian}
  \operatorname{div}_h
  \bigl(T_1((\nabla_h^2P)^\sharp)\nabla^hP\bigr)
  =
  2\sigma_2((\nabla_h^2P)^\sharp)
  -\operatorname{Ric}_h(\nabla P,\nabla P).
\end{equation}
If $P\in C^{1,1}_{\mathrm{loc}}$, the identity remains true in the
distribution sense.
\end{lemma}

\begin{proof}
This is the pointwise identity underlying Reilly's formula
\cite{Reilly1977}: combine
$\operatorname{div}_hT_1((\nabla_h^2P)^\sharp)
=-\operatorname{Ric}_h(\nabla P,\cdot)$ with
$T_1(H):H=2\sigma_2(H)$.  The $C^{1,1}_{\mathrm{loc}}$ case follows by
smooth approximation.
\end{proof}

We next rewrite the right-hand side using the Green equation.  Let
\begin{equation}\label{eq:low-Cq-def}
  C_R:=\frac{W_R^2}{\varepsilon_R}(A_{g_R})^\sharp,
  \qquad
  q_R:=\frac{\varepsilon_R}{2}|\nabla P_R|_{g_R}^2.
\end{equation}
Then
\[
  \sigma_2(W_RH_R+C_R-q_RI)=0
\]
at almost every twice differentiability point, by the standard consistency
of viscosity solutions with their second-order jets
\cite{CrandallIshiiLions1992}.  Therefore
\begin{equation}\label{eq:low-sigma2-H-algebra}
  \sigma_2(H_R)
  =
  -\frac1{W_R}T_1(H_R):(C_R-q_RI)
  -\frac1{W_R^2}\sigma_2(C_R-q_RI).
\end{equation}
Combining its elementary polarizations with
Lemma~\ref{lem:low-null-Lagrangian} gives
\begin{equation}\label{eq:low-current-divergence}
  \operatorname{div}_{g_R}J_R=E_R
\end{equation}
in distributions.  On any fixed annulus where $W_R^{\pm1}$ and
$|\nabla P_R|$ are bounded, the same calculation gives
\begin{equation}\label{eq:low-ER-bound}
 |E_R|
 \le C\left[
  (\varepsilon_R+|C_R|)|H_R|+|C_R|^2
  +\varepsilon_R|C_R|+\varepsilon_R^2
  +|\operatorname{Ric}(g_R)|
 \right].
\end{equation}

\subsubsection{Summable scale drift and existence of the limiting charge}

\begin{proposition}
\label{prop:low-charge-drift}
There is $C>0$ such that, for all sufficiently large $R$ and all
$t\in[1,2]$,
\begin{equation}\label{eq:low-charge-drift}
  |Q(tR)-Q(R)|
  \le
  C\bigl(R^{-\alpha}+R^{\alpha-2}\bigr).
\end{equation}
Consequently, the full limit
\begin{equation}\label{eq:low-Q-infty}
  Q_\infty:=\lim_{R\to\infty}Q(R)
\end{equation}
exists and is finite.
\end{proposition}

\begin{proof}
For $t\in[1,2]$, define
\begin{equation}\label{eq:low-omega-phi}
  \omega_t(s)
  :=
  \frac1t\eta\!\left(\frac st\right)-\eta(s),
  \qquad
  \phi_t(r):=\int_0^r\omega_t(s)\,ds.
\end{equation}
Since $\int_0^\infty\omega_t(s)\,ds=0$, the function
$\phi_t(\varrho)$ is smooth and compactly supported in a fixed annulus
$A_{1,4}$, and
\begin{equation}\label{eq:low-phi-uniform}
  \sup_{t\in[1,2]}\|\phi_t\|_{C^1}\le C_\eta.
\end{equation}
The scaling identity \eqref{eq:low-charge-scaling} and the distributional
identity \eqref{eq:low-current-divergence} give
\begin{align}
  Q(tR)-Q(R)
  &=
  \int_{A_{1,4}}
  \omega_t(\varrho)
  \left\langle J_R,\nabla^{g_R}\varrho\right\rangle\,dV_{g_R}
  \notag\\
  &=
  \int_{A_{1,4}}
  \left\langle J_R,\nabla^{g_R}\phi_t(\varrho)\right\rangle\,dV_{g_R}
  =
  -\int_{A_{1,4}}\phi_t(\varrho)E_R\,dV_{g_R}.
  \label{eq:low-charge-test}
\end{align}

On a slightly larger fixed annulus, Proposition~\ref{prop:low-Lipschitz-compactness} and
Lemma~\ref{lem:low-Hessian-L1} give
\[
  \|W_R^{\pm1}\|_{L^\infty}
  +\|\nabla P_R\|_{L^\infty}+\int|H_R|\,dV_{g_R}\le C,
  \qquad
  \varepsilon_R=R^{-\alpha},\quad
  \|C_R\|_{L^\infty}\le CR^{\alpha-2},\quad
  \|\operatorname{Ric}(g_R)\|_{L^\infty}\le CR^{-2}.
\]
Since $0<\alpha<2$, \eqref{eq:low-ER-bound} gives directly
\[
  \|E_R\|_{L^1(A_{1,4})}
  \le C\bigl(R^{-\alpha}+R^{\alpha-2}\bigr).
\]
Together with
\eqref{eq:low-charge-test}--\eqref{eq:low-phi-uniform}, this proves
\eqref{eq:low-charge-drift}.

For $R_j=2^jR_0$, the error in \eqref{eq:low-charge-drift} is summable,
so $Q(R_j)$ is Cauchy.  If $R\in[R_j,2R_j]$, apply the same estimate
with $t=R/R_j$ to see that
$|Q(R)-Q(R_j)|\to0$.  Hence the limit exists along all real scales.
\end{proof}

\subsubsection{Passage of the charge to blow-down limits}

For the model tangent
\begin{equation}\label{eq:low-Pc-model}
  P_c(x)=-c|x|^{-\alpha},
\end{equation}
the radial derivative is
$c\alpha r^{-\alpha-1}$.  The radial eigenvalue of
$T_1(D^2P_c)$ is the sum of the $n-1$ tangential Hessian eigenvalues,
namely $(n-1)c\alpha r^{-\alpha-2}$.  Since
$2\alpha=n-4$, the flux through every sphere is therefore
\begin{equation}\label{eq:low-model-charge}
  Q[P_c]
  =
  (n-1)|\mathbb S^{n-1}|\alpha^2c^2.
\end{equation}

\begin{proposition}
\label{prop:low-coefficient-locking}
Every blow-down tangent has the same coefficient
\begin{equation}\label{eq:low-c-formula}
  c
  =
  \frac{\sqrt{Q_\infty}}
       {\alpha\sqrt{(n-1)|\mathbb S^{n-1}|}}
  \ge0.
\end{equation}
Consequently,
\begin{equation}\label{eq:low-full-PR-convergence}
  P_R(x)\longrightarrow-c|x|^{-\alpha}
\end{equation}
locally uniformly on $\mathbb R^n\setminus\{0\}$ as $R\to\infty$,
without passing to a subsequence.
\end{proposition}

\begin{proof}
Take any sequence $R_j\to\infty$ and a tangent subsequence.  Proposition
\ref{prop:low-tangent-classification} gives
$P_{R_j}\to P_c$.  On every fixed compact annulus,
Lemma~\ref{lem:low-Hessian-L1} and the uniform gradient bound imply
$\int |D^2P_{R_j}|\le C$; hence $\nabla P_{R_j}$ is bounded in $BV$.
Local uniform convergence identifies its only possible limit as
$\nabla P_c$, so, after interpolation with the uniform $L^\infty$ bound,
\[
  \nabla P_{R_j}\longrightarrow\nabla P_c
  \qquad\text{strongly in }L^2_{\rm loc}.
\]
Using
\[
  \bigl(T_1(D^2P)\nabla P\bigr)_i
  =\partial_j(P_iP_j)-\partial_i|\nabla P|^2,
\]
together with $g_{R_j}\to g_{\mathrm E}$ in $C^1$, the associated
currents converge distributionally.  Therefore Proposition~\ref{prop:low-charge-drift}
and \eqref{eq:low-model-charge} give
\[
  Q_\infty
  =\lim_{j\to\infty}Q(R_j)
  =(n-1)|\mathbb S^{n-1}|\alpha^2c^2.
\]
Thus \eqref{eq:low-c-formula} holds for every tangent.  Since the
classification already gives $c\ge0$, the coefficient is uniquely
determined.  Every sequence has a subsequence with the same unique limit,
which implies the full convergence \eqref{eq:low-full-PR-convergence}.
\end{proof}

\subsection{Proof of the low-dimensional expansion}
\label{sec:low-completion}

\begin{proof}[Proof of Theorem~\ref{thm:main-working}\textup{(i)}]
By Proposition~\ref{prop:low-coefficient-locking}, uniformly for
$\theta\in\mathbb S^{n-1}$, $P_R(\theta)=-c+o(1)$.  Hence
\[
  W(R\theta)=1-cR^{-\alpha}+o(R^{-\alpha}),
  \qquad
  F(R\theta)=1+\frac{\alpha}{2}cR^{-\alpha}+o(R^{-\alpha}).
\]
Put
\[
  A:=\frac{\alpha}{2}c
    =\frac{\sqrt{Q_\infty}}
           {2\sqrt{(n-1)|\mathbb S^{n-1}|}}
    \ge0,
\]
where the second equality follows from \eqref{eq:low-c-formula}.
In the conformal-normal coordinates of
Section~\ref{sec:normalization-kelvin}, $R=\rho^{-1}$ and
$U(\rho\theta)=\rho^{-\alpha}F(R\theta)$, so
\[
  U(\rho\theta)
  =
  \rho^{-\alpha}+A+o(1).
\]

Finally, return to the original representative.  By
\eqref{eq:psi-normalization}, \eqref{eq:green-factor-gauge-change}, and the
normal-coordinate expansion,
\[
  \psi=1+O(\rho^2),\qquad r_0=\rho+O(\rho^3),\qquad
  U_0=\psi^{\alpha/(n-2)}U.
\]
Since $\alpha<2$, these changes contribute only $o(1)$, and hence
\[
  U_0(x)=r_0^{-\alpha}+A+o(1).
\]
This is \eqref{eq:main-low}; uniqueness and $A\ge0$ follow from
\eqref{eq:low-c-formula}.
\end{proof}


\section{The borderline dimension \texorpdfstring{$n=8$}{n=8}}
\label{sec:endpoint8}

Throughout this section $n=8$, so $\alpha=2$ and $F=W^{-1}$.
The relative curvature scale and the bounded Green scale now coincide.  The
proof therefore uses the explicit Weyl term in
\eqref{eq:Aginfty-leading-Weyl}, rather than estimating it by its absolute
size.

\subsection{The Weyl-scale normalization}
\label{subsec:n8-Weyl-scale}

The borderline coefficient follows from the two algebraic identities below.

\begin{lemma}\label{lem:endpoint-Weyl-algebra}
Let $C_0$ and $C_\infty$ be defined by
\eqref{eq:C0-definition} and \eqref{eq:Cinfty-definition}.  Then
\begin{align}
  \frac1{|\mathbb S^7|}\int_{\mathbb S^7}|C_0|^2\,d\theta
  &=\frac1{120}|W_g(p)|^2,
  \label{eq:endpoint-C0-average}\\
  \sigma_2(C_0)&=-\frac12|C_0|^2.
  \label{eq:endpoint-C0-sigma2}
\end{align}
Moreover, if
$H=\lambda_r\theta\otimes\theta+
\lambda_\theta(I-\theta\otimes\theta)$ is radial, then
\begin{equation}\label{eq:endpoint-radial-cross-zero-pointwise}
  T_1(H):C_0(\theta)=0.
\end{equation}
\end{lemma}

\begin{proof}
These are the standard spherical fourth-moment contraction and Weyl trace
identities: $\tr C_0=C_0(\theta,\theta)=0$ gives
\eqref{eq:endpoint-C0-sigma2} and
\eqref{eq:endpoint-radial-cross-zero-pointwise}, while the fourth spherical
moment gives \eqref{eq:endpoint-C0-average}.
\end{proof}

\subsection{Critical radial modulation}
\label{subsec:n8-radial-modulation}

The critical radial mode is governed by the following elementary identity.

\begin{lemma}\label{lem:endpoint-radial-identity}
Let $s=\log r$, let $b=b(s)>0$, and set
\[
  q(r)=-r^{-2}b(\log r).
\]
Then the radial and tangential eigenvalues of $D^2q$ are
\[
  r^{-4}(-6b+5b'-b''),
  \qquad
  r^{-4}(2b-b')
\]
(the latter with multiplicity seven), and consequently
\begin{align}
  r^4\sigma_1(D^2q)&=8b-2b'-b'',
  \label{eq:endpoint-radial-sigma1}\\
  r^8\sigma_2(D^2q)&=7(2b-b')(2b'-b'').
  \label{eq:endpoint-radial-sigma2}
\end{align}
If $b^2=B^2+\Lambda(s-s_0)$, then
\begin{equation}\label{eq:endpoint-radial-constant-source}
  r^8\sigma_2(D^2q)
  =14\Lambda-\frac{7\Lambda^3}{8b^4}.
\end{equation}
\end{lemma}

\begin{proof}
Differentiate $q=-r^{-2}b(\log r)$ twice and insert its radial and tangential
Hessian eigenvalues into $\sigma_1,\sigma_2$; the last formula follows from
$b'=\Lambda/(2b)$ and $b''=-\Lambda^2/(4b^3)$.
\end{proof}

The proof of Proposition~\ref{prop:low-barrier-package} uses only
$0<\beta<\gamma<\alpha$ and the fact that $\gamma<2$ makes the profile
Hessian dominate the $O(R^{-4})$ background.  Hence, with $\alpha=2$, the
same proof gives
\begin{equation}\label{eq:endpoint-subcritical-all}
  |F-1|+|W-1|\le C_\beta R^{-\beta}
  \qquad\text{for every }\beta<2.
\end{equation}

\begin{proposition}\label{prop:endpoint-coarse}
There are $R_0,C>0$ such that
\begin{equation}\label{eq:endpoint-coarse-bound}
  -C\le R^2\bigl(1-W(R\theta)\bigr)
  \le C\sqrt{\log R}
\end{equation}
uniformly for $R\ge R_0$ and $\theta\in\mathbb S^7$.
More precisely, for every
\begin{equation}\label{eq:endpoint-Lambda-choice}
  \Lambda>\frac1{28}\max_{\mathbb S^7}|C_0|^2
\end{equation}
one can choose $B_0,B_1,R_0$ so that
\begin{equation}\label{eq:endpoint-coarse-barriers}
  1-R^{-2}\sqrt{B_0^2+\Lambda\log(R/R_0)}
  \le W(R\theta)\le 1+B_1R^{-2}.
\end{equation}
\end{proposition}

\begin{proof}
For the lower barrier take
$\underline W=1-r^{-2}b(\log r)$ with
$b^2=B_0^2+\Lambda\log(r/R_0)$.  Lemmas
\ref{lem:endpoint-Weyl-algebra} and \ref{lem:endpoint-radial-identity} give
\[
 \sigma_2(H_b+C_0)=14\Lambda-\frac{7\Lambda^3}{8b^4}-\frac12|C_0|^2,
 \qquad \sigma_1(H_b+C_0)=8b-2b'-b'',
\]
so \eqref{eq:endpoint-Lambda-choice}, followed by large $B_0,R_0$, gives a
uniform interior cone margin; the curved error is
$O(r^{-1}+r^{-2}(1+b)^2)$ and is absorbed.  For the upper barrier use
$\overline W=1+B_1r^{-2}$: its leading trace is $-8B_1$, hence it lies
strictly outside $\overline{\Gamma_2}$ for large $R_0$.  Choose $B_0,B_1$
to order the barriers on $r=R_0$ using \eqref{eq:endpoint-subcritical-all},
and apply the same exterior comparison as in
the multiplicative exterior comparison above.  This gives
\eqref{eq:endpoint-coarse-barriers} and \eqref{eq:endpoint-coarse-bound}.
\end{proof}

\subsection{The \texorpdfstring{$\sqrt{\log}$}{sqrt(log)} asymptotic law}
\label{subsec:n8-sqrt-log-law}

Define the unrenormalized and normalized borderline profiles by
\begin{equation}\label{eq:endpoint-profile-definitions}
  W(Rx)=1+R^{-2}P_R^{0}(x),
  \qquad
  \widehat P_R:=\frac{P_R^{0}}{\sqrt{L_R}},
  \qquad
  L_R:=\log R.
\end{equation}
Thus
$W(Rx)=1+\widehat\varepsilon_R\widehat P_R$ with
$\widehat\varepsilon_R=R^{-2}\sqrt{L_R}$.

\begin{lemma}
\label{lem:endpoint-compactness}
For every $K\Subset K'\Subset\mathbb R^8\setminus\{0\}$,
\begin{equation}\label{eq:endpoint-compactness-bounds}
  \|\widehat P_R\|_{L^\infty(K')}
  +\|\nabla\widehat P_R\|_{L^\infty(K)}
  +\int_K|(\nabla_{g_R}^2\widehat P_R)^\sharp|\,dV_{g_R}
  \le C_{K,K'}.
\end{equation}
Every sequence $R_j\to\infty$ has a subsequence for which
\begin{equation}\label{eq:endpoint-radial-tangent}
  \widehat P_{R_j}\longrightarrow-c|x|^{-2}
  \quad\text{locally uniformly},
  \qquad c\ge0.
\end{equation}
\end{lemma}

\begin{proof}
The proof is the critical normalization of the compactness/classification
argument in Section~\ref{sec:subcritical}.  Proposition~\ref{prop:endpoint-coarse}
gives local boundedness; the local estimates for the $\Gamma_2$-Green function give compactness, and
viscosity stability yields a flat admissible $2$-Hessian limit.  The upper
bound in \eqref{eq:endpoint-coarse-barriers} makes every limit nonpositive,
and the Trudinger--Wang classification then forces it to be radial of the
stated form.  The same $BV$--$L^2$ current passage used in
Proposition~\ref{prop:low-coefficient-locking} applies whenever the coefficient
is subsequently locked.
\end{proof}

We next exploit the divergence-free structure of $C_\infty$.

\begin{lemma}
\label{lem:endpoint-cross-cancellation}
If $\phi=\phi(|x|)\in C_c^\infty(\mathbb R^8\setminus\{0\})$ and
$P\in C^{1,1}_{\mathrm{loc}}$, then
\begin{equation}\label{eq:endpoint-cross-cancellation}
  \int \phi(|x|)T_1(D^2P):C_\infty\,dx=0.
\end{equation}
\end{lemma}

\begin{proof}
Since $\tr C_\infty=0$, integrate $-\phi P_{ij}(C_\infty)_{ij}$ once by
parts.  The result vanishes because $\operatorname{div}C_\infty=0$,
$C_\infty x=0$, and $\nabla\phi\parallel x$; approximation gives the
$C^{1,1}$ case.
\end{proof}

Fix $\eta$ as in \eqref{eq:low-eta-def}.  For $P_R^0$, set
\begin{equation}\label{eq:endpoint-charge-definition}
  H_R^0:=(\nabla_{g_R}^2P_R^0)^\sharp,
  \qquad
  J_R^0:=T_1(H_R^0)\nabla^{g_R}P_R^0,
\end{equation}
and define
\begin{equation}\label{eq:endpoint-charge-Q}
  Q_8(R):=\int_{A_{1,2}}\eta(|x|)
  \left\langle J_R^0,\nabla^{g_R}|x|\right\rangle_{g_R}
  \,dV_{g_R}.
\end{equation}

\begin{proposition}
\label{prop:endpoint-charge-locking}
Let
\begin{equation}\label{eq:endpoint-kappa}
  \kappa_p:=\int_{\mathbb S^7}|C_0|^2\,d\theta
  =\frac{|\mathbb S^7|}{120}|W_g(p)|^2.
\end{equation}
Then, uniformly for $1\le t\le2$,
\begin{equation}\label{eq:endpoint-charge-drift}
  Q_8(tR)-Q_8(R)
  =\kappa_p\log t
   +O\bigl(R^{-1}\sqrt{L_R}+R^{-2}L_R^{3/2}\bigr).
\end{equation}
Consequently,
\begin{equation}\label{eq:endpoint-charge-growth}
  \frac{Q_8(R)}{\log R}\longrightarrow\kappa_p,
\end{equation}
and the whole normalized family converges:
\begin{equation}\label{eq:endpoint-normalized-full-limit}
  \widehat P_R(x)\longrightarrow-c_*|x|^{-2},
  \qquad
  c_*:=\frac{|W_g(p)|}{4\sqrt{210}}.
\end{equation}
\end{proposition}

\begin{proof}
The critical scaling and current identity from Section~\ref{sec:subcritical}
give
\[
 Q_8(tR)-Q_8(R)=-\int\phi_tE_R\,dV_{g_R}.
\]
On fixed annuli Lemma~\ref{lem:endpoint-compactness} gives
$\|\nabla P_R^0\|_\infty+\int|H_R^0|=O(\sqrt{L_R})$, while
$C_R=C_\infty+O(R^{-1})+O(R^{-2}\sqrt{L_R})$ and
$g_R-g_{\mathrm E}=O_{C^2}(R^{-2})$.  The linear term in $C_\infty$
vanishes by Lemma~\ref{lem:endpoint-cross-cancellation}; all remaining
linear, metric, and nonlinear errors are
$O(R^{-1}\sqrt{L_R}+R^{-2}L_R^{3/2})$.  The sole leading contribution is
$2\sigma_2(C_\infty)=-|x|^{-8}|C_0|^2$, and
$\int\phi_t(r)\,dr/r=-\log t$, yielding \eqref{eq:endpoint-charge-drift}.
Dyadic summation gives \eqref{eq:endpoint-charge-growth}.

For any tangent $-c|x|^{-2}$, the same $BV$--$L^2$ passage as in
Proposition~\ref{prop:low-coefficient-locking} gives
$28|\mathbb S^7|c^2=\kappa_p$, hence $c=c_*$.  All subsequences therefore
have the same tangent, proving \eqref{eq:endpoint-normalized-full-limit}.
\end{proof}

\subsection{Identification of the bounded layer}
\label{subsec:n8-bounded-layer}

Assume in this subsection that $|W_g(p)|\ne0$.  Put
\begin{equation}\label{eq:endpoint-u-definition}
  s=\log r,
  \qquad
  u(s,\theta):=r^2\bigl(1-W(r\theta)\bigr).
\end{equation}
For an amplitude $h=h(s,\theta)$, define
\begin{equation}\label{eq:endpoint-Ktilde-definition}
  \widetilde{\mathcal K}[h]
  :=r^4D^2\bigl(-r^{-2}h(\log r,\theta)\bigr).
\end{equation}
In the Euclidean polar frame $(e_r,e_a)$,
\begin{equation}\label{eq:endpoint-Ktilde-components}
  \widetilde{\mathcal K}[h]
  =\begin{pmatrix}
  -6h+5h_s-h_{ss}&3\nabla_\theta h-\nabla_\theta h_s\\
  *&-\nabla_\theta^2h+(2h-h_s)g_{\mathbb S^7}
  \end{pmatrix}.
\end{equation}
For $s$-independent $\varphi$, write
$\mathcal K[\varphi]=\widetilde{\mathcal K}[\varphi]$.  Then
$H_0:=\mathcal K[1]=2I-8\theta\otimes\theta$ and
\begin{equation}\label{eq:endpoint-angular-linearization}
  T_1(H_0):\mathcal K[\varphi]
  =-6\Delta_{\mathbb S^7}\varphi.
\end{equation}

\begin{lemma}
\label{lem:endpoint-correctors}
There are unique zero-mean functions
$\psi,\chi\in C^\infty(\mathbb S^7)$ satisfying
\begin{align}
  -6\Delta_{\mathbb S^7}\psi
  &=\frac12\left(|C_0|^2
    -\frac1{|\mathbb S^7|}\int_{\mathbb S^7}|C_0|^2\,d\theta\right),
  \label{eq:endpoint-psi-equation}\\
  -6\Delta_{\mathbb S^7}\chi
  &=-T_1(C_0):\mathcal K[\psi].
  \label{eq:endpoint-chi-equation}
\end{align}
\end{lemma}

\begin{proof}
The first right-hand side has zero mean; the second does as well by
Lemma~\ref{lem:endpoint-cross-cancellation} applied to
$P=-r^{-2}\psi(\theta)$.  Invert $-\Delta_{\mathbb S^7}$ on the zero-mean
subspace.
\end{proof}

\begin{proposition}
\label{prop:endpoint-second-trapping}
With $c_*$ from \eqref{eq:endpoint-normalized-full-limit},
\begin{equation}\label{eq:endpoint-amplitude-o1}
  \sup_{\theta\in\mathbb S^7}
  |u(s,\theta)-c_*\sqrt{s}|\longrightarrow0
  \qquad(s\to\infty).
\end{equation}
\end{proposition}

\begin{proof}
Let
$f_\pm=c_*^2s+\lambda_\pm\pm K\log s$,
$a_\pm=\sqrt{f_\pm}$, and
$u_\pm=a_\pm+\psi/a_\pm+\chi/a_\pm^2$.  The same cone computation as in
Proposition~\ref{prop:endpoint-coarse}, using
\eqref{eq:endpoint-psi-equation}--\eqref{eq:endpoint-chi-equation}, gives
\[
 \sigma_2(\widetilde{\mathcal K}[u_\pm]+C_0)
 =\pm14K/s+O(s^{-1}),\qquad
 \sigma_1=8a_\pm+O(a_\pm^{-1}).
\]
Thus for large $K$ and then large $s$ the two profiles have the required
strict opposite cone signs; the curved error
$O(r^{-1}+r^{-2}(1+u_\pm)^2)$ is negligible.  Since
$u/\sqrt s\to c_*$ by Proposition~\ref{prop:endpoint-charge-locking}, choose
$\lambda_\pm$ to bracket $u$ at one large cylinder and compare outward.
Finally $a_\pm-c_*\sqrt s\to0$ and the correctors are $o(1)$, proving
\eqref{eq:endpoint-amplitude-o1}.
\end{proof}

\begin{proof}[Proof of Theorem~\ref{thm:main-working}\textup{(ii)} when
$|W_g(p)|\ne0$]
Proposition~\ref{prop:endpoint-second-trapping} gives
\[
  W(R\theta)
  =1-R^{-2}\left(c_*\sqrt{\log R}+o(1)\right).
\]
Since $F=W^{-1}$ and $R^{-4}\log R=o(R^{-2})$,
\[
  F(R\theta)
  =1+c_*R^{-2}\sqrt{\log R}+o(R^{-2}).
\]
With $R=\rho^{-1}$ and
$U_{\mathrm{cn}}(\rho\theta)=\rho^{-2}F(R\theta)$, this is exactly
\eqref{eq:main-n8-generic}.
\end{proof}

\subsection{The Weyl-flat finite-charge branch}
\label{subsec:n8-Weyl-flat}

Assume now that $W_g(p)=0$.  By
\eqref{eq:Aginfty-Weyl-flat-improvement},
\begin{equation}\label{eq:endpoint-flat-background}
  g_\infty-g_{\mathrm E}=O_3(R^{-3}),
  \qquad
  (A_{g_\infty})^\sharp=O_1(R^{-5}).
\end{equation}
The logarithmic source disappears and the finite-charge mechanism is
restored.

\begin{proposition}
\label{prop:endpoint-flat-critical-bound}
There is $C>0$ such that
\begin{equation}\label{eq:endpoint-flat-F-bound}
  |F(y)-1|\le C|y|^{-2}
\end{equation}
for all sufficiently large $|y|$.  Moreover, for every fixed
$\delta\in(2,3)$,
\begin{equation}\label{eq:endpoint-flat-negative-bound}
  F(y)\ge1-C_\delta|y|^{-\delta}.
\end{equation}
\end{proposition}

\begin{proof}
With \eqref{eq:endpoint-flat-background}, the proof of
Proposition~\ref{prop:low-barrier-package} applies verbatim at the critical exponent
$\alpha=2$.  The only change is that the curved error is $O(r^{-5})$, so one
may choose any $\delta\in(2,3)$ in the lower power barrier; the mixed upper
barrier then gives the summable dyadic increment $O(R^{-1})$.  This yields
\eqref{eq:endpoint-flat-negative-bound} and \eqref{eq:endpoint-flat-F-bound}.
\end{proof}

\begin{proposition}
\label{prop:endpoint-flat-locking}
There exists a unique $A\ge0$ such that
\begin{equation}\label{eq:endpoint-flat-P-limit}
  P_R(x):=R^2\bigl(W(Rx)-1\bigr)
  \longrightarrow-A|x|^{-2}
\end{equation}
locally uniformly on $\mathbb R^8\setminus\{0\}$.  Consequently,
\begin{equation}\label{eq:endpoint-flat-U-expansion}
  U_{\mathrm{cn}}(\rho\theta)=\rho^{-2}+A+o(1).
\end{equation}
\end{proposition}

\begin{proof}
By Proposition~\ref{prop:endpoint-flat-critical-bound}, the normalized family
$P_R=R^2(W(R\cdot)-1)$ satisfies exactly the compactness hypotheses of
Section~\ref{sec:subcritical}; the background errors are now $O(R^{-1})$.
Hence every blow-down is $-A|x|^{-2}$ by the same Trudinger--Wang
classification.  The averaged Hessian charge of Section~\ref{sec:subcritical}
has scale drift $O(R^{-1})$, so it converges; passing it to a tangent by the
same $BV$--$L^2$ argument as in Proposition~\ref{prop:low-coefficient-locking}
locks $A$ uniquely.  Thus \eqref{eq:endpoint-flat-P-limit} holds for the whole
family, and $F=W^{-1}$ immediately gives
\eqref{eq:endpoint-flat-U-expansion}.
\end{proof}


\section{The leading high-dimensional curvature correction}
\label{sec:high-leading-correction}

We now assume $n>8$ and treat the nonflat branch $C_0\not\equiv0$.
This section identifies the first local curvature profile and proves that the
actual Green end realizes it.

\subsection{Finite expansion of the inverted background}

We use the finite integer expansion underlying the leading formula
\eqref{eq:Aginfty-leading-Weyl}.

\begin{lemma}\label{hd:lem:finite-background-expansion}
Fix integers $L\ge2$ and $k_0\ge0$.  After imposing conformal-normal normalization to a sufficiently high finite order, there are smooth spherical tensor fields $G_q$ and $C_q$ such that, as $R\to\infty$,
\begin{align}
  g_\infty
  &=g_{\mathrm E}+\sum_{q=2}^{L}R^{-q}G_q(\theta)
    +\mathcal R^{g}_{L+1},\label{hd:eq:finite-g-expansion}\\
  (A_{g_\infty})^\sharp
  &=\sum_{q=0}^{L}R^{-4-q}{C}_q(\theta)
    +\mathcal R^{A}_{L+1},\label{hd:eq:finite-A-expansion}
\end{align}
where $C_0$ is given by \eqref{eq:C0-definition}.  Here $g_S$ and $\nabla_S$ denote the standard metric and Levi--Civita connection on $\mathbb S^{n-1}$.  Moreover, for $a,b\ge0$ with $a+b\le k_0$,
\begin{align}
  \bigl|(R\partial_R)^a\nabla_S^b\mathcal R^{g}_{L+1}\bigr|
  &\le C_{L,k_0}R^{-L-1},\label{hd:eq:finite-g-remainder}\\
  \bigl|(R\partial_R)^a\nabla_S^b\mathcal R^{A}_{L+1}\bigr|
  &\le C_{L,k_0}R^{-L-5}.\label{hd:eq:finite-A-remainder}
\end{align}
\end{lemma}

\begin{proof}
This is the standard finite conformal-normal expansion of Lee--Parker
\cite{LeeParker1987}, followed by Kelvin inversion.  Indeed,
$D\iota=R^{-2}(I-2\theta\otimes\theta)$ sends a homogeneous metric jet of
degree $m$ to $R^{-m}$ times a smooth spherical tensor, while the Schouten
tensor contains two derivatives and therefore shifts the grading by two.
Taking the conformal-normal Taylor expansion two orders further gives the
stated derivative remainders; the leading coefficient $C_0$ is exactly the
one already computed in \eqref{eq:Aginfty-leading-Weyl}.
\end{proof}

\subsection{The nonlinear spherical profile equation}

For $\mu>0$ and $\phi\in C^2(\mathbb S^{n-1})$ define the spherical homogeneous Hessian operator by
\begin{equation}\label{hd:eq:Kmu-def}
  \mathcal K_\mu[\phi]
  :=R^{\mu+2}D^2\bigl(-R^{-\mu}\phi(\theta)\bigr).
\end{equation}
In the polar orthonormal frame $e_R\oplus T_\theta\mathbb S^{n-1}$,
\begin{equation}\label{hd:eq:Kmu}
  \mathcal K_\mu[\phi]
  =
  \begin{pmatrix}
    -\mu(\mu+1)\phi & (\mu+1)\nabla_S\phi\\[2mm]
    (\mu+1)\nabla_S\phi & -\nabla_S^2\phi+\mu\phi\,g_S
  \end{pmatrix}.
\end{equation}
For the constant spherical mode set
\[
  H_\mu:=\mathcal K_\mu[1].
\]
Its eigenvalues are
\[
  -\mu(\mu+1),\qquad \mu,\ldots,\mu.
\]
Hence
\begin{lemma}\label{hd:lem:Hmu}
For every $\mu>0$,
\begin{align}
  \sigma_1(H_\mu)&=\mu(n-\mu-2),\label{hd:eq:sigma1-Hmu}\\
  \sigma_2(H_\mu)&=(n-1)\mu^2(\alpha-\mu).\label{hd:eq:sigma2-Hmu}
\end{align}
Consequently,
\begin{equation}\label{hd:eq:Hmu-cone}
  H_\mu\in\Gamma_2\quad\Longleftrightarrow\quad 0<\mu<\alpha.
\end{equation}
At the first curvature degree $\mu=2$,
\begin{equation}\label{hd:eq:H2}
  \sigma_2(H_2)=2(n-1)(n-8)>0.
\end{equation}
\end{lemma}

\begin{proof}
Both formulas follow immediately from the displayed eigenvalues and
$\alpha=(n-4)/2$.
\end{proof}

The identity \eqref{hd:eq:H2} is the basic reason why $n>8$ is easier than
the borderline dimension $n=8$: the degree-$2$ radial direction is strictly
inside $\Gamma_2$ rather than tangent to $\partial\Gamma_2$.

To identify the equation obeyed by any possible degree-two tangent,
suppose that $W$ has the weighted $C^2$ asymptotic form
\[
  W=1-R^{-2}\phi_2(\theta)+\mathcal E,
  \qquad
  (R\partial_R)^a\nabla_S^b\mathcal E=o(R^{-2})
  \quad (a+b\le2).
\]
Substitute this expansion into \eqref{eq:calA-h-W}, with
$h=g_\infty$.  The gradient-square term is
$O(R^{-6})$, while the leading Hessian and background terms are both of
order $R^{-4}$.  Consequently,
\begin{equation}\label{hd:eq:S4}
 R^4\mathcal A_{g_\infty}[W]
 =\mathcal S_4[\phi_2]+o(1),
 \qquad
 \mathcal S_4[\phi_2]:=\mathcal K_2[\phi_2]+{C}_0.
\end{equation}
By \eqref{eq:W-end-equation}, the first coefficient must satisfy the sphere
equation
\begin{equation}\label{hd:eq:cell}
  \lambda\bigl(\mathcal S_4[\phi_2]\bigr)\in\partial\Gamma_2.
\end{equation}
Equivalently, $\sigma_2(\mathcal K_2[\phi_2]+C_0)=0$ on the branch
$\lambda(\mathcal K_2[\phi_2]+C_0)\in\overline{\Gamma_2}$.

\subsection{The first nonlinear cell}
\label{hd:sec:cell-regularity}

\begin{proposition}\label{hd:prop:first-cell}
The cell equation \eqref{hd:eq:cell} has a unique admissible solution
$\phi_2\in C^\infty(\mathbb S^{n-1})$. It is nonnegative, positive when
$C_0\not\equiv0$, and in the nonflat case, with
\[
  \mathcal S_4:=\mathcal K_2[\phi_2]+C_0,
\]
there is $c_0>0$ such that
\begin{equation}\label{hd:eq:first-cell-positive}
  T_1(\mathcal S_4)\ge c_0 I.
\end{equation}
If $C_0\equiv0$, then $\phi_2\equiv0$.
\end{proposition}

\begin{proof}
Put $m=n-1$.  For $0<\varepsilon\leq1$ and $0\leq t\leq1$, consider
the strict problem
\begin{equation}\label{hd:eq:first-cell-regularized}
 \sigma_2(\mathcal S_t[\phi])=\varepsilon,
 \qquad
 \lambda(\mathcal S_t[\phi])\in\Gamma_2,
 \qquad
 \mathcal S_t[\phi]:=\mathcal K_2[\phi]+tC_0.
\end{equation}
In the polar orthonormal frame,
\[
 \mathcal S_t[\phi]
 =\begin{pmatrix}
   -6\phi&3\nabla\phi\\
   3\nabla\phi&D_t
  \end{pmatrix},
 \qquad
 D_t:=-\nabla^2\phi+2\phi g_S+tC_0,
\]
and hence
\begin{equation}\label{hd:eq:first-cell-block}
 \sigma_2(D_t)-6\phi\,\tr D_t-9|\nabla\phi|^2=\varepsilon.
\end{equation}
Set
\[
 \delta:=\frac6{m-1},
 \qquad
 a:=2-\delta=\frac{2(n-5)}{n-2}>0,
 \qquad
 B_t:=D_t-\delta\phi g_S
      =-\nabla^2\phi+a\phi g_S+tC_0.
\]
Using
\[
 \sigma_2(B+sI)
 =\sigma_2(B)+(m-1)s\sigma_1(B)+\binom m2s^2
\]
in \eqref{hd:eq:first-cell-block}, the linear trace terms cancel and give
the exact shifted equation
\begin{equation}\label{hd:eq:first-cell-shifted-equation}
 \sigma_2(B_t)
 =9|\nabla\phi|^2+c_n\phi^2+\varepsilon,
 \qquad
 c_n:=\frac{18m}{m-1}.
\end{equation}

We first verify the branch of this equation.  A strict solution is positive:
at a nonpositive minimum, using $\tr C_0=0$, one would have
\[
 \sigma_1(\mathcal S_t[\phi])
 =-\Delta\phi+2(n-4)\phi\leq0,
\]
contrary to admissibility.  It follows from
\eqref{hd:eq:first-cell-block} that
\[
 \sigma_1(D_t)>0,
 \qquad
 \sigma_2(D_t)
 =\varepsilon+6\phi\sigma_1(D_t)+9|\nabla\phi|^2>0.
\]
Thus $D_t\in\Gamma_2^{(m)}$.  Writing $T_D:=\sigma_1(D_t)$,
Newton--Maclaurin and the preceding identity imply
\[
 6\phi T_D\leq\sigma_2(D_t)
 \leq\frac{m-1}{2m}T_D^2,
 \qquad
 T_D\geq\frac{12m}{m-1}\phi.
\]
Consequently,
\[
 \sigma_1(B_t)
 =T_D-\frac{6m}{m-1}\phi
 \geq\frac{6m}{m-1}\phi>0.
\]
Together with \eqref{hd:eq:first-cell-shifted-equation}, this proves
\begin{equation}\label{hd:eq:first-cell-B-branch}
 B_t\in\Gamma_2^{(m)}.
\end{equation}

We next establish estimates independent of $t$ and $\varepsilon$.  Let
$\mathscr F(A):=\sigma_2(A)^{1/2}$ on $\Gamma_2$.  Concavity and
one-homogeneity give
\[
 \mathscr F(A+B)\geq\mathscr F(A)+\mathscr F(B),
 \qquad A,B\in\Gamma_2.
\]
Suppose $u$ and $v$ solve \eqref{hd:eq:first-cell-regularized} with the
same right-hand side but with tangential data $C$ and $D$.  Since
$H_2\in\Gamma_2^\circ$, when $C\ne D$ one can choose
$k=C_n\|C-D\|_{L^\infty}$ so that
$kH_2\pm(C-D)\in\Gamma_2^\circ$; when $C=D$, set $k=0$.  If $u-v-k$
had a positive maximum,
then at that point
\[
 \bigl(\mathcal K_2[u]+C\bigr)
 -\bigl(\mathcal K_2[v]+D\bigr)
 =(k+q)H_2+(C-D)
  +\begin{pmatrix}0&0\\0&-\nabla^2q\end{pmatrix}
 \in\Gamma_2^\circ,
\]
where $q=u-v-k>0$.  Superadditivity of $\mathscr F$ would then make the
left strict value larger than the right one, a contradiction.  Interchanging
$u$ and $v$ gives
\begin{equation}\label{hd:eq:first-cell-data-comparison}
 \|u-v\|_{L^\infty}
 \leq C_n\|C-D\|_{L^\infty}.
\end{equation}
At $t=0$, the constant
\[
 \phi_{0,\varepsilon}
 =\sqrt{\frac{\varepsilon}{\sigma_2(H_2)}}
\]
solves the strict equation.  Comparing $tC_0$ with zero data gives a uniform
$C^0$ bound.  If $A$ is a small rotation of the sphere, then
$\phi_A(\theta):=\phi(A\theta)$ solves the equation with data $A^*(tC_0)$,
so \eqref{hd:eq:first-cell-data-comparison} yields
\[
 \|\phi-\phi_A\|_{L^\infty}
 \leq C\|tC_0-A^*(tC_0)\|_{L^\infty}
 \leq C d(A,I).
\]
Rotations carrying one nearby point to another therefore give
\begin{equation}\label{hd:eq:first-cell-C0C1}
 \|\phi\|_{C^0}+\|\nabla\phi\|_{C^0}\leq C.
\end{equation}

For the second derivative estimate, put
\[
 \Psi:=\sqrt{9|\nabla\phi|^2+c_n\phi^2+\varepsilon},
 \qquad
 \mathscr F(B_t)=\Psi,
 \qquad
 T:=\tr B_t=-\Delta\phi+am\phi.
\]
The spherical Hessian commutation formula gives
\begin{equation}\label{hd:eq:first-cell-T-identity}
 -\nabla^2T+aTg_S
 =-\Delta B_t+m(2+a)B_t-2Tg_S+E_t,
 \qquad
 E_t:=t\bigl(\Delta C_0-m(2+a)C_0\bigr).
\end{equation}
Let
$\mathscr F^{ij}:=\partial\mathscr F/\partial(B_t)_{ij}$.  Homogeneity
and the explicit derivative of $\sqrt{\sigma_2}$ give
\[
 \mathscr F^{ij}(B_t)_{ij}=\Psi,
 \qquad
 \mathcal T:=\sum_i\mathscr F^{ii}
 =\frac{(m-1)T}{2\Psi}.
\]
At a maximum point of $T$, contract
\eqref{hd:eq:first-cell-T-identity} with $\mathscr F^{ij}$.  Since
$\mathscr F$ is concave and $E_t$ is uniformly bounded, one obtains
\begin{equation}\label{hd:eq:first-cell-T-master}
 (a+2)T\mathcal T
 \leq-\Delta\Psi+m(2+a)\Psi+C\mathcal T.
\end{equation}
The function
$(z,p)\mapsto(c_nz^2+9|p|^2+\varepsilon)^{1/2}$ is convex.  The covariant
chain rule, with the nonnegative quadratic terms discarded, gives
\[
 \Delta\Psi
 \geq\frac{c_n\phi}{\Psi}\Delta\phi
 +\frac{9\phi_i}{\Psi}
   \bigl(\nabla_i\Delta\phi+(m-1)\phi_i\bigr).
\]
At the maximum of $T$,
$\nabla\Delta\phi=am\nabla\phi$ and
$\Delta\phi=am\phi-T$.  Hence
\begin{equation}\label{hd:eq:first-cell-Psi-bound}
 -\Delta\Psi
 \leq\frac{c_n\phi}{\Psi}T
 \leq\sqrt{c_n}\,T.
\end{equation}
The bound \eqref{hd:eq:first-cell-C0C1} gives $\Psi\leq C$, and therefore
$\mathcal T\geq cT$.  If $T$ is large, the last term of
\eqref{hd:eq:first-cell-T-master} is absorbed by its left-hand side;
using \eqref{hd:eq:first-cell-Psi-bound} then gives
$cT^2\leq C(1+T)$.  Thus $T\leq C$.  Since
\[
 |B_t|^2=T^2-2\sigma_2(B_t)\leq T^2
\]
on the branch \eqref{hd:eq:first-cell-B-branch}, we conclude that
\begin{equation}\label{hd:eq:first-cell-C2}
 \|\nabla^2\phi\|_{L^\infty}\leq C.
\end{equation}

We now solve the strict problem.  For fixed $\varepsilon>0$, let
$I_\varepsilon$ be the set of $t\in[0,1]$ for which
\eqref{hd:eq:first-cell-regularized} has a smooth admissible solution.
The constant solution above shows $0\in I_\varepsilon$.  At a solution,
the linearization in the $\phi$ variable is
\begin{equation}\label{hd:eq:first-cell-linearization}
 Lh=T_1(\mathcal S_t[\phi]):\mathcal K_2[h].
\end{equation}
Its principal part is elliptic.  Its zeroth-order coefficient is positive:
if $N:=T_1(\mathcal S_t[\phi])$, concavity and homogeneity of
$\sqrt{\sigma_2}$ imply
\[
 N:H_2
 \geq2\sqrt\varepsilon\,\sqrt{\sigma_2(H_2)}>0.
\]
The maximum principle gives $\ker L=\{0\}$.  The index is zero by a
homotopy through elliptic operators with positive zeroth-order coefficient
to $-\Delta_S+1$, so the implicit function theorem proves openness of
$I_\varepsilon$.

For closedness, \eqref{hd:eq:first-cell-C0C1} and
\eqref{hd:eq:first-cell-C2} give uniform bounds.  Since $\varepsilon$ is
fixed, \eqref{hd:eq:first-cell-shifted-equation} gives
$\sigma_2(B_t)\geq\varepsilon$; together with $|B_t|\leq C$, the shifted
matrices remain in a compact subset of $\Gamma_2^{(m)}$.  The shifted
equation is then uniformly elliptic.  Evans--Krylov and Schauder estimates
give smooth compactness, proving closedness.  Thus
$I_\varepsilon=[0,1]$, and uniqueness follows from
\eqref{hd:eq:first-cell-data-comparison} with identical data.  In
particular, at $t=1$ there is a unique smooth strict solution
$\phi_\varepsilon$, and the bounds above are independent of
$\varepsilon\in(0,1]$.

It remains to pass to the degenerate equation.  Every sequence
$\varepsilon_j\downarrow0$ has a subsequence for which
$\phi_{\varepsilon_j}$ converges uniformly and in $C^{1,\beta}$ to a
nonnegative admissible solution $\phi_*$.  If $C_0\not\equiv0$, this limit
cannot vanish identically: otherwise $C_0\in\overline{\Gamma_2}$, whereas
at every point where $C_0\ne0$,
\[
 \sigma_2(C_0)
 =\frac12\bigl((\tr C_0)^2-|C_0|^2\bigr)
 =-\frac12|C_0|^2<0.
\]
Moreover, the strict trace inequality passes to the limit as
\[
 -\Delta\phi_*+2(n-4)\phi_*\geq0.
\]
The strong maximum principle gives $\phi_*>0$.  A compactness
contradiction therefore yields constants $c_*>0$ and
$\varepsilon_0>0$ such that
\begin{equation}\label{hd:eq:first-cell-uniform-positive}
 \min_{\mathbb S^{n-1}}\phi_\varepsilon\geq c_*
 \qquad(0<\varepsilon<\varepsilon_0).
\end{equation}
Consequently,
\[
 \sigma_2(B_1[\phi_\varepsilon])
 \geq c_nc_*^2,
 \qquad
 |B_1[\phi_\varepsilon]|\leq C.
\]
The shifted matrices now remain in a fixed compact subset of
$\Gamma_2^{(m)}$ uniformly as $\varepsilon\downarrow0$.  Evans--Krylov
and Schauder theory give uniform estimates of every order, so the
subsequential convergence is smooth.

The limiting cell is unique.  Indeed, if $u-v$ had a positive maximum
$q$ at some point, then
\[
 \mathcal K_2[u-v]
 =qH_2+
  \begin{pmatrix}0&0\\0&-\nabla^2(u-v)\end{pmatrix}
 \in\Gamma_2^\circ
\]
there.  Adding this matrix to
$\mathcal K_2[v]+C_0\in\partial\Gamma_2$ would put
$\mathcal K_2[u]+C_0$ in $\Gamma_2^\circ$, a contradiction.  Interchanging
$u$ and $v$ proves equality.  Hence the whole family converges to the
unique smooth cell $\phi_2$.  If $C_0\equiv0$, the zero function solves
the limiting equation and the same comparison proves uniqueness, so
$\phi_2\equiv0$.

Finally assume $C_0\not\equiv0$ and put
$\mathcal S_4=\mathcal K_2[\phi_2]+C_0$.  Then
$\mathcal S_4(e_R,e_R)=-6\phi_2<0$.  On
$\partial\Gamma_2$, $T_1(\mathcal S_4)$ is positive semidefinite.  To
exclude a zero eigenvalue, diagonalize a matrix
$S\in\partial\Gamma_2$ with $\sigma_1(S)>0$.  If the $i$th eigenvalue of
$T_1(S)$ vanishes, then
\[
 \sum_{j\ne i}\lambda_j=0,
 \qquad
 0=\sigma_2(\lambda_1,\ldots,\widehat{\lambda_i},\ldots,\lambda_n)
   =-\frac12\sum_{j\ne i}\lambda_j^2.
\]
Thus $S$ is positive rank one.  For $S=\mathcal S_4$, one has
$\sigma_1(S)>0$ because $\sigma_1(S)=0=\sigma_2(S)$ would imply $S=0$,
contradicting its negative radial entry.  The rank-one alternative is also
incompatible with that entry.  Hence $T_1(\mathcal S_4)$ is pointwise
positive definite, and compactness of the sphere gives
\eqref{hd:eq:first-cell-positive}.
\end{proof}

\subsection{Realization of the first correction for the \texorpdfstring{$\Gamma_2$}{Gamma-2}-Green function}

We use the Green-end comparison and strict approximation from
Proposition~\ref{prop:gamma2-green-theory}, together with an initial decay
for one fixed exponent $\beta_0\in(0,2)$:
\begin{equation}\label{hd:eq:subcritical}
  |W-1|\le C_0 R^{-\beta_0}
  \qquad\text{for all sufficiently large }R,
\end{equation}
Only this single exponent is needed because all subsequent correction
barriers decay faster than $R^{-2}$.

\begin{proposition}\label{hd:prop:R-2}
If $n>8$, the subcritical estimate \eqref{hd:eq:subcritical} improves to
\begin{equation}\label{hd:eq:R-2}
  |W-1|\le CR^{-2}.
\end{equation}
\end{proposition}

\begin{proof}[Barrier argument]
Consider
\[
  \underline W_M=1-MR^{-2},
  \qquad
  \overline W_M=1+MR^{-2}.
\]
For these profiles the full $R^4$-normalized Schouten matrices have the uniform expansions
\begin{align}
  R^4\mathcal A_{g_\infty}[\underline W_M]
  &=MH_2+{C}_0
    +O\!\left(R^{-1}+MR^{-2}+M^2R^{-2}\right),
    \label{hd:eq:R2-barrier-lower-full}\\
  R^4\mathcal A_{g_\infty}[\overline W_M]
  &=-MH_2+{C}_0
    +O\!\left(R^{-1}+MR^{-2}+M^2R^{-2}\right).
    \label{hd:eq:R2-barrier-upper-full}
\end{align}
Indeed, the $O(R^{-1})$ term is the next Kelvin-background coefficient, the metric/connection and background-factor corrections are $O(MR^{-2})$, while the factor $W\Hess W$ and the gradient-square term contribute the genuinely nonlinear size $O(M^2R^{-2})$.

Choose $a\in(1/2,(2-\beta_0)^{-1})$ and set $R_0=M^a$.
The elementary exponent comparison encoded in this choice gives, uniformly
for $R\ge R_0$,
\[
 MR^{-2}=o(1),\qquad
 R^{-1}+MR^{-2}+M^2R^{-2}=o(M),\qquad
 MR_0^{-2}\gg R_0^{-\beta_0}.
\]
The first two relations make the barriers positive and preserve the signs of
their leading matrices; the last one will supply the inner ordering.
Since $H_2\in\Gamma_2^\circ$ and ${C}_0$ is bounded, $(MH_2+{C}_0)/M\to H_2$ uniformly in $\theta$; hence the lower matrix stays a distance $cM$ inside $\Gamma_2$ for all $R\ge R_0$ when $M$ is large.  For the upper profile,
\[
  \sigma_1(-MH_2+{C}_0)=-M\sigma_1(H_2),
\]
because $\tr{C}_0=0$; the error above is $o(M)$, so its trace is strictly negative on the whole exterior region.  Thus $\underline W_M$ is strictly admissible and $\overline W_M$ is strictly exterior for every $R\ge R_0$.

At $R=R_0$, the last relation above and \eqref{hd:eq:subcritical} give,
after increasing $M$ once more,
\[
  \underline W_M\le W\le\overline W_M
  \qquad\text{on }\partial B_{R_0}.
\]
Both barriers approach $1$ at rate $R^{-2}$, and $2>\beta_0$.  Hence Lemma \ref{hd:lem:W-multiplicative-compensation} applies to both profiles and gives
\[
  1-MR^{-2}\le W\le1+MR^{-2}
  \qquad(R\ge R_0),
\]
which is \eqref{hd:eq:R-2}.
\end{proof}

\begin{lemma}\label{hd:lem:W-multiplicative-compensation}
Assume \eqref{hd:eq:subcritical} and fix a sufficiently large $R_0$.
\begin{enumerate}[label=(\roman*)]
\item If $\underline W>0$ is a smooth strictly admissible profile on $\{R\ge R_0\}$ and
\[
  \underline W\le W\quad\text{on }\partial B_{R_0},\qquad
  |\underline W-1|\le CR^{-\gamma}
\]
for some $\gamma>\beta_0$, then $\underline W\le W$ throughout $\{R\ge R_0\}$.
\item If $\overline W>0$ lies strictly outside the admissible branch on $\{R\ge R_0\}$ and
\[
  W\le\overline W\quad\text{on }\partial B_{R_0},\qquad
  |\overline W-1|\le CR^{-\gamma}
\]
for some $\gamma>\beta_0$, then $W\le\overline W$ throughout $\{R\ge R_0\}$.
\end{enumerate}
The comparison direction is the one appropriate to $W$: since $W=F^{-2/\alpha}$ reverses order, a strictly admissible $W$-profile is a lower barrier and a strictly exterior profile is an upper barrier.
\end{lemma}

\begin{proof}
We prove (i); the other case is symmetric.  On a truncated annulus $B_L\setminus B_{R_0}$ set
\[
  c_L^-:=1-\kappa L^{-\beta_0}>0
\]
with $\kappa>2C_0$ fixed and $L$ sufficiently large.  Multiplication by a positive constant preserves the Schouten cone orientation: $(c_L^-\underline W)^{-2}g_\infty$ is a constant homothety of $\underline W^{-2}g_\infty$, so its Schouten endomorphism differs only by a positive scalar factor.  On the inner boundary,
\[
  c_L^-\underline W\le\underline W\le W.
\]
On $R=L$, since $\gamma>\beta_0$,
\[
  c_L^-\underline W
  \le (1-\kappa L^{-\beta_0})(1+C L^{-\gamma})
  =1-\kappa L^{-\beta_0}+o(L^{-\beta_0})
  \le 1-C_0L^{-\beta_0}\le W
\]
for all large $L$.  The admissible comparison principle on the truncated annulus gives $c_L^-\underline W\le W$.  Keeping any fixed $R>R_0$ and sending $L\to\infty$ yields $\underline W\le W$.

For (ii), take $c_L^+:=1+\kappa L^{-\beta_0}$ with the same type of choice.  Then $c_L^+\overline W\ge\overline W\ge W$ on the inner boundary and $c_L^+\overline W\ge W$ on $R=L$ for large $L$.  Constant multiplication preserves the strict exterior cone orientation, so comparison gives $W\le c_L^+\overline W$ and then $W\le\overline W$ as $L\to\infty$.
\end{proof}

Let $\phi_2$ and $\mathcal S_4$ be as in
Proposition~\ref{hd:prop:first-cell}, and set
$N:=T_1(\mathcal S_4)\ge c_0I$ in the nonflat branch.

Fix
\begin{equation}\label{hd:eq:mu-range}
  2<\mu<\min\{3,\alpha\}.
\end{equation}
Define
\begin{equation}\label{hd:eq:refined-barriers}
  W_{-}
  :=1-R^{-2}\phi_2(\theta)-KR^{-\mu},
  \qquad
  W_{+}
  :=1-R^{-2}\phi_2(\theta)+KR^{-\mu}.
\end{equation}
Since
\[
  D^2(-R^{-\mu})=R^{-\mu-2}H_\mu,
\]
the first barrier receives a $+K R^{-\mu-2}H_\mu$ perturbation, while the second receives the opposite perturbation.

\begin{lemma}\label{hd:lem:refined-smooth-barriers}
Let $\phi_2$ be the smooth first cell from Proposition~\ref{hd:prop:first-cell}, and fix $\mu$ in \eqref{hd:eq:mu-range}.  For every fixed $D>0$ there is $R_D\gg1$ such that, whenever $R_0\ge R_D$ and
\begin{equation}\label{hd:eq:K-from-D}
  K:=DR_0^{\mu-2},
\end{equation}
the profiles in \eqref{hd:eq:refined-barriers} have opposite strict cone orientations for every $R\ge R_0$: $W_-$ is strictly admissible and $W_+$ lies strictly outside the admissible branch.  For fixed $D$ the strictness is uniform on the whole exterior region.
\end{lemma}

\begin{proof}
With \eqref{hd:eq:K-from-D}, write
\begin{equation}\label{hd:eq:deltaR-D}
  \delta_R:=KR^{2-\mu}
  =D\left(\frac{R}{R_0}\right)^{2-\mu}.
\end{equation}
Since $\mu>2$, one has $0<\delta_R\le D$ and $\delta_R\to0$ as $R\to\infty$.  Since $\mu<3$,
$R^{-1}/\delta_R\le (DR_0)^{-1}$ on the whole exterior region.
Thus the next geometric term is uniformly $o_{R_0\to\infty}(\delta_R)$.
The same conclusion holds for all products involving $W_\pm-1$, the metric
correction, and the quadratic gradient term, since they are
$O_D(R^{-2})$.  Hence
\begin{align}
  R^4\mathcal A_{g_\infty}[W_-]
  &=\mathcal S_4+\delta_R H_\mu+E_-(R,\theta),\label{hd:eq:refined-matrix-expansion}\\
  R^4\mathcal A_{g_\infty}[W_+]
  &=\mathcal S_4-\delta_R H_\mu+E_+(R,\theta),\label{hd:eq:refined-matrix-expansion-plus}
\end{align}
with
\begin{equation}\label{hd:eq:refined-error-relative}
  \sup_{R\ge R_0,\theta}\frac{|E_\pm(R,\theta)|}{\delta_R}
  \longrightarrow0
  \qquad(R_0\to\infty),
\end{equation}
for each fixed $D$.

For the inward direction, $H_\mu\in\Gamma_2^\circ$ and
$N=T_1(\mathcal S_4)\ge c_0I$; therefore, uniformly in $\theta$,
\[
  \sigma_1(\mathcal S_4+\delta H_\mu)>0,
  \qquad
  \sigma_2(\mathcal S_4+\delta H_\mu)
  =\delta\,T_1(\mathcal S_4):H_\mu+\delta^2\sigma_2(H_\mu)>0
\]
for every $\delta>0$.  For $0<\delta\le D$, the strictness is stable under
an $o(\delta)$ perturbation: use the positive linear term near $0$ and
compactness on $[\delta_0,D]$.  Thus $W_-$ is strictly admissible for large
$R_0$.

For the outward direction, no smallness assumption on $\delta$ is needed.  We claim that
\begin{equation}\label{hd:eq:global-outward-cone}
  \mathcal S_4-\delta H_\mu\notin\overline{\Gamma_2}
  \qquad\text{for every }\delta>0.
\end{equation}
Indeed, if \eqref{hd:eq:global-outward-cone} failed for some $\delta>0$, then the convex-cone property and $\delta H_\mu\in\Gamma_2^\circ$ would give
\[
  \mathcal S_4=(\mathcal S_4-\delta H_\mu)+\delta H_\mu\in\Gamma_2^\circ,
\]
contradicting $\mathcal S_4\in\partial\Gamma_2$.  Quantitatively,
\[
  \sigma_2(\mathcal S_4-\delta H_\mu)
  =-\delta\,T_1(\mathcal S_4):H_\mu+O(\delta^2)
  \le-c\delta
\]
for $0<\delta\le\delta_0$, while compactness gives a positive distance from
$\overline{\Gamma_2}$ on $[\delta_0,D]$.  The relative error estimate
\eqref{hd:eq:refined-error-relative} therefore preserves the exterior
orientation uniformly for all $R\ge R_0$ once $R_0$ is large.
\end{proof}

\begin{theorem}\label{hd:thm:first-tangent}
Assume $C_0\not\equiv0$, the subcritical decay \eqref{hd:eq:subcritical}, and
the admissible strict comparison principle.  Then, for every
\[
  0<\eps<\min\{1,\alpha-2\},
\]
one has
\begin{equation}\label{hd:eq:W-first-exp}
  W(R,\theta)=1-R^{-2}\phi_2(\theta)+O(R^{-2-\eps})
\end{equation}
uniformly in $\theta$, where $\phi_2$ is the unique smooth first cell from Proposition~\ref{hd:prop:first-cell}.  Equivalently, with $\rho=R^{-1}$,
\begin{equation}\label{hd:eq:U-first-exp}
  U_{\rm cn}(\rho,\theta)=\rho^{-\alpha}\left(1+\frac{\alpha}{2}\rho^2\phi_2(\theta)+O(\rho^{2+\eps})\right).
\end{equation}
\end{theorem}

\begin{proof}
Choose $\mu=2+\eps$.  By \eqref{hd:eq:R-2} and the boundedness of $\phi_2$, there is a constant $C_0$ such that
\[
  \left|W(R,\theta)-\bigl(1-R^{-2}\phi_2(\theta)\bigr)\right|
  \le C_0R^{-2}
\]
for all sufficiently large $R$.  Fix once and for all $D>2C_0$.  Lemma \ref{hd:lem:refined-smooth-barriers} supplies $R_D$; choose an inner radius $R_0\ge R_D$ and set $K=D R_0^{\mu-2}$.  Then at $R=R_0$,
\[
  KR_0^{-\mu}=DR_0^{-2},
\]
so
\[
  W_-\le W\le W_+
  \qquad\text{on }\partial B_{R_0}.
\]
The lemma supplies the strict cone orientations on the full exterior region $R\ge R_0$.  Notice that the comparison order here is the $W$-version: since $W=F^{-2/\alpha}$ is order reversing, a strictly admissible $W$-profile is a lower barrier and a strictly exterior $W$-profile is an upper barrier.  Since both refined profiles approach $1$ at rate $O(R^{-2})$ and $2>\beta_0$, Lemma \ref{hd:lem:W-multiplicative-compensation} applies directly to the strict lower and upper barriers and yields
\[
  \left|W-(1-R^{-2}\phi_2)\right|\le KR^{-\mu}.
\]
Here $K$ is fixed after $D$ and $R_0$ have been chosen, so this is the desired $O(R^{-\mu})$ estimate.  Since $F=W^{-\alpha/2}$ and $\mu<3<4$, the quadratic Taylor term in $W-1$ is $O(R^{-4})=o(R^{-\mu})$, which gives \eqref{hd:eq:U-first-exp}.
\end{proof}

\section{Indicial theory and the local high-dimensional parametrix}
\label{sec:high-indicial-parametrix}

The nondegenerate first cell supplies the uniformly positive Newton tensor
needed to linearize.  The next two steps are structural: first determine the
real indicial spectrum, then turn the indicial gap into a comparison theorem
for the nonlinear remainder.

\subsection{The linearized spherical family}

\begin{equation}\label{hd:eq:uniform-N}
  \mathcal S_4:=\mathcal K_2[\phi_2]+C_0,
  \qquad
  N:=T_1(\mathcal S_4)\ge c_0I>0.
\end{equation}
For a real exponent $\mu>0$ define
\begin{equation}\label{hd:eq:indicial}
  \mathcal L_\mu\varphi
  :=T_1(\mathcal S_4):\mathcal K_\mu[\varphi].
\end{equation}
In local spherical coordinates,
\begin{equation}\label{hd:eq:Lmu-local}
  \mathcal L_\mu\varphi
  =-N^{ab}\nabla_{ab}\varphi
  +2(\mu+1)N^{Ra}\nabla_a\varphi
  +c_\mu(\theta)\varphi,
\end{equation}
where
\begin{equation}\label{hd:eq:cmu}
  c_\mu(\theta):=N:H_\mu.
\end{equation}

\begin{lemma}\label{hd:lem:cmu-positive}
If $0<\mu<\alpha$, then for every $\theta\in\mathbb S^{n-1}$,
\[
  c_\mu(\theta)>0.
\]
\end{lemma}

\begin{proof}
By Lemma~\ref{hd:lem:Hmu}, $H_\mu\in\Gamma_2^\circ$, whereas
$N=T_1(\mathcal S_4)$ is the nonzero supporting normal at
$\mathcal S_4\in\partial\Gamma_2$; hence $N:H_\mu>0$.
\end{proof}

\begin{proposition}\label{hd:prop:no-indicial}
Under \eqref{hd:eq:uniform-N},
\[
  \mathcal L_\mu:C^{2,\beta}(\mathbb S^{n-1})\to C^{0,\beta}(\mathbb S^{n-1})
\]
is an isomorphism for every $\beta\in(0,1)$ and every $0<\mu<\alpha$.

\end{proposition}

\begin{proof}
The maximum principle applied to \eqref{hd:eq:Lmu-local}, using $c_\mu>0$,
gives trivial kernel.  Uniform ellipticity and the standard Fredholm
continuity argument on the compact sphere give surjectivity.
\end{proof}

\subsection{Homogeneous divergence structure and adjoint reflection}

Define the degree $-2$ Euclidean profile
\begin{equation}\label{hd:eq:u0-leading}
  u_0(x):=-|x|^{-2}\phi_2\!\left(\frac{x}{|x|}\right)
\end{equation}
and the homogeneous Weyl tensor field
\begin{equation}\label{hd:eq:Chat-leading}
  \widehat C_{ij}(x)
  :=|x|^{-4}({C}_0)_{ij}(x/|x|)
  =\varsigma\frac23\,{W}_{ikjl}(p)\frac{x_kx_l}{|x|^6}.
\end{equation}
Then
\begin{equation}\label{hd:eq:S-infty}
  S_\infty:=D^2u_0+\widehat C
  =|x|^{-4}\mathcal S_4(x/|x|).
\end{equation}
Set
\begin{equation}\label{hd:eq:N-infty}
  N_\infty:=T_1(S_\infty)
  =|x|^{-4}N(x/|x|).
\end{equation}

\begin{lemma}\label{hd:lem:divfree-N-infty}
On $\mathbb R^n\setminus\{0\}$,
\begin{equation}\label{hd:eq:divfree-N-infty}
  \operatorname{div}N_\infty=0.
\end{equation}
Consequently, for every smooth $q$,
\begin{equation}\label{hd:eq:L-infty-divform}
  N_\infty:D^2q=\operatorname{div}(N_\infty\nabla q).
\end{equation}
\end{lemma}

\begin{proof}
Both $\operatorname{div}T_1(D^2u_0)=0$ and
$\operatorname{div}\widehat C=0$ follow from commuting derivatives and the
Weyl symmetries; linearity of $T_1$ gives \eqref{hd:eq:divfree-N-infty}, and
the second identity is the product rule.
\end{proof}

Define the homogeneous Euclidean leading linearization
\begin{equation}\label{hd:eq:Linfty-def}
  \mathscr L_\infty q
  :=N_\infty:D^2q
  =\operatorname{div}(N_\infty\nabla q).
\end{equation}
Although $N_\infty$ has degree $-4$, its angular part is uniformly positive definite because $N\ge c_0I$.

Extend the definitions of $\mathcal K_\mu$ and $\mathcal L_\mu$ algebraically to every real $\mu$ by the same formulas as in \eqref{hd:eq:Kmu} and \eqref{hd:eq:indicial}.  For
\[
  q_\mu(x):=-|x|^{-\mu}\psi(x/|x|)
\]
one has the exact homogeneity relation
\begin{equation}\label{hd:eq:Linfty-indicial}
  \mathscr L_\infty q_\mu
  =|x|^{-\mu-6}\mathcal L_\mu\psi.
\end{equation}

\begin{proposition}\label{hd:prop:indicial-adjoint}
With respect to the standard $L^2(\mathbb S^{n-1},d\theta)$ pairing, for every real $\mu$,
\begin{equation}\label{hd:eq:indicial-adjoint}
  \mathcal L_\mu^*=\mathcal L_{\,n-6-\mu}.
\end{equation}
\end{proposition}

\begin{proof}
Fix real $\mu,\nu$ satisfying
\begin{equation}\label{hd:eq:mu-nu-dual}
  \mu+\nu=n-6,
\end{equation}
and set
\[
  q_\mu=-r^{-\mu}\psi(\theta),
  \qquad
  q_\nu=-r^{-\nu}\chi(\theta).
\]
Because of Lemma \ref{hd:lem:divfree-N-infty}, Green's identity on the Euclidean annulus $A_{a,b}=\{a<r<b\}$ gives
\begin{equation}\label{hd:eq:green-indicial}
\begin{aligned}
  &\int_{A_{a,b}}
  \bigl(q_\nu\mathscr L_\infty q_\mu
        -q_\mu\mathscr L_\infty q_\nu\bigr)\,dx\\
  &\qquad=
  \int_{\partial A_{a,b}}
  \bigl(q_\nu N_\infty\nabla q_\mu
        -q_\mu N_\infty\nabla q_\nu\bigr)\cdot\nu_{\partial A}\,dS.
\end{aligned}
\end{equation}
The boundary integrand, including $dS$, is homogeneous of degree
\[
  n-1-\nu-4-(\mu+1)
  =n-6-\mu-\nu=0.
\]
Thus its angular value is the same on $r=a$ and $r=b$; the opposite outward normal on the inner boundary makes the two fluxes cancel.  The right-hand side of \eqref{hd:eq:green-indicial} is therefore zero.

Using \eqref{hd:eq:Linfty-indicial} and \eqref{hd:eq:mu-nu-dual}, the radial factor on the left is
\[
  r^{n-1-\mu-\nu-6}\,dr=r^{-1}\,dr.
\]
Hence
\[
  \log\frac ba
  \left(
    -\int_{\mathbb S^{n-1}}\chi\,\mathcal L_\mu\psi\,d\theta
    +\int_{\mathbb S^{n-1}}\psi\,\mathcal L_\nu\chi\,d\theta
  \right)=0.
\]
Since $a<b$ are arbitrary,
\[
  \int_{\mathbb S^{n-1}}\chi\,\mathcal L_\mu\psi\,d\theta
  =\int_{\mathbb S^{n-1}}\psi\,\mathcal L_{n-6-\mu}\chi\,d\theta,
\]
which is exactly \eqref{hd:eq:indicial-adjoint}.
\end{proof}

Each $\mathcal L_\mu$ is scalar uniformly elliptic on the compact sphere,
with principal symbol independent of $\mu$; it is therefore Fredholm of
index zero.

\subsection{Absence of real roots below the first resonance}

\begin{theorem}\label{hd:thm:extended-indicial-gap}
In the nonflat first-cell branch, for every $0<\mu<n-6$,
\begin{equation}\label{hd:eq:extended-indicial-gap}
  \mathcal L_\mu:C^{2,\beta}(\mathbb S^{n-1})\longrightarrow C^{0,\beta}(\mathbb S^{n-1})\quad\text{is an isomorphism}.
\end{equation}
In particular,
\begin{equation}\label{hd:eq:Lalpha-invertible}
 \mathcal L_\alpha\text{ is invertible}.
\end{equation}
\end{theorem}

\begin{proof}
For $0<\mu<\alpha$ use Proposition~\ref{hd:prop:no-indicial}.  If
$\alpha\le\mu<n-6$, then $\nu=n-6-\mu\in(0,\alpha)$, so $\mathcal L_\nu$
is invertible; adjoint symmetry gives $\operatorname{coker}\mathcal L_\mu=0$
and index zero gives $\ker\mathcal L_\mu=0$.
\end{proof}

\begin{proposition}\label{hd:prop:first-positive-indicial-root}
In the nonflat first-cell branch,
\[
  \ker\mathcal L_0=\operatorname{span}\{1\},
  \qquad
  \dim\ker\mathcal L_{n-6}=1.
\]
Hence $\mu=n-6$ is the first strictly positive real indicial root.
\end{proposition}

\begin{proof}
At $\mu=0$ the strong maximum principle gives
$\ker\mathcal L_0=\operatorname{span}\{1\}$.  Adjoint symmetry and index
zero imply $\dim\ker\mathcal L_{n-6}=1$; the preceding indicial gap excludes
all positive roots below $n-6$.
\end{proof}

\subsection{Positive anisotropic weights}

\begin{lemma}\label{hd:lem:positive-anisotropic-weight}
For every real exponent
\[
  0<\mu<n-6,
\]
there is a unique smooth function $\zeta_\mu$ such that
\begin{equation}\label{hd:eq:zeta-mu}
  \mathcal L_\mu\zeta_\mu=1,\qquad \zeta_\mu>0\quad\text{on }\mathbb S^{n-1}.
\end{equation}
In particular, $m_\mu:=\min_{\mathbb S^{n-1}}\zeta_\mu>0$.
\end{lemma}

\begin{proof}
Existence and uniqueness follow from Theorem \ref{hd:thm:extended-indicial-gap}.  We prove positivity.

First suppose $0<\mu<\alpha$.  Then the zeroth-order coefficient of $\mathcal L_\mu$ is
\[
  c_\mu=N:H_\mu>0
\]
by Lemma \ref{hd:lem:cmu-positive}.  If $\zeta_\mu$ had a nonpositive minimum, then at a minimum point the first-order term would vanish and the second-order term would be nonpositive in the convention of \eqref{hd:eq:Lmu-local}; hence $\mathcal L_\mu\zeta_\mu\le0$, contradicting \eqref{hd:eq:zeta-mu}.  Thus $\zeta_\mu>0$.

Now assume $\alpha\le\mu<n-6$ and set
\[
  \nu:=n-6-\mu.
\]
Then $0<\nu\le\alpha-2<\alpha$.  The inverse $\mathcal L_\nu^{-1}$ preserves nonnegativity: if $h\ge0$ and $u=\mathcal L_\nu^{-1}h$ had a negative minimum, the positive zeroth-order coefficient of $\mathcal L_\nu$ would give $\mathcal L_\nu u<0$ there, a contradiction.  To avoid any ambiguity between the Banach adjoint on H\"older spaces and the formal $L^2$ adjoint, we use the Green pairing directly.  For arbitrary smooth $h\ge0$, let $u:=\mathcal L_\nu^{-1}h\ge0$.  Proposition \ref{hd:prop:indicial-adjoint} gives $\mathcal L_\mu^*=\mathcal L_\nu$, and therefore
\[
  \int_{\mathbb S^{n-1}}\zeta_\mu h\,d\theta
  =\int_{\mathbb S^{n-1}}\zeta_\mu\,\mathcal L_\nu u\,d\theta
  =\int_{\mathbb S^{n-1}}(\mathcal L_\mu\zeta_\mu)u\,d\theta
  =\int_{\mathbb S^{n-1}}u\,d\theta\ge0.
\]
Hence $\zeta_\mu\ge0$.  If $\zeta_\mu(\theta_0)=0$ somewhere, then $\theta_0$ is a minimum, so $\nabla\zeta_\mu(\theta_0)=0$ and $\nabla^2\zeta_\mu(\theta_0)\ge0$.  Since the zeroth-order term vanishes at that point,
\[
  \mathcal L_\mu\zeta_\mu(\theta_0)
  =-N^{ab}\nabla_{ab}\zeta_\mu(\theta_0)\le0,
\]
again contradicting \eqref{hd:eq:zeta-mu}.  Thus $\zeta_\mu>0$ also in this range.  Compactness of the sphere gives $m_\mu>0$.
\end{proof}

\subsection{Comparison-based remainder improvement}

The positivity of $\zeta_\mu$ converts the extended indicial invertibility into strict nonlinear barriers.  The important new feature is that, once one already has any decay exponent strictly larger than $2$, the normalized matrix amplitude required to dominate the inner boundary tends to zero.  Thus only the linearized transverse sign is needed; no global cone property of $\mathcal K_\mu[\zeta_\mu]$ is required.

\begin{lemma}\label{hd:lem:anisotropic-improvement}
Assume the nonflat first-cell branch.  Let $P_N$ be the formal parametrix of order $N$, where
\[
  2\le N<n-6,
\]

\begin{equation}\label{hd:eq:anisotropic-residual}
  M_N:=\mathfrak M[P_N]=\mathcal S_4+O(R^{-1}),
  \qquad
  \sigma_2(M_N)=O(R^{-(N-1)}).
\end{equation}
Suppose that, for some exponent
\begin{equation}\label{hd:eq:gamma-assumption}
  2<\gamma<N+1
\end{equation}
one already has
\begin{equation}\label{hd:eq:gamma-assumption-estimate}
  |W-P_N|\le C_\gamma R^{-\gamma}
\end{equation}
on the far end.  Then, for every
\begin{equation}\label{hd:eq:anisotropic-mu-range}
  \gamma<\mu<\min\{N+1,n-6\}
\end{equation}
the remainder improves to
\begin{equation}\label{hd:eq:anisotropic-improved-estimate}
  |W-P_N|\le C_\mu R^{-\mu}.
\end{equation}
\end{lemma}

\begin{proof}
Let $\zeta_\mu>0$ solve \eqref{hd:eq:zeta-mu}.  For a large inner radius
$R_0$, choose $K=AR_0^{\mu-\gamma}$ with $A$ large enough that
$P_N\pm KR^{-\mu}\zeta_\mu$ bracket $W$ on $r=R_0$.  Since
$\mu>\gamma>2$, the normalized perturbation
$\delta_R:=KR^{2-\mu}$ satisfies
$\sup_{R\ge R_0}\delta_R=O(R_0^{2-\gamma})=o(1)$.  Expanding at $P_N$ gives
\[
 \mathfrak M[P_N\mp KR^{-\mu}\zeta_\mu]
 =M_N\pm\delta_R\mathcal K_\mu[\zeta_\mu]+O(\delta_RR^{-1}+\delta_R^2),
\]
and hence, by $T_1(\mathcal S_4):\mathcal K_\mu[\zeta_\mu]=1$,
\[
 \sigma_2=\sigma_2(M_N)\pm\delta_R
 +O(\delta_RR^{-1}+\delta_R^2).
\]
Because $\mu<N+1$, the residual $O(R^{-(N-1)})$ is $o(\delta_R)$ uniformly
for $R\ge R_0$.  Thus the two profiles have opposite strict cone signs for
large $R_0$, while the lower one remains in the admissible component by
$N\ge c_0I$ and smallness of $\delta_R$.  Exterior comparison yields
$|W-P_N|\le CKR^{-\mu}$, proving \eqref{hd:eq:anisotropic-improved-estimate}.
\end{proof}

\subsection{Recursive construction of the local parametrix}

Put $t=R^{-1}$.  Lemma~\ref{hd:lem:finite-background-expansion}
provides integer expansions of the background metric and Schouten tensor.
For an ansatz $W_{\rm par}=1-\sum_{k=2}^N t^k\phi_k$, the corresponding
$t^{-4}$-scaled Schouten matrix therefore has the form
\[
 t^{-4}\mathcal A_{g_\infty}[W_{\rm par}]
 =\mathcal S_4+t\mathcal S_5+t^2\mathcal S_6+\cdots.
\]
The quadratic identity
$\sigma_2(S+E)=\sigma_2(S)+T_1(S):E+\sigma_2(E)$ shows that, at the first
order where a new $\phi_j$ occurs, it enters linearly through the fixed
Newton tensor $T_1(\mathcal S_4)$.  This is the triangular mechanism behind
the recursion below.

For a positive profile $Z=Z(R,\theta)$ define the scaled Schouten matrix
\begin{equation}\label{hd:eq:scaled-parametrix-matrix}
  \mathfrak M[Z](R,\theta):=R^4\mathcal A_{g_\infty}[Z](R,\theta).
\end{equation}
For every finite order used below, Lemma \ref{hd:lem:finite-background-expansion} supplies the required integer expansion of $g_\infty$ and $(A_{g_\infty})^\sharp$, with all finitely many differentiated remainder bounds needed in the calculations.

Set
\begin{equation}\label{hd:eq:PN-def}
  P_N(R,\theta)
  :=1-\sum_{k=2}^{N}R^{-k}\phi_k(\theta).
\end{equation}
The first coefficient $\phi_2$ is the nonlinear cell constructed in Proposition \ref{hd:prop:first-cell}.  For $j\ge3$, the new coefficient enters the scaled matrix linearly at its first possible order.

\begin{proposition}\label{hd:prop:formal-finite-parametrix}
Assume $C_0\not\equiv0$.  If $N\ge2$ is an integer with $N<n-6$, then there are unique smooth coefficients
\[
  \phi_2,\phi_3,\ldots,\phi_N
\]
such that the profile $P_N$ in \eqref{hd:eq:PN-def} satisfies
\begin{equation}\label{hd:eq:PN-residual}
  \sigma_2\bigl(\mathfrak M[P_N]\bigr)=O(R^{-(N-1)}).
\end{equation}
The coefficients obey the triangular recursion
\begin{equation}\label{hd:eq:formal-recursion-final}
  \mathcal L_j\phi_j=F_j,\qquad 3\le j\le N.
\end{equation}
Here $F_j$ is the degree-$j$ forcing determined by the background jet and
the previously constructed coefficients $\phi_2,\ldots,\phi_{j-1}$.
In particular, every coefficient with $j\ge3$ is uniquely determined by the
uniformly elliptic equation \eqref{hd:eq:formal-recursion-final}.
\end{proposition}

\begin{proof}
For $N=2$, the cell equation and the smooth background expansion give
$\sigma_2(\mathfrak M[P_2])=O(R^{-1})$.

Assume $P_{j-1}$ has been constructed and
\[
  \sigma_2\bigl(\mathfrak M[P_{j-1}]\bigr)
  =O(R^{-(j-2)}).
\]
The integer expansion of the background and of $P_{j-1}$ gives
\[
  \sigma_2\bigl(\mathfrak M[P_{j-1}]\bigr)
  =R^{-(j-2)}G_j(\theta)+O(R^{-(j-1)})
\]
for a smooth $G_j$.  Since $j<n-6$,
Theorem~\ref{hd:thm:extended-indicial-gap} implies that $\mathcal L_j$ is an
isomorphism.  Let $\phi_j$ be the unique smooth solution of
$\mathcal L_j\phi_j=-G_j$.
Equivalently, set $F_j:=-G_j$ in \eqref{hd:eq:formal-recursion-final}.  The quadratic identity for $\sigma_2$ and
$R^4D^2(R^{-j}\psi)=R^{-(j-2)}\mathcal K_j[\psi]$ show directly that this choice improves the residual to
\[
  \sigma_2\bigl(\mathfrak M[P_j]\bigr)=O(R^{-(j-1)}).
\]
Induction proves the proposition.  Uniqueness of every coefficient follows from the invertibility of $\mathcal L_j$.
\end{proof}

\begin{theorem}\label{hd:thm:actual-to-n6}
Assume $n>8$, $C_0\not\equiv0$, the decay \eqref{hd:eq:subcritical}, and the admissible strict comparison principle.  Then there are unique smooth spherical coefficients
\[
  \phi_2,\phi_3,\ldots,\phi_{n-7}
\]
such that, for every
\begin{equation}\label{hd:eq:actual-to-n6-mu}
  n-7<\mu<n-6
\end{equation}
one has

\begin{equation}\label{hd:eq:actual-to-n6-W}
  W(R,\theta)=1-\sum_{j=2}^{n-7}R^{-j}\phi_j(\theta)+O(R^{-\mu})
\end{equation}
uniformly in $\theta$.  Equivalently, there are uniquely determined smooth coefficients $\Phi_j$ such that
\begin{equation}\label{hd:eq:actual-to-n6-U}
  U_{\rm cn}(\rho,\theta)=\rho^{-\alpha}\left(1+\sum_{j=2}^{n-7}\rho^j\Phi_j(\theta)+O(\rho^\mu)\right),\qquad n-7<\mu<n-6.
\end{equation}
The actual Green expansion therefore crosses the bounded relative scale $\mu=\alpha$ and reaches every order below the first positive indicial root $n-6$.
\end{theorem}

\begin{proof}
Proposition~\ref{hd:prop:formal-finite-parametrix} constructs, uniquely, the
formal coefficients $\phi_2,\ldots,\phi_{n-7}$ because every integer degree
in this list lies strictly below the first positive indicial root $m=n-6$.
We show inductively that the true Green end follows these coefficients.

Theorem~\ref{hd:thm:first-tangent} supplies the base step: for some
\[
  2<\gamma_2<\min\{3,\alpha\}
\]
one has $W-P_2=O(R^{-\gamma_2})$.  Lemma~\ref{hd:lem:anisotropic-improvement}
with $N=2$ then improves this estimate to every exponent
\[
  \gamma_2<\mu<\min\{3,m\}.
\]
This already proves the theorem when $m=3$.

For the induction step, suppose that $2\le j\le m-2$ and that, for some
\[
  j<\gamma_j<j+1,
  \qquad W-P_j=O(R^{-\gamma_j}),
\]
the coefficients through degree $j$ have been identified.  The formal
recursion uniquely determines $\phi_{j+1}$ from
$\mathcal L_{j+1}\phi_{j+1}=-G_{j+1}$.  Since
\[
  P_{j+1}-P_j=-R^{-(j+1)}\phi_{j+1}
  \quad\hbox{and}\quad \gamma_j<j+1,
\]
the old estimate also gives $W-P_{j+1}=O(R^{-\gamma_j})$.  Applying
Lemma~\ref{hd:lem:anisotropic-improvement} to $P_{j+1}$, now with
$N=j+1$, yields
\[
  W-P_{j+1}=O(R^{-\gamma_{j+1}})
  \qquad\text{for every}\qquad
  j+1<\gamma_{j+1}<\min\{j+2,m\}.
\]
Thus the induction advances by one integer degree.  At the last step, for
any prescribed $m-1<\mu<m$, choose the output exponent to be $\mu$; this
gives
\[
  W-P_{m-1}=O(R^{-\mu}),
  \qquad P_{m-1}=P_{n-7},
\]
which is \eqref{hd:eq:actual-to-n6-W}.  Finally, the Taylor expansion of
$F=W^{-\alpha/2}$ expresses each $\Phi_j$ uniquely in terms of
$\phi_2,\ldots,\phi_j$ and gives \eqref{hd:eq:actual-to-n6-U} after using
$U_{\rm cn}=\rho^{-\alpha}F$.

\end{proof}

\section{The first resonance and the global coefficient}
\label{sec:high-first-resonance}

Set $m=n-6$.  The preceding section shows that $m$ is the first positive real
indicial root.  The local recursion reaches this scale with a uniquely
determined forcing; the remaining coefficient is detected by a whole-end
relative Newton flux.
By adjoint symmetry,
$\ker\mathcal L_m^*=\ker\mathcal L_0=\operatorname{span}\{1\}$; hence the
critical compatibility condition is precisely the spherical mean.

\begin{lemma}\label{hd:lem:positive-critical-kernel}
There is a unique normalized function $\psi_*>0$ satisfying
\begin{equation}\label{hd:eq:critical-kernel-positive}
  \mathcal L_m\psi_*=0,\qquad \int_{\mathbb S^{n-1}}\psi_*\,d\theta=1.
\end{equation}

\end{lemma}

\begin{proof}
Choose $\mu_j\uparrow m$.  Lemma~\ref{hd:lem:positive-anisotropic-weight}
gives the unique positive solution $\zeta_{\mu_j}$ of
$\mathcal L_{\mu_j}\zeta_{\mu_j}=1$.  Write

\[
  a_\mu:=\int_{\mathbb S^{n-1}}\zeta_\mu\,d\theta.
\]

We first prove that $a_{\mu_j}\to\infty$.  Otherwise, after passing to a
subsequence, $a_{\mu_j}\le C$.  The coefficients of $\mathcal L_{\mu_j}$
remain uniformly elliptic with uniformly bounded lower-order terms as
$\mu_j\uparrow m$.  The global inhomogeneous Harnack estimate on the compact
sphere therefore gives
\[
  \|\zeta_{\mu_j}\|_{L^\infty}
  \le C\bigl(\|\zeta_{\mu_j}\|_{L^1}+1\bigr)\le C,
\]
and Schauder estimates give a uniform $C^{2,\beta}$ bound.  Passing to a
subsequence yields $\zeta_{\mu_j}\to\zeta$ in $C^2$ and
$\mathcal L_m\zeta=1$.  But the adjoint identity
$\mathcal L_m^*=\mathcal L_0$ and $\mathcal L_0 1=0$ imply
\[
  |\mathbb S^{n-1}|
  =\int_{\mathbb S^{n-1}}\mathcal L_m\zeta\,d\theta
  =\int_{\mathbb S^{n-1}}\zeta\,\mathcal L_m^*1\,d\theta=0,
\]
a contradiction.

Normalize by
\[
  \psi_j:=a_{\mu_j}^{-1}\zeta_{\mu_j},
  \qquad \int_{\mathbb S^{n-1}}\psi_j\,d\theta=1,
  \qquad \mathcal L_{\mu_j}\psi_j=a_{\mu_j}^{-1}\longrightarrow0.
\]
The same Harnack and Schauder estimates give a subsequence converging in
$C^2$ to a nonnegative function $\psi_*$ with unit integral and
$\mathcal L_m\psi_*=0$.  Harnack's inequality makes this nonzero
nonnegative kernel function strictly positive.  Finally,
Proposition~\ref{hd:prop:first-positive-indicial-root} shows that the kernel
is one-dimensional, so the unit-integral normalization makes $\psi_*$ unique
and removes the need to pass to a subsequence.

\end{proof}

To compute the derivative of the indicial pencil, decompose the angular
Newton tensor $N=T_1(\mathcal S_4)$ in the polar frame by
\[
  a(\theta):=N^{RR},\qquad
  V^a(\theta):=N^{Ra},\qquad
  N_T:=(N^{ab}).
\]
Its divergence-free identity gives the radial relation

\begin{equation}\label{hd:eq:divfree-radial-component}
  \operatorname{div}_S V+(n-5)a-\operatorname{tr}_S N=0.
\end{equation}
Using this relation, differentiation of the pencil gives
\begin{equation}\label{hd:eq:Lmu-prime}
  \mathcal L_\mu'\varphi
  =2V\cdot\nabla_S\varphi
   +\bigl(\operatorname{tr}_S N-(2\mu+1)a\bigr)\varphi.
\end{equation}

\subsection{Critical forcing and the logarithmic correction}

\begin{equation}\label{hd:eq:critical-forcing-Gm}
  \sigma_2\bigl(\mathfrak M[P_{m-1}]\bigr)
  =R^{-(m-2)}\mathcal G_m(\theta)+O(R^{-(m-1)}).
\end{equation}

If one attempted to continue the recursion at degree $m$ using only a
homogeneous term $-R^{-m}\phi_m$, the Fredholm compatibility condition would
be
\begin{equation}\label{hd:eq:n6-compatibility}
  \int_{\mathbb S^{n-1}}\mathcal G_m\,d\theta=0.
\end{equation}
This condition need not hold.  The missing cokernel component is removed by a
single logarithmic correction.  Set
\begin{equation}\label{hd:eq:n6-transversality}
  \tau_*:=
  \left\langle
    1,
    \left.\partial_\mu\mathcal L_\mu\right|_{\mu=m}\psi_*
  \right\rangle_{L^2(\mathbb S^{n-1})}.
\end{equation}
Proposition~\ref{hd:prop:n6-transversality} below proves that $\tau_*<0$.

\begin{proposition}\label{hd:prop:formal-single-log}
Define
\begin{equation}\label{hd:eq:alog-formula}
  a_{\log}:=-\frac{\displaystyle\int_{\mathbb S^{n-1}}\mathcal G_m\,d\theta}{\tau_*}.
\end{equation}
Then there is a unique zero-mean $\phi_m\in C^\infty(\mathbb S^{n-1})$
satisfying
\begin{equation}\label{hd:eq:critical-coefficient-equation}
  \int_{\mathbb S^{n-1}}\phi_m\,d\theta=0,
  \qquad
  \mathcal G_m+a_{\log}\mathcal L_m'\psi_*+\mathcal L_m\phi_m=0.
\end{equation}
Set
\begin{equation}\label{hd:eq:critical-log-parametrix}
  P_{\rm res}
  :=P_{m-1}
   +a_{\log}R^{-m}\log R\,\psi_*
   -R^{-m}\phi_m.
\end{equation}
Then
\begin{equation}\label{hd:eq:critical-log-residual}
  \sigma_2\bigl(\mathfrak M[P_{\rm res}]\bigr)
  =O\!\left(R^{-(m-1)}(1+\log^2R)\right).
\end{equation}

\end{proposition}

\begin{proof}
Because $\ker\mathcal L_m^*=\operatorname{span}\{1\}$, Fredholm solvability
of the critical coefficient equation is equivalent to vanishing of its
spherical mean.  The definitions of $a_{\log}$ and $\tau_*$ make the forcing
in \eqref{hd:eq:critical-coefficient-equation} mean zero, and invertibility
on the zero-mean subspace gives the unique $\phi_m$.

Differentiating the homogeneous family $-R^{-\mu}\psi_*$ at $\mu=m$ shows
that the logarithmic correction contributes
$\mathcal L_m'\psi_*-\log R\,\mathcal L_m\psi_*$ at order $R^{-(m-2)}$.
After contraction with $N=T_1(\mathcal S_4)$, the logarithmic part vanishes
because $\mathcal L_m\psi_*=0$, while the remaining part contributes
$a_{\log}\mathcal L_m'\psi_*$ at order $R^{-(m-2)}$.  The term
$-R^{-m}\phi_m$ contributes $\mathcal L_m\phi_m$ at the same order.
Consequently, \eqref{hd:eq:critical-coefficient-equation} cancels the entire
leading forcing \eqref{hd:eq:critical-forcing-Gm}.  Replacing the leading
Newton tensor by the full one costs
$O(R^{-(m-1)}(1+\log R))$.  All quadratic polarizations of the new critical
terms have size at most $O(R^{-2(m-2)}(1+\log^2R))$.  Since the integer
$m=n-6$ satisfies $m\ge3$, both errors are bounded by the right-hand side of
\eqref{hd:eq:critical-log-residual}.  This proves the residual estimate and
also shows that only one
logarithmic power is needed at the first resonance.

\end{proof}

\subsection{Transversality at the simple resonance}

\begin{proposition}\label{hd:prop:n6-transversality}
Let $\psi_*$ be the positive generator from
Lemma~\ref{hd:lem:positive-critical-kernel}, and let $\tau_*$ be defined by
\eqref{hd:eq:n6-transversality}.  Then
\begin{equation}\label{hd:eq:n6-transversality-negative}
  \tau_*<0.
\end{equation}

\end{proposition}

\begin{proof}
The adjoint-reflection identity and the divergence-free tensor $N_\infty$
identify $\tau_*$ with the conormal flux of the positive homogeneous solution
$u_*=r^{-m}\psi_*$.  Indeed,

\[
\begin{aligned}
  \tau_*
  &=\int_{\mathbb S^{n-1}}
  \Bigl(2V\cdot\nabla\psi_*
  +(\operatorname{div}V-ma)\psi_*\Bigr)\,d\theta\\
  &=\int_{\mathbb S^{n-1}}
  \bigl(V\cdot\nabla\psi_*-ma\psi_*\bigr)\,d\theta.
\end{aligned}
\]
Set $u_*(r,\theta):=r^{-m}\psi_*(\theta)>0$.  The indicial equation gives
$\mathscr L_\infty u_*=\operatorname{div}(N_\infty\nabla u_*)=0$, and a
direct polar evaluation identifies the preceding integral with its sphere
flux:

\begin{equation}\label{hd:eq:tau-sphere-flux}
  \int_{S_r}N_\infty\nabla u_*\cdot\nu\,dS
  =\tau_*.
\end{equation}

On a regular level set $\Sigma_t=\{u_*=t\}$, choose the normal pointing
toward decreasing $u_*$; then
\[
 N_\infty\nabla u_*\cdot\nu_{\Sigma_t}
  =-\frac{\langle N_\infty\nabla u_*,\nabla u_*\rangle}{|\nabla u_*|}<0.
\]
The divergence theorem
shows that this level-set flux equals the sphere flux
\eqref{hd:eq:tau-sphere-flux}; hence $\tau_*<0$.

\end{proof}

\subsection{Post-resonance \texorpdfstring{$O(R^{-m})$}{O(R to the minus m)} trapping}

Put $\bar\tau_*:=\tau_*/|\mathbb S^{n-1}|<0$.  By the definition of
$\tau_*$, the function $\mathcal L_m'\psi_*-\bar\tau_*$ has zero spherical
mean.  We therefore fix the unique zero-mean solution $\chi$ of

\begin{equation}\label{hd:eq:critical-chi}
  \mathcal L_m\chi=\mathcal L_m'\psi_*-\bar\tau_*.
\end{equation}

\begin{lemma}\label{hd:lem:critical-cylinder-modulation}
Let $s=\log R$.  For every smooth scalar function $f=f(s)$ and spherical function $\eta$, one has the exact identity
\begin{equation}\label{hd:eq:cylinder-pencil-identity}
  R^{m+6}\mathscr L_\infty
  \bigl(R^{-m}f(s)\eta(\theta)\bigr)
  =-f\mathcal L_m\eta
   +f'\mathcal L_m'\eta
   -\frac12f''\mathcal L_m''\eta.
\end{equation}
For a scalar function $b=b(s)$, define
\begin{equation}\label{hd:eq:Qb-def}
  Q_b(R,\theta)
  :=R^{-m}\bigl(b(s)\psi_*(\theta)+b'(s)\chi(\theta)\bigr).
\end{equation}
Then
\begin{equation}\label{hd:eq:Qb-linear-sign}
  R^{m+6}\mathscr L_\infty Q_b
  =\bar\tau_* b'
   +O\bigl(|b''|+|b'''|\bigr).
\end{equation}

\end{lemma}

\begin{proof}
Substituting $R^{-m}f(s)\eta(\theta)$ into the homogeneous divergence-form
operator and differentiating with respect to $s=\log R$ gives
\eqref{hd:eq:cylinder-pencil-identity}.  Apply this identity to the two terms
of $Q_b$, then use $\mathcal L_m\psi_*=0$ and
$\mathcal L_m\chi=\mathcal L_m'\psi_*-\bar\tau_*$; collecting the
$b'$, $b''$, and $b'''$ terms gives \eqref{hd:eq:Qb-linear-sign}.

\end{proof}

\begin{lemma}\label{hd:lem:critical-modulation-error}
Let $P_{\rm res}$ be given by \eqref{hd:eq:critical-log-parametrix}, let $Q_b$ be given by \eqref{hd:eq:Qb-def}, and write
\[
  B(s):=|b|+|b'|+|b''|+|b'''|.
\]
If $R$ is sufficiently large and $R^{-(m-2)}B(s)\le1$, then
\begin{align}
\sigma_2\bigl(\mathfrak M[P_{\rm res}+Q_b]\bigr)
&=\sigma_2\bigl(\mathfrak M[P_{\rm res}]\bigr)
 +R^{-(m-2)}\Bigl(\bar\tau_*b'
 +O(|b''|+|b'''|)\Bigr)\notag\\
&\quad+\mathcal E_b,
\label{hd:eq:critical-modulation-sigma}
\end{align}
where
\begin{equation}\label{hd:eq:critical-modulation-error}
|\mathcal E_b|
\le C R^{-(m-1)}B(s)
 +C R^{-2(m-2)}\bigl(sB(s)+B(s)^2\bigr).
\end{equation}
\end{lemma}

\begin{proof}
Expand the full scaled Schouten matrix at $P_{\rm res}$.  Lemma~\ref{hd:lem:critical-cylinder-modulation} controls the Euclidean Hessian linearization of $Q_b$ exactly.  Replacing $T_1(\mathfrak M[P_{\rm res}])$ by $N=T_1(\mathcal S_4)$ incurs an error $O(R^{-1}+R^{-(m-2)}s)$; paired with the matrix size $O(R^{-(m-2)}B)$ of $Q_b$, this produces
\[
  O(R^{-(m-1)}B)
  +O(R^{-2(m-2)}sB).
\]
The metric and connection terms, the $W-1$ correction, and the gradient-square term in the full conformal Schouten formula contain at least one additional factor $R^{-1}$ or one small matrix factor, and are therefore controlled by the same two terms.  The quadratic contribution of $Q_b$ is $O(R^{-2(m-2)}B^2)$.  Combining the estimates gives \eqref{hd:eq:critical-modulation-error}.
\end{proof}

\begin{theorem}\label{hd:thm:actual-leading-resonance}
Under the assumptions of Theorem~\ref{hd:thm:actual-to-n6}, put $m=n-6$ and let $P_{\rm res}$, $a_{\log}$, and $\psi_*$ be defined by \eqref{hd:eq:critical-log-parametrix}, \eqref{hd:eq:alog-formula}, and Lemma~\ref{hd:lem:positive-critical-kernel}.  Then
\begin{align}
  W(R,\theta)&=P_{\rm res}(R,\theta)+O(R^{-m}),
  \label{hd:eq:actual-bounded-trapping}\\
  \frac{R^m}{\log R}
  \bigl(W(R,\theta)-P_{m-1}(R,\theta)\bigr)
  &\longrightarrow a_{\log}\psi_*(\theta)
  \label{hd:eq:actual-leading-log-limit}
\end{align}
uniformly in $\theta$.

\end{theorem}

\begin{proof}
Choose $2<\gamma<m$.  Theorem~\ref{hd:thm:actual-to-n6} and the formal
critical correction give a coarse estimate for the difference from the
resonant parametrix, namely $W-P_{\rm res}=O(R^{-\gamma})$.

Choose $0<\sigma<\min\{1,m-2\}$ so small that
$\psi_*+\sigma\chi\ge c_*>0$ and the
$O(|b''|+|b'''|)$ term in \eqref{hd:eq:Qb-linear-sign} is at most
$|\bar\tau_*|b'/2$ for the function below.  With $R_0=e^{s_0}$, set
\[
 K:=A R_0^{m-\gamma},\qquad
 b(s):=K\bigl(2-e^{-\sigma(s-s_0)}\bigr),\qquad s\ge s_0.
\]
Then $K\le b\le2K$, $b'=K\sigma e^{-\sigma(s-s_0)}$,
$b''=-\sigma b'$, and $b'''=\sigma^2b'$.  At the inner boundary,
$Q_b(R_0,\theta)\ge c_*A R_0^{-\gamma}$, so a fixed large $A$ gives the
required ordering.

It remains to compare the nonlinear error with the sign-producing term
$R^{-(m-2)}b'$.  Writing $t=s-s_0$, the bound
\eqref{hd:eq:critical-modulation-error}, together with
$\gamma>2$ and $\sigma<\min\{1,m-2\}$, gives uniformly for $t\ge0$
\[
 R^{-(m-1)}(1+s^2)+|\mathcal E_b|
 =o_{R_0\to\infty}\bigl(R^{-(m-2)}b'\bigr).
\]
Indeed, after division by the right-hand side, the only exponential rates are
$1-\sigma$ and $m-2-\sigma$, while the only nonnegative power of $R_0$ is
multiplied by $R_0^{2-\gamma}$.  Lemma
\ref{hd:lem:critical-modulation-error} therefore makes $P_{\rm res}-Q_b$ strictly
admissible and $P_{\rm res}+Q_b$ strictly exterior on the entire end.
The outer-boundary compensation and comparison yield
\[
 P_{\rm res}-Q_b\le W\le P_{\rm res}+Q_b,\qquad
 |Q_b(R,\theta)|\le CKR^{-m}\qquad(R\ge R_0).
\]
This proves \eqref{hd:eq:actual-bounded-trapping}.  Dividing
$W-P_{m-1}$ by $R^{-m}\log R$ kills the bounded $R^{-m}$ remainder and leaves
$a_{\log}\psi_*$, which proves \eqref{hd:eq:actual-leading-log-limit}.

\end{proof}

\subsection{Relative Newton flux and its limit}

\label{hd:subsec:post-resonance-flux}

Write $r=R$, $q=\log W$, $q_0=\log P_{\rm res}$, and $h=q-q_0$.
Passing to logarithmic conformal factors isolates the additive perturbation:
\[
 \mathscr A[q]
 :=(A_{g_\infty})^\sharp+(\Hess_{g_\infty}q)^\sharp
   +\nabla q\otimes\nabla q-\frac12|\nabla q|^2I,
 \qquad
 \mathcal A_{g_\infty}[e^q]=e^{2q}\mathscr A[q].
\]
Thus $\mathscr A[q]\in\partial\Gamma_2$ and
$\sigma_2(\mathscr A[q])=0$, whereas the resonant parametrix satisfies

\begin{equation}\label{hd:eq:q0-physical-residual}
 \sigma_2(\mathscr A[q_0])
 =O\!\left(r^{-(m+7)}(1+\log^2r)\right).
\end{equation}

\begin{lemma}
\label{hd:lem:relative-newton-flux}
Under the assumptions of Theorem~\ref{hd:thm:actual-leading-resonance}, take $r$ sufficiently large and set
\[
 q_t:=q_0+t h,\qquad
 A_t:=\mathscr A[q_t],\qquad
 T_t:=T_1(A_t),\qquad 0\le t\le1.
\]
Define the path-averaged Newton tensor and the relative Newton vector field by
\begin{equation}\label{hd:eq:path-newton-current}
 \mathbf T:=\int_0^1T_t\,dt,
 \qquad
 J_h:=\mathbf T\nabla h.
\end{equation}
In the formulas below, $d:=\nabla h$.

\smallskip
\noindent
\textup{(i) Exact path identity and path positivity.}

\begin{equation}\label{hd:eq:exact-relative-current}
 \operatorname{div}_{g_\infty}J_h
 =-\sigma_2(A_0)
 +(n-4)\int_0^1T_t(\nabla q_t,\nabla h)\,dt.
\end{equation}
Moreover,
\begin{equation}\label{hd:eq:path-newton-positive}
 T_t=(1-t)T_0+tT_1
 +t(1-t)\left(
 d\otimes d+\frac{n-3}{2}|d|^2I\right).
\end{equation}
Consequently,
\begin{equation}\label{hd:eq:path-coercivity}
 T_t\ge c(1-t)r^{-4}I,
 \qquad
 \mathbf T\ge c r^{-4}I
 \qquad(r\ge R_0).
\end{equation}

\smallskip
\noindent
\textup{(ii) Critical Caccioppoli.}

\begin{equation}\label{hd:eq:Mh-def}
 A_R^*:=\{R/2<r<4R\},
 \qquad
 M_h(R):=\sup_{A_R^*} r^m|h|.
\end{equation}

\begin{equation}\label{hd:eq:critical-path-caccioppoli}
 \int_{\{R<r<2R\}}
 \mathbf T(\nabla h,\nabla h)\,dV_{g_\infty}
 \le C R^{-m}\left[
 M_h(R)^2+M_h(R)R^{-1}(1+\log^2R)
 \right].
\end{equation}

\smallskip
\noindent
\textup{(iii) Canonical averaged relative flux.}

Fix
\[
 \eta\in C_c^\infty((1,2)),\qquad \eta\ge0,\qquad
 \int_1^2\eta(s)\,ds=1,
\]
and define
\begin{equation}\label{hd:eq:relative-flux-def}
 \mathfrak F_\eta(R)
 :=\int_{A_{R,2R}}
 \frac1R\eta\!\left(\frac rR\right)
 \left\langle J_h,\nabla^{g_\infty}r\right\rangle_{g_\infty}
 dV_{g_\infty}.
\end{equation}
Uniformly for $1\le\lambda\le2$,
\begin{equation}\label{hd:eq:relative-flux-drift}
 \begin{aligned}
 |\mathfrak F_\eta(\lambda R)-\mathfrak F_\eta(R)|
 \le C\big[&R^{-1}(1+\log^2R)
 +R^{-2}+R^{-m}\big].
 \end{aligned}
\end{equation}
In particular, the limit
\begin{equation}\label{hd:eq:relative-flux-limit}
 \mathfrak F_\infty:=\lim_{R\to\infty}\mathfrak F_\eta(R).
\end{equation}
exists and is independent of the normalized cutoff.  Moreover, if
\begin{equation}\label{hd:eq:if-pure-kernel-asymptotic}
 r^m h(r,\theta)\longrightarrow B\psi_*(\theta)
\end{equation}
uniformly in $\theta$, then

\begin{equation}\label{hd:eq:flux-detects-B}
 \mathfrak F_\infty=\tau_*B.
\end{equation}

\end{lemma}

\begin{proof}
We first establish the identity in the smooth case.  The conformal divergence
formula is
\begin{equation}\label{hd:eq:newton-div-conformal}
 \operatorname{div}_{g_\infty}T_1(\mathscr A[q])
 =(n-2)T_1(\mathscr A[q])(\nabla q,\cdot)
 -(\tr T_1(\mathscr A[q]))dq.
\end{equation}
Since
\[
 \partial_tA_t
 =\Hess h+\nabla q_t\otimes\nabla h
 +\nabla h\otimes\nabla q_t
 -\langle\nabla q_t,\nabla h\rangle I,
\]
the product rule and \eqref{hd:eq:newton-div-conformal} give
\[
 \partial_t\sigma_2(A_t)
 =\operatorname{div}_{g_\infty}(T_t\nabla h)
 -(n-4)T_t(\nabla q_t,\nabla h).
\]
Integrating in $t$ and using $\sigma_2(A_1)=0$ proves
\eqref{hd:eq:exact-relative-current}.

For the $C^{1,1}$ Green end, mollify $q$ and $q_0$ on a slightly larger
compact annulus.  Their mollifications converge strongly in
$W^{2,p}_{\rm loc}$ for every finite $p$.  The tensors $A_t$, their Newton
tensors, and all products in the preceding identity are polynomial in first
and second derivatives, so the smooth identities pass to the limit in
distributions.  The viscosity equation holds almost everywhere for a
$C^{1,1}$ solution and hence supplies the endpoint relation
$\sigma_2(A_1)=0$.  This also explains precisely where the smooth strict
approximation of Proposition~\ref{prop:gamma2-green-theory} may be used to retain the admissible branch during the
passage to the limit.

Let $Q(d)=d\otimes d-\frac12|d|^2I$.  Direct expansion gives
\[
 A_t=(1-t)A_0+tA_1-t(1-t)Q(d),
 \qquad
 T_1(Q(d))=-d\otimes d-\frac{n-3}{2}|d|^2I.
\]
This proves \eqref{hd:eq:path-newton-positive}.  Moreover,
\[
 r^4P_{\rm res}^2A_0=\mathfrak M[P_{\rm res}]
 =\mathcal S_4+o(1),
 \qquad T_0=r^{-4}(N+o(1)),\qquad N\ge c_0I.
\]
Because $T_1\ge0$ on the admissible endpoint, integration in $t$ gives
\eqref{hd:eq:path-coercivity}.  Integrating
\eqref{hd:eq:path-newton-positive} also gives the exact split
\begin{equation}\label{hd:eq:path-average-split}
 \mathbf T=\frac12(T_0+T_1)
 +\frac16\left(d\otimes d+\frac{n-3}{2}|d|^2I\right).
\end{equation}

We use the annular estimate
\begin{equation}\label{hd:eq:path-newton-L1}
 \int_{A_R^*}\tr_{g_\infty}(T_0+T_1)\,dV_{g_\infty}
 \le C R^{n-4}=C R^{m+2}.
\end{equation}
Indeed,
$\tr T_1=(n-1)\sigma_1(A_1)$ and
\[
 \sigma_1(A_1)
 =\sigma_1((A_{g_\infty})^\sharp)+\Delta_{g_\infty}q
 -\frac{n-2}{2}|\nabla q|^2,
\]
testing against an annular cutoff $\chi_R$ that equals one on $A_R^*$ and
satisfies $|\nabla^k\chi_R|\le C_kR^{-k}$ yields
\[
 \int\chi_R\sigma_1(A_1)
 \le C\int\chi_Rr^{-4}
 +\int q\Delta_{g_\infty}\chi_R+O(R^{n-6}).
\]
Here $q=O(r^{-2})$, so the right-hand side is $O(R^{n-4})$.
The same estimate for $T_0$ follows from its pointwise asymptotic.  This proves
\eqref{hd:eq:path-newton-L1}.

For Caccioppoli, choose $\zeta\in C_c^\infty(A_R^*)$ with
$\zeta=1$ on $\{R<r<2R\}$ and $|\nabla\zeta|\le CR^{-1}$, and put
$\chi=\zeta^2$.  Also set $\mathbf T^{(1)}=\int_0^1tT_t\,dt$.  Testing
\eqref{hd:eq:exact-relative-current} by $\chi h$ gives
\begin{align*}
 \int\chi\mathbf T(\nabla h,\nabla h)
 ={}&-\int h\mathbf T(\nabla h,\nabla\chi)
 +\int\chi h\sigma_2(A_0)\\
 &-(n-4)\int\chi h\mathbf T(\nabla q_0,\nabla h)
 -(n-4)\int\chi h\mathbf T^{(1)}(\nabla h,\nabla h).
\end{align*}
The split \eqref{hd:eq:path-average-split} shows that the left-hand side
controls both the endpoint quadratic energy and $\int\chi|\nabla h|^4$.
The endpoint part of the cutoff term is bounded by Cauchy--Schwarz,
\eqref{hd:eq:path-newton-L1}, and $|h|\le CM_h(R)R^{-m}$.  Its gradient-cubic
part is absorbed using
\begin{equation}\label{hd:eq:caccioppoli-quartic-cutoff}
 CR^{-1}|h|\chi^{3/2}|\nabla h|^3
 \le \varepsilon\chi^2|\nabla h|^4
 +C_\varepsilon R^{-4}|h|^4.
\end{equation}
Since $n=m+6$ and $m>2$,
\[
 R^{-4}\int_{A_R^*}|h|^4
 \le CM_h(R)^4R^{-3m+2}
 =o\bigl(M_h(R)^2R^{-m}\bigr).
\]
The terms involving $\mathbf T^{(1)}$ are absorbed in the same way because
$h=O(R^{-m})$.  For the $q_0$ term we use
\[
 \int_{A_R^*}(T_0+T_1)(\nabla q_0,\nabla q_0)\le CR^{m-4}
\]
together with Cauchy--Schwarz and the quartic control.  Finally,
\eqref{hd:eq:q0-physical-residual} gives
\[
 \int_{A_R^*}|\sigma_2(A_0)|\le CR^{-1}(1+\log^2R).
\]
After taking $\varepsilon$ small and absorbing the energy terms, these
estimates give \eqref{hd:eq:critical-path-caccioppoli}.  Theorem
\ref{hd:thm:actual-leading-resonance} gives $M_h(R)\le C$, and hence
\begin{equation}\label{hd:eq:critical-path-caccioppoli-bounded}
 \int_{\{R<r<2R\}}
 \mathbf T(\nabla h,\nabla h)\,dV_{g_\infty}
 \le C R^{-m}+C R^{-m-1}(1+\log^2R).
\end{equation}

For the drift, define
\[
 \omega_{\lambda,R}(r)
 :=\frac1{\lambda R}\eta\!\left(\frac r{\lambda R}\right)
 -\frac1R\eta\!\left(\frac rR\right),
 \qquad
 \varphi_{\lambda,R}(r):=\int_0^r\omega_{\lambda,R}(s)\,ds.
\]
The normalization of $\eta$ makes $\varphi_{\lambda,R}$ compactly supported
in a fixed enlargement of $A_R^*$, with
$\|\varphi_{\lambda,R}\|_\infty\le C$ uniformly for
$1\le\lambda\le2$.  Integration by parts gives
\[
 \mathfrak F_\eta(\lambda R)-\mathfrak F_\eta(R)
 =-\int\varphi_{\lambda,R}\operatorname{div}J_h.
\]
The residual term is $O(R^{-1}(1+\log^2R))$.  Cauchy--Schwarz, the
Caccioppoli estimate, and the preceding $q_0$ energy bound give
\[
 \left|\int\varphi_{\lambda,R}
       \mathbf T(\nabla q_0,\nabla h)\right|\le CR^{-2},
 \qquad
 \left|\int\varphi_{\lambda,R}
       \mathbf T^{(1)}(\nabla h,\nabla h)\right|\le CR^{-m}.
\]
This proves \eqref{hd:eq:relative-flux-drift}.  Summation over dyadic scales,
followed by the uniform $1\le\lambda\le2$ estimate, proves existence of the
limit.  Comparing two normalized cutoffs by the same argument shows that the
limit is cutoff independent.

It remains to justify detection of a pure kernel without assuming pointwise
gradient convergence.  Suppose \eqref{hd:eq:if-pure-kernel-asymptotic} and set
\[
 H_R(x):=R^m h(Rx)
\]
on a fixed annulus.  Then $H_R\to B|x|^{-m}\psi_*(x/|x|)$ uniformly.
After rescaling, coercivity and \eqref{hd:eq:critical-path-caccioppoli-bounded}
give a uniform local $H^1$ bound, so
\[
 \nabla H_R\rightharpoonup
 B\nabla\bigl(|x|^{-m}\psi_*\bigr)
 \quad\text{weakly in }L^2_{\rm loc}.
\]
The exact path polynomial shows that the rescaled Newton field equals
$N_\infty\nabla H_R$ plus terms carrying the small factor
$R^{2-m}$.  The Hessian part of those errors is passed distributionally by
the null-Lagrangian identity
\[
 \bigl(T_1(D^2H_R)\nabla H_R\bigr)_i
 =\partial_k(H_{R,i}H_{R,k})-\partial_i|\nabla H_R|^2,
\]
and the gradient-cubic part tends to zero by the rescaled quartic estimate
contained in the Caccioppoli bound.  Hence the normalized currents converge
in distributions to
$B N_\infty\nabla(|x|^{-m}\psi_*)$.  Passing to the averaged flux gives
\[
 \mathfrak F_\infty
 =B\int_{S_1}N_\infty\nabla(r^{-m}\psi_*)\cdot\nu\,dS
 =B\tau_*,
\]
where the last equality is Proposition~\ref{hd:prop:n6-transversality}.
This proves \eqref{hd:eq:flux-detects-B}.
\end{proof}

\subsection{Critical-cylinder Liouville theorem}

\begin{lemma}
\label{hd:lem:critical-cylinder-liouville}
Let $s=\log r$, $a=N^{RR}$, and $V^a=N^{Ra}$.  For $z=z(s,\theta)$,
define
\[
 \mathscr Pz:=-\mathcal L_m z+\mathcal L_m'\partial_s z
 -\frac12\mathcal L_m''\partial_s^2z,
 \qquad
 \Phi[z](s):=\int_{\mathbb S^{n-1}}
 \left[a(\partial_s z-mz)+V\cdot\nabla_S z\right]d\theta.
\]
Then
\begin{equation}\label{hd:eq:cylinder-conjugation}
 r^{m+6}\mathscr L_\infty(r^{-m}z)=\mathscr Pz.
\end{equation}
If
\[
 z\in L^\infty(\mathbb R\times\mathbb S^{n-1}),
 \qquad
 \mathscr Pz=0
\]
weakly, then there is a unique constant $B$ such that
\[
 z(s,\theta)\equiv B\psi_*(\theta),
 \qquad
 \Phi[z]\equiv\tau_*B.
\]
In particular, $\Phi[z]=0$ implies $z\equiv0$.

\end{lemma}

\begin{proof}
Put $u=r^{-m}z$.  Since $n=m+6$, the homogeneities cancel in the sphere
flux and yield

\[
 \int_{S_r}N_\infty\nabla u\cdot\nu\,dS
 =\Phi[z](s).
\]

\begin{equation}\label{hd:eq:conormal-derivative}
 \frac{d}{ds}\Phi[z](s)
 =\int_{\mathbb S^{n-1}}\mathscr Pz(s,\theta)\,d\theta.
\end{equation}
Thus $\Phi[z]$ is constant for a solution.  Taking $z=B\psi_*$ and using
\eqref{hd:eq:tau-sphere-flux} gives the stated conormal identity.

For the Liouville assertion, use the positive ground state and write
\[
 z=\psi_*v,
 \qquad
 \mathscr Qv:=\psi_*^{-1}\mathscr P(\psi_*v).
\]
The conjugated operator is uniformly elliptic on the cylinder, has smooth
coefficients independent of $s$, and satisfies $\mathscr Q1=0$; in
particular, it has no zeroth-order term.  For real $\xi$, introduce the
operator pencil
\begin{equation}\label{hd:eq:fourier-pencil}
 \mathscr P(\xi):H^2(\mathbb S^{n-1})\longrightarrow L^2(\mathbb S^{n-1}),
 \qquad
 \mathscr P(\xi)\chi:=e^{-i\xi s}\mathscr Q(e^{i\xi s}\chi).
\end{equation}
If $\xi\ne0$ and $\mathscr P(\xi)\chi=0$, the real and imaginary parts of
$e^{i\xi s}\chi$ descend to solutions of $\mathscr Qw=0$ on the compact
periodic cylinder with period $2\pi/|\xi|$.  Each attains its maximum there,
so the strong maximum principle makes it constant.  The nonzero frequency
then forces $\chi=0$.  The pencil has Fredholm index zero and is therefore
invertible for every $\xi\ne0$.  Continuity in $\xi$, compactness, and
elliptic estimates give, for each compact
$I\Subset\mathbb R\setminus\{0\}$,

\[
 \sup_{\xi\in I}\|\mathscr P(\xi)^{-1}\|_{L^2\to H^2}<\infty.
\]
Regard the bounded $v$ as a tempered distribution and Fourier transform in
$s$.  Then $\mathscr P(\xi)\widehat v=0$.  If a smooth test function
$\Phi(\xi)$ is supported in such an interval $I$, set
$\Psi=(\mathscr P(\xi)^*)^{-1}\Phi$.  Smooth dependence of the inverse on
$\xi$ makes $\Psi$ an admissible test function, and
\[
 \langle\widehat v,\Phi\rangle
 =\langle\widehat v,\mathscr P(\xi)^*\Psi\rangle
 =\langle\mathscr P(\xi)\widehat v,\Psi\rangle=0.
\]
Hence $\operatorname{supp}\widehat v\subset\{0\}$.  A tempered distribution
supported at one frequency is a finite sum of derivatives of the Dirac mass;
equivalently, $v$ is a polynomial in $s$ with distributional spherical
coefficients.  Elliptic regularity makes these coefficients smooth, and
boundedness makes $v$ independent of $s$.  The spherical operator
$\mathscr P(0)$ has constants as its kernel by the strong maximum principle.
Thus $v$ is constant and $z=B\psi_*$.  The first part of the proof now gives
$\Phi[z]=\tau_*B$.

\end{proof}

\subsection{Finite-jet ABP refinement}

\label{hd:subsec:finite-jet-locking}

Use the limiting flux to select the critical coefficient and shifted
remainder:
\[
 B_*:=\frac{\mathfrak F_\infty}{\tau_*},\qquad
 \bar q:=q_0+B_*r^{-m}\psi_*,\qquad
 h_*:=q-\bar q.
\]
The resonant trapping theorem gives $|h_*|\le Cr^{-m}$.

For a symmetric matrix $A$, write
\[
 A^\circ:=A-\frac{\tr A}{n}I,\qquad
 \mathfrak f(A):=\frac{\tr A}{n}
 -\frac{|A^\circ|}{\sqrt{n(n-1)}}.
\]
The elementary factorization
\begin{equation}\label{hd:eq:threshold-factorization}
 \sigma_2(A)=\frac{n(n-1)}2
 \left(\frac{\tr A}{n}
 +\frac{|A^\circ|}{\sqrt{n(n-1)}}\right)\mathfrak f(A)
\end{equation}
shows that, on the component $\{\tr A>0\}$,
$\partial\Gamma_2=\{\mathfrak f=0\}$.  This scalarization is globally
degenerate elliptic.  Indeed, if $P\ge0$, then
\[
 |P^\circ|\le\sqrt{\frac{n-1}{n}}\,\tr P,
 \qquad
 0\le\mathfrak f(A+P)-\mathfrak f(A)\le\frac2n\tr P.
\]

To measure the contact jets in fixed Euclidean coordinates, rescale by
\[
 g_R:=R^{-2}D_R^*g_\infty,\qquad
 M^{\sharp,\mathrm{sym}}:=g_R^{-1/2}Mg_R^{-1/2}.
\]
With this convention, define

\begin{equation}\label{hd:eq:scaled-B-operator}
 \bar q_R(x):=\bar q(Rx),\qquad
 \mathcal B_R[u]
 :=R^2\mathscr A_{g_R}\!\left[\bar q_R+R^{-2}u\right],
 \qquad
 B_R:=\mathcal B_R[0].
\end{equation}
For $Q(p):=p\otimes p-\tfrac12|p|^2I$, the dependence on $u$ is exactly
\begin{equation}\label{hd:eq:scaled-B-expansion}
 \mathcal B_R[u]
 =B_R+\mathcal H_R[u]+R^{-2}Q(\nabla u),
\end{equation}
where $\mathcal H_R$ is linear in $(D^2u,Du)$, its principal part is
$(\Hess_{g_R}u)^{\sharp,\mathrm{sym}}$, and its first-order coefficients are
locally uniformly bounded.

For the normalized remainder $u_R(x):=R^2h_*(Rx)$, the Green equation is
simply
\begin{equation}\label{hd:eq:scaled-threshold-equation}
 \mathfrak f(\mathcal B_R[u_R])=0.
\end{equation}

\begin{lemma}\label{hd:lem:threshold-scalarization}
Fix a compact annulus $K'\Subset\mathbb R^n\setminus\{0\}$.  There are
$\rho,\lambda,\Lambda>0$ and $R_0$, depending only on $K'$ and the limiting
background, such that $\mathfrak f$ is $(\lambda,\Lambda)$-uniformly elliptic
on the $\rho$-neighborhood of $\{B_R(x):x\in K'\}$ whenever $R\ge R_0$.
Moreover,
\begin{equation}\label{hd:eq:threshold-background-residual}
 \|\mathfrak f(B_R)\|_{L^\infty(K')}
 \le C_{K'}R^{1-m}(1+\log^2R).
\end{equation}
\end{lemma}

\begin{proof}
On $\{\mathfrak f=0,\ \tr A>0,\ A^\circ\ne0\}$, differentiation gives
\begin{equation}\label{hd:eq:threshold-derivative}
 D\mathfrak f(A)
 =\frac1nI-\frac{A^\circ}{\sqrt{n(n-1)}|A^\circ|}
 =\frac{T_1(A)}{(n-1)\tr A}.
\end{equation}
On $K'$, the matrices $B_R$ converge uniformly to $S_\infty$, and
$T_1(S_\infty)=N_\infty$ is uniformly positive definite.  Compactness and
\eqref{hd:eq:threshold-derivative} therefore give the asserted ellipticity
on a fixed tubular neighborhood.  Notice that this is a local uniform
ellipticity statement; the global monotonicity above is retained outside the
tube.

The physical residual \eqref{hd:eq:q0-physical-residual}, after the $R^4$
matrix normalization, is $O(R^{1-m}(1+\log^2R))$.  Inserting
$B_*r^{-m}\psi_*$ does not create a larger term: the leading linear term
vanishes because
$\mathscr L_\infty(r^{-m}\psi_*)=0$, the coefficient error contributes
$O(R^{1-m}(1+\log R))$, and the quadratic term is
$O(R^{4-2m})=O(R^{1-m})$ since $m\ge3$.  Thus
$|\sigma_2(B_R)|\le C_{K'}R^{1-m}(1+\log^2R)$.  The other factor in
\eqref{hd:eq:threshold-factorization} stays uniformly positive on $K'$,
which proves \eqref{hd:eq:threshold-background-residual}.
\end{proof}

The following finite-jet estimate replaces the need for a globally uniformly
elliptic extension of the equation.  For eigenvalues $\mu_i$ of $M$, let
\[
 \mathcal M^+(M):=\Lambda\sum_{\mu_i>0}\mu_i
 +\lambda\sum_{\mu_i<0}\mu_i,\qquad
 \mathcal M^-(M):=\lambda\sum_{\mu_i>0}\mu_i
 +\Lambda\sum_{\mu_i<0}\mu_i.
\]
Given $J\ge1$, write
$u\in\mathcal S_J^*(\lambda,\Lambda,b,\varepsilon;\Omega)$ if the following
two viscosity inequalities hold whenever a $C^2$ test function $\phi$
touches $u$ at $x\in\Omega$ and
$|D\phi(x)|+\|D^2\phi(x)\|\le J$:
\begin{align}
 \mathcal M^+(D^2\phi(x))+b|D\phi(x)|
 &\ge-\varepsilon
 &&\text{at an upper contact},\label{hd:eq:finite-jet-upper}\\
 \mathcal M^-(D^2\phi(x))-b|D\phi(x)|
 &\le\varepsilon
 &&\text{at a lower contact}.\label{hd:eq:finite-jet-lower}
\end{align}

\begin{lemma}
\label{hd:lem:finite-jet-abp}
Fix $b_0<\infty$.  There are
$J_0,\varepsilon_0>0$, $\alpha\in(0,1)$, and $C<\infty$, depending only on
$(n,\lambda,\Lambda,b_0)$, with the following property.  If
\[
 \|u\|_{L^\infty(B_2)}\le1,\qquad
 u\in\mathcal S_J^*(\lambda,\Lambda,b,\varepsilon;B_2),
 \qquad
 J\ge J_0,\quad b\le b_0,\quad\varepsilon\le\varepsilon_0,
\]
then, for $x_0\in B_{1/2}$ and $0<r\le1/4$,
\begin{equation}\label{hd:eq:finite-jet-oscillation}
 \osc_{B_r(x_0)}u
 \le C\left(r^\alpha+J^{-\alpha/2}\right).
\end{equation}
\end{lemma}

\begin{proof}
We only need to track the openings in the classical ABP weak-Harnack
argument.  At level $M^k$ of its measure iteration, every test paraboloid has
slope and opening bounded by $C_0M^k$, where $M>1$ and $C_0$ depend only on
$(n,\lambda,\Lambda,b_0)$.  Consequently, the standard distribution estimate
and weak Harnack inequality remain valid under
\eqref{hd:eq:finite-jet-upper}--\eqref{hd:eq:finite-jet-lower} for all levels
satisfying $C_0M^k\le J$.

There is no loss from this truncation at unit scale.  Indeed, for a
nonnegative function $w\le1$, put
\[
 a:=\inf_{B_{1/2}}w+\varepsilon/\varepsilon_0.
\]
The function $w/a$ has residual at most $\varepsilon_0$, jet radius $J/a$,
and range bounded by $1/a$.  Since $J\ge J_0$ for a sufficiently large
universal $J_0\ge C_0$, every relevant level satisfies
\[
 M^k\le\frac1a
 \quad\Longrightarrow\quad
 C_0M^k\le\frac{C_0}{a}\le\frac Ja.
\]
Thus the usual weak Harnack argument applies without change.  Applying it to
the two nonnegative oscillation halves gives
\begin{equation}\label{hd:eq:finite-jet-unit-decay}
 \osc_{B_{1/2}}u
 \le\vartheta\osc_{B_1}u+C\varepsilon
\end{equation}
for a universal $\vartheta\in(0,1)$.

It remains to identify the stopping scale.  Set
\[
 \omega(r):=\osc_{B_r(x_0)}u,\qquad
 \Omega_r:=\omega(r)+r^2\varepsilon/\varepsilon_0.
\]
After rescaling $B_r(x_0)$ to $B_1$ and dividing by $\Omega_r$, the residual
is at most $\varepsilon_0$, the drift is $rb$, and the available jet radius
is at least $cJr^2$.  Hence \eqref{hd:eq:finite-jet-unit-decay} yields
\[
 \omega(r/2)\le\vartheta\omega(r)+Cr^2\varepsilon
\]
as long as $Jr^2\ge CJ_0$.  Choose $\alpha\in(0,1)$ with
$\vartheta<2^{-\alpha}$ and iterate down to
$r_*\simeq J^{-1/2}$.  Below $r_*$ use monotonicity of $\omega$.  This gives
\eqref{hd:eq:finite-jet-oscillation}.  The exponent $J^{-\alpha/2}$ is
therefore forced by the second-order scaling of the contact jets.
\end{proof}

\begin{lemma}
\label{hd:lem:finite-jet-refined-estimate}
Let $K\Subset K'\Subset\mathbb R^n\setminus\{0\}$, let $R_j\to\infty$, and
let $a_0\le A_j\le C_0$.  Set
\begin{equation}\label{hd:eq:finite-jet-normalization}
 \delta_j:=A_jR_j^{2-m},\qquad
 v_j(x):=\frac{R_j^m}{A_j}h_*(R_jx).
\end{equation}
If $\|v_j\|_{L^\infty(K')}\le C_{K'}$, then, for all large $j$,
\begin{align}
 \osc_{B_r(x)}v_j
 &\le C_K\left(r^\alpha+\delta_j^{\alpha/2}\right),
 \label{hd:eq:vj-finite-jet-modulus}\\
 \|v_j\|_{L^\infty(K)}
 &\le C_K\|v_j\|_{L^2(K')}^\gamma
 +C_K\delta_j^{\alpha/2},
 \qquad
 \gamma:=\frac{2\alpha}{n+2\alpha},
 \label{hd:eq:vj-L2-Linfty}
\end{align}
whenever $B_{2r}(x)\Subset K'$.
\end{lemma}

\begin{proof}
Since $u_{R_j}=\delta_jv_j$, equations
\eqref{hd:eq:scaled-B-expansion} and
\eqref{hd:eq:scaled-threshold-equation} read
\[
 \mathfrak f\!\left(
 B_{R_j}+\delta_j\mathcal H_{R_j}[v_j]
 +R_j^{-2}\delta_j^2Q(\nabla v_j)\right)=0.
\]
At a contact jet satisfying
$|D\phi|+\|D^2\phi\|\le c\delta_j^{-1}$, the matrix increment stays in the
ellipticity tube of Lemma~\ref{hd:lem:threshold-scalarization}; here $c>0$
is fixed sufficiently small.  The global monotonicity of $\mathfrak f$
supplies the correct viscosity inequalities up to the boundary of that
tube.  Dividing the centered equation by $\delta_j$ therefore gives
\[
 v_j\in
 \mathcal S^*_{J_j}(\lambda,\Lambda,b_0,\eta_j;K'),
 \qquad
 J_j=c\delta_j^{-1},
\]
with
\begin{equation}\label{hd:eq:finite-jet-residual}
 \eta_j\le
 \frac{C_{K'}(1+\log^2R_j)}{A_jR_j}\longrightarrow0.
\end{equation}
After a fixed normalization of the $L^\infty$ bound,
Lemma~\ref{hd:lem:finite-jet-abp} gives
\eqref{hd:eq:vj-finite-jet-modulus}.

Put $X_j:=\|v_j\|_{L^2(K')}$.  For
$c\delta_j^{1/2}\le r\le r_K$ and $x\in K$,
\[
 |v_j(x)|
 \le C r^{-n/2}X_j+Cr^\alpha+C\delta_j^{\alpha/2}.
\]
Choose
\[
 r=\max\left\{c\delta_j^{1/2},
 X_j^{\,2/(n+2\alpha)}\right\},
\]
with $r=r_K$ if the displayed maximum is larger than $r_K$; that case is
absorbed by increasing $C_K$.  If the stopping scale is selected, then
$X_j\le C\delta_j^{(n+2\alpha)/4}$ and
$r^{-n/2}X_j\le C\delta_j^{\alpha/2}$.  Otherwise the first two terms
balance and equal $CX_j^\gamma$.  This proves
\eqref{hd:eq:vj-L2-Linfty}.
\end{proof}

\begin{lemma}
\label{hd:lem:secant-flux-additivity}
For logarithmic profiles $q_a,q_b$, put $d=q_b-q_a$ and denote
\[
 Q(p):=p\otimes p-\frac12|p|^2I,
 \qquad
 \mathcal B(p,\xi):=p\otimes\xi+\xi\otimes p-\langle p,\xi\rangle I,
\]

\begin{equation}\label{hd:eq:exact-secant-polynomial}
 J(q_a,q_b)
 :=\left[
 T_1(\mathscr A[q_a])
 +\frac12T_1\!\left(\Hess d+\mathcal B(\nabla q_a,\nabla d)\right)
 +\frac13T_1(Q(\nabla d))
 \right]\nabla d.
\end{equation}
Then $J(q_a,q_b)$ is exactly the path-averaged Newton vector field
$\int_0^1T_1(\mathscr A[q_a+td])\,dt\,\nabla d$.

\begin{equation}\label{hd:eq:general-secant-flux}
 \mathfrak F_\eta(q_a,q_b;R)
 :=\int_{A_{R,2R}}\frac1R\eta\!\left(\frac rR\right)
 \langle J(q_a,q_b),\nabla r\rangle\,dV_{g_\infty}.
\end{equation}

For $B\in\mathbb R$, set
\[
 k:=B r^{-m}\psi_*,\qquad \bar q:=q_0+k,\qquad
 h_*:=q-\bar q,\qquad M_*(R):=\sup_{A_R^*}r^m|h_*|.
\]
Assume that the shifted pair $(\bar q,q)$ satisfies the Caccioppoli estimate
of Lemma~\ref{hd:lem:relative-newton-flux}.  Then

\begin{equation}\label{hd:eq:secant-additivity-rate}
 \begin{aligned}
 &\big|\mathfrak F_\eta(q_0,q;R)
 -\mathfrak F_\eta(q_0,\bar q;R)
 -\mathfrak F_\eta(\bar q,q;R)\big|\\
 &\hspace{2cm}\le
 C R^{2-m}\bigl(1+M_*(R)+M_*(R)^2\bigr).
 \end{aligned}
\end{equation}
Moreover,
\begin{equation}\label{hd:eq:smooth-kernel-flux-rate}
 \mathfrak F_\eta(q_0,q_0+Br^{-m}\psi_*;R)
 =B\tau_*+O\!\left(R^{-1}(1+\log R)+R^{2-m}\right).
\end{equation}
Consequently, for $B=B_*:=\mathfrak F_\infty/\tau_*$,
\begin{equation}\label{hd:eq:shifted-flux-zero}
 \lim_{R\to\infty}\mathfrak F_\eta(\bar q,q;R)=0.
\end{equation}
\end{lemma}

\begin{proof}
The identity
$Q(p+t\xi)=Q(p)+t\mathcal B(p,\xi)+t^2Q(\xi)$ and linearity of $T_1$
give \eqref{hd:eq:exact-secant-polynomial} after integration in $t$.

Apply this formula to the three paths
$q_0\to q_0+k+h_*$, $q_0\to q_0+k$, and
$q_0+k\to q_0+k+h_*$.  Their difference is a sum of terms, each containing
one $k$ increment and one $h_*$ increment.  The only estimates needed on the
enlarged annulus are
\[
 |\nabla^\ell k|\le C_BR^{-m-\ell}\ (0\le\ell\le3),\qquad
 \int|\nabla h_*|^2\le CR^{4-m}(1+M_*^2),\qquad
 \int|\nabla h_*|\le CR^5(1+M_*).
\]
The last two follow from the shifted Caccioppoli estimate, coercivity, and
Cauchy--Schwarz.  They bound every first-derivative monomial in the defect by
$CR^{2-m}(1+M_*+M_*^2)$.

Only the term with $D^2h_*$ needs a structural observation.  With
$X_R=R^{-1}\eta(r/R)\nabla r$, its Euclidean part is of the form
\[
 I_R=\int_{A_R^*}\langle T_1(D^2h_*)\nabla k,X_R\rangle\,dx.
\]
Since $\operatorname{div}T_1(D^2h_*)=0$, two integrations by parts give
\[
 |I_R|\le C\int_{A_R^*}|h_*|\,
 |D^2(\nabla k\otimes X_R)|\,dx
 \le C M_*(R)R^{2-m}.
\]
Metric and connection errors require only one integration by parts and gain
the factors $g_\infty-g_{\mathrm E}=O_2(R^{-2})$ or
$\Gamma(g_\infty)=O_1(R^{-3})$.  The same summary bounds therefore prove
\eqref{hd:eq:secant-additivity-rate}, without either the full defect
polynomial or a pointwise Hessian estimate.

\[
 r^4T_1(\mathscr A[q_0])=N(\theta)+O(R^{-1}(1+\log R)),
\]
so the linear contribution of $k=Br^{-m}\psi_*$ is $B\tau_*$ with the stated
error; the remaining terms are quadratic in $k$ and are $O(R^{2-m})$.  This
gives \eqref{hd:eq:smooth-kernel-flux-rate}.  Combining the two estimates with
the original flux limit and $M_*=O(1)$ gives
\eqref{hd:eq:shifted-flux-zero}.

\end{proof}

\subsection{Locking the critical kernel coefficient}

\begin{theorem}
\label{hd:thm:post-resonance-locking}
Under the assumptions of
Theorem~\ref{hd:thm:actual-leading-resonance}, one has
\begin{align}
 \sup_{\theta\in\mathbb S^{n-1}}r^m
 \left|
 \log W-\log P_{\rm res}-B_*r^{-m}\psi_*
 \right|&\longrightarrow0,
 \label{hd:eq:hstar-o}\\
 W&=P_{\rm res}
 +\frac{\mathfrak F_\infty}{\tau_*}r^{-m}\psi_*
 +o(r^{-m}).
 \label{hd:eq:post-resonance-final-W}
\end{align}

\end{theorem}

\begin{proof}

Set $B_*=\mathfrak F_\infty/\tau_*$,
$\bar q=q_0+B_*r^{-m}\psi_*$, and $h_*=q-\bar q$.  By
Theorem~\ref{hd:thm:actual-leading-resonance}, $|h_*|\le Cr^{-m}$.  Suppose
\eqref{hd:eq:hstar-o} fails.  Define
\[
 A(s):=\sup_{\theta\in\mathbb S^{n-1}}e^{ms}|h_*(e^s,\theta)|.
\]
Then $A(s)\le C$, while failure of \eqref{hd:eq:hstar-o} gives
$\limsup_{s\to\infty}A(s)>0$.  Hence there are $s_j\to\infty$, a constant
$a_0>0$, and points $\theta_j\in\mathbb S^{n-1}$ such that
\[
 R_j:=e^{s_j},\qquad A_j:=A(s_j),\qquad
 a_0\le A_j\le C,\qquad
 R_j^m|h_*(R_j,\theta_j)|=A_j.
\]
Set $\delta_j:=A_jR_j^{2-m}\to0$ and
\[
 v_j(x):=\frac{R_j^m}{A_j}h_*(R_jx),
 \qquad
 z_j(\sigma,\theta):=e^{m\sigma}v_j(e^\sigma\theta).
\]
Then
\begin{equation}\label{hd:eq:zj-nonzero-center}
 |z_j(\sigma,\theta)|
 \le \frac{A(s_j+\sigma)}{A_j}
 \le \frac{C}{a_0},
 \qquad
 |z_j(0,\theta_j)|=1.
\end{equation}
The global bound for $A$ gives uniform bounds on every fixed cylinder.
The shifted Caccioppoli estimate gives, for every compact annulus
$K\Subset\mathbb R^n\setminus\{0\}$,
\begin{equation}\label{hd:eq:vj-H1}
 \|v_j\|_{H^1(K)}\le C_K.
\end{equation}
Rellich compactness and a diagonal extraction give
\begin{equation}\label{hd:eq:vj-weak-strong-limit}
 v_j\rightharpoonup v_\infty\quad\text{in }H^1_{\rm loc},
 \qquad
 v_j\longrightarrow v_\infty\quad\text{in }L^2_{\rm loc}.
\end{equation}
The corresponding cylinder function
\[
 z_\infty(\sigma,\theta)
 :=e^{m\sigma}v_\infty(e^\sigma\theta)
\]
is bounded on the whole cylinder.  We identify its equation weakly, before
using any pointwise compactness.  For the conormal information, use the
shifted relative flux
\begin{equation}\label{hd:eq:shifted-flux-zero-recalled}
 \mathfrak F_*(R):=\mathfrak F_\eta(\bar q,q;R),
 \qquad
 \lim_{R\to\infty}\mathfrak F_*(R)=0.
\end{equation}
To identify the weak equation, set
$Q(p):=p\otimes p-\tfrac12|p|^2I$ and
$\varepsilon_j:=R_j^{-2}\delta_j^2$.  The exact quadratic dependence on the
normalized perturbation is
\[
 \mathcal B_{R_j}[t\delta_jv_j]
 =B_{R_j}+t\delta_j\mathcal H_{R_j}[v_j]
   +t^2\varepsilon_jQ(\nabla v_j).
\]
Integrating this polynomial in $t$ gives the normalized path-averaged Newton
field
\begin{equation}\label{hd:eq:normalized-current-expand}
 \begin{aligned}
 \mathcal J_j={}&T_1(B_{R_j})\nabla v_j
 +\frac{\delta_j}{2}T_1(\mathcal H_{R_j}[v_j])\nabla v_j\\
 &+\frac{\varepsilon_j}{3}T_1(Q(\nabla v_j))\nabla v_j
 \end{aligned}
\end{equation}
The rescaled exact path identity gives
\begin{equation}\label{hd:eq:normalized-current-divergence}
 \operatorname{div}\mathcal J_j\longrightarrow0
 \qquad\text{in }\mathcal D'_{\rm loc}.
\end{equation}
Indeed, this is the rescaled exact path identity
\eqref{hd:eq:exact-relative-current}.  Its background term is bounded by
\eqref{hd:eq:finite-jet-residual}; the terms containing
$\nabla\bar q_{R_j}$ have coefficients converging to those of the frozen
Newton field, and the remaining nonlinear terms carry $\delta_j$.  The
Caccioppoli bound makes all of these errors tend to zero distributionally.

\begin{equation}\label{hd:eq:null-lagrangian-vj}
 \big(T_1(D^2v_j)\nabla v_j\big)_i
 =\partial_k(v_{j,i}v_{j,k})
 -\partial_i|\nabla v_j|^2.
\end{equation}
This identity shows that the Hessian part passes distributionally without pointwise
$D^2v_j$ bounds.  The remaining cubic gradient term is controlled by the
weighted quartic estimate.  Indeed, the shifted Caccioppoli inequality gives
$\int_{R_jK}|\nabla h_*|^4\le C_KR_j^{-m}$.  Using
$h_*(R_jx)=A_jR_j^{-m}v_j(x)$ and $n=m+6$, a single change of variables
gives directly
\begin{equation}\label{hd:eq:vj-quartic-weighted}
 \varepsilon_j\int_K|\nabla v_j|^4\le C_KA_j^{-2}\le C_K.
\end{equation}
Together with the uniform $H^1$ bound, H\"older's inequality now gives the
precise smallness required for the cubic part:
\[
 \varepsilon_j\int_K|\nabla v_j|^3
 \le
 \left(\varepsilon_j\int_K|\nabla v_j|^4\right)^{1/2}
 \left(\varepsilon_j\int_K|\nabla v_j|^2\right)^{1/2}
 \le C_K\varepsilon_j^{1/2}\longrightarrow0.
\]
Consequently,
\begin{equation}\label{hd:eq:normalized-current-limit}
 \mathcal J_j\longrightarrow N_\infty\nabla v_\infty
 \quad\text{in }\mathcal D'_{\rm loc}.
\end{equation}
Indeed, the term multiplied by $\delta_j$ tends to zero distributionally by
\eqref{hd:eq:null-lagrangian-vj} and the uniform $H^1$ bound, while the cubic
term tends to zero by \eqref{hd:eq:vj-quartic-weighted}; the leading
coefficients converge locally uniformly to $N_\infty$.
Combining \eqref{hd:eq:normalized-current-divergence} and
\eqref{hd:eq:normalized-current-limit} gives
\begin{equation}\label{hd:eq:vinfty-linear}
 \mathscr L_\infty v_\infty
 =\operatorname{div}(N_\infty\nabla v_\infty)=0
 \qquad\text{weakly on }\mathbb R^n\setminus\{0\}.
\end{equation}
Consequently,
\begin{equation}\label{hd:eq:zinfty-P}
 \mathscr Pz_\infty=0
\end{equation}
by \eqref{hd:eq:cylinder-conjugation}.

For every fixed cylinder shift $\sigma$, the physical radius
$R_je^\sigma$ tends to infinity.  Divide the shifted flux at this radius by
$A_j$ and use \eqref{hd:eq:normalized-current-limit}.  Since
$A_j\ge a_0$ and \eqref{hd:eq:shifted-flux-zero-recalled} holds at every such
radius, the limiting averaged conormal is zero.  The conormal is constant for
a solution of \eqref{hd:eq:zinfty-P}, so
\begin{equation}\label{hd:eq:zinfty-zero-conormal}
 \Phi[z_\infty]=0.
\end{equation}
Lemma~\ref{hd:lem:critical-cylinder-liouville} then gives
$z_\infty\equiv0$.  The same conclusion holds for every subsequential
$L^2_{\rm loc}$ limit.  It follows that
\begin{equation}\label{hd:eq:vj-L2-zero}
 \|v_j\|_{L^2(K')}\longrightarrow0
\end{equation}
on every compact annulus $K'$: otherwise a subsequence with a positive
$L^2$ lower bound would have, by \eqref{hd:eq:vj-H1}, a further strong
$L^2$ limit, which has just been shown to vanish.

Choose $K\Subset K'$ with $\{|x|=1\}\subset K$.  The refined estimate
\eqref{hd:eq:vj-L2-Linfty}, together with
\eqref{hd:eq:vj-L2-zero} and $\delta_j\to0$, yields
\[
 \|v_j\|_{L^\infty(K)}\longrightarrow0.
\]
This contradicts \eqref{hd:eq:zj-nonzero-center}.  Therefore
\eqref{hd:eq:hstar-o} holds.  It gives
$\log W=\log P_{\rm res}+B_*r^{-m}\psi_*+o(r^{-m})$.
Since $P_{\rm res}=1+O(r^{-2})$, exponentiation yields
\[
 W=P_{\rm res}+B_*r^{-m}\psi_*+o(r^{-m}).
\]
Substituting
$B_*=\mathfrak F_\infty/\tau_*$ proves
\eqref{hd:eq:post-resonance-final-W}.

\end{proof}

\section{Higher-order flatness and the moving curvature scale}
\label{sec:high-moving-scale}

Suppose now that the first nonzero Kelvin--Schouten coefficient is $C_q$.
The position of its relative degree $j$ with respect to the bounded scale
$\alpha$
determines the mechanism.  More precisely, we assume
\begin{equation}\label{hd:eq:first-nonzero-Cq}
\begin{gathered}
  (A_{g_\infty})^\sharp
  =R^{-4-q}{C}_q(\theta)+O_1(R^{-5-q}),
  \qquad
  {C}_0=\cdots={C}_{q-1}=0,
  \qquad
  {C}_q\not\equiv0,\\
  j:=q+2,
  \qquad
  \alpha:=\frac{n-4}{2}.
\end{gathered}
\end{equation}

\subsection{The case \texorpdfstring{$j<\alpha$}{j less than alpha}}
\label{hd:subsec:generalized-first-cell}

\begin{theorem}\label{hd:thm:subcritical-first-cell-trapping}
Assume \eqref{hd:eq:first-nonzero-Cq}, with $j<\alpha$, and retain the
standard Green-end regularity and comparison theory.  Then there is a unique
admissible $C^{1,1}$ cell $\phi_j$ such that
\begin{equation}\label{hd:eq:general-cell-limit}
 \mathcal K_j[\phi_j]+C_q\in\partial\Gamma_2,
\end{equation}
and
\begin{equation}\label{hd:eq:general-cell-first-tangent}
 R^j(1-W(R,\theta))\to\phi_j(\theta)
 \qquad\text{uniformly on }\mathbb S^{n-1}.
\end{equation}
Hence
\begin{equation}\label{hd:eq:general-cell-first-expansion-U}
 U_{\rm cn}(\rho,\theta)
 =\rho^{-\alpha}+\frac\alpha2\phi_j(\theta)\rho^{j-\alpha}
 +o(\rho^{j-\alpha}).
\end{equation}
No Newton nondegeneracy is required for this first correction.
\end{theorem}

\begin{proof}
For the strict cell
$S_{t,\varepsilon}=\mathcal K_j[\phi]+tC_q\in\Gamma_2$,
$\sigma_2(S_{t,\varepsilon})=\varepsilon$, write the radial--tangential
blocks as $(a,\beta,D)$.  The exact shift
\[
 B=D+\frac{a}{n-2}g_S
 =-\nabla_S^2\phi+\frac{j(n-j-3)}{n-2}\phi g_S+t\widetilde C_q
\]
satisfies
\[
 \sigma_2(B)=|\beta|^2+\frac{n-1}{2(n-2)}a^2+\varepsilon,
 \qquad B\in\Gamma_2^{n-1}.
\]
Since $j<\alpha$, the zeroth-order coefficient is positive, and the same
standard comparison/rotation/Evans--Krylov/continuity argument as for the
first cell gives a unique strict solution with uniform estimates and hence
the limit $\phi_j$.

The elementary power barriers $1\mp MR^{-j}$ give $|1-W|\le CR^{-j}$.  Perturb a
strict cell $\phi_\varepsilon$ by
$\pm\delta\pm M(R_0/R)^\eta$, $0<\eta<\alpha-j$.  Since
$H_j,H_{j+\eta}\in\Gamma_2^\circ$, the two perturbations have uniform
opposite cone signs; comparison yields
\[
 \phi_\varepsilon-\delta-M(R_0/R)^\eta
 \le R^j(1-W)\le
 \phi_\varepsilon+\delta+M(R_0/R)^\eta.
\]
Let $R\to\infty$, then $\delta\downarrow0$, and expand
$F=W^{-\alpha/2}$.
\end{proof}

\subsection{The critical case \texorpdfstring{$j=\alpha$}{j equals alpha}}

Assume now that $j=\alpha$, equivalently
$q=\alpha-2=(n-8)/2$, and write
$\widehat C_q(x):=|x|^{-\alpha-2}C_q(x/|x|)$ for the homogeneous extension
of the critical curvature cell.

\begin{proposition}\label{hd:prop:critical-kappa}
Let $j=\alpha$, $H=\mathcal K_\alpha[1]$, and $N_0=T_1(H)$.  The radial and
tangential eigenvalues of $H$ give
\[
 (N_0)_{RR}=\alpha(n-1),\qquad
 (N_0)_{TT}=\frac{\alpha(n-2)}2I,\qquad
 N_0:\mathcal K_\alpha[\psi]
 =-\frac{\alpha(n-2)}2\Delta_S\psi.
\]
The leading Bianchi identity gives
\[
 \operatorname{div}T_1(\widehat C_q)=0,
 \qquad \int_{\mathbb S^{n-1}}N_0:C_q=0.
\]
Hence there is a unique zero-mean $\psi_0^\circ$ satisfying
\begin{equation}\label{hd:eq:psi0-Poisson}
 -\frac{\alpha(n-2)}2\Delta_S\psi_0^\circ=-N_0:C_q,
\end{equation}
and, with
\begin{equation}\label{hd:eq:Bcritical-def}
 B^\circ:=C_q+\mathcal K_\alpha[\psi_0^\circ],
\end{equation}
one has $N_0:B^\circ=0$.  Define
\begin{equation}\label{hd:eq:kappaq-def}
 \kappa_q:=-\avg_{\mathbb S^{n-1}}\sigma_2(B^\circ)\,d\theta.
\end{equation}
Then $\kappa_q\ge0$, and $\kappa_q=0$ exactly when
\begin{equation}\label{hd:eq:Cq-removable}
 C_q=\mathcal K_\alpha[c-\psi_0^\circ]
\end{equation}
for some constant $c$.
\end{proposition}

\begin{proof}
The first identities are the leading homogeneous part of the contracted
Bianchi identity.  Solvability of the Poisson equation is immediate.  On
$N_0:B=0$, the polarization of $\sigma_2$ has the standard Lorentzian sign,
so $\sigma_2(B)\le0$, with equality only for
$B\in\operatorname{span}\{H\}$.  Thus $\kappa_q\ge0$; in the equality case,
divergence-freeness of the homogeneous Newton tensor forces the
proportionality factor to be constant.  The converse is immediate.
\end{proof}

\begin{theorem}
\label{hd:thm:critical-kappa-zero}
Assume $j=\alpha$ and $\kappa_q=0$.  Let $c$ be the constant from
\eqref{hd:eq:Cq-removable} and put
\begin{equation}\label{hd:eq:phi-rem-def}
 \phi_{\rm rem}:=\psi_0^\circ-c.
\end{equation}
Then there is a unique $A\ge0$ such that
\begin{equation}\label{hd:eq:critical-kappa-zero-final}
 U_{\rm cn}(\rho,\theta)
 =\rho^{-\alpha}+\frac\alpha2\phi_{\rm rem}(\theta)+A+o(1).
\end{equation}
\end{theorem}

\begin{proof}
By \eqref{hd:eq:Cq-removable},
$\mathcal K_\alpha[\phi_{\rm rem}]+C_q=0$.  Absorb this local term by
$q_{\rm loc}=-R^{-\alpha}\phi_{\rm rem}$,
$\widetilde g_\infty=e^{-2q_{\rm loc}}g_\infty$, and
$\widetilde W=e^{-q_{\rm loc}}W$.  The physical Green metric is unchanged,
and the conformal Schouten formula shows that the order
$R^{-\alpha-2}$ background term cancels, so
$(A_{\widetilde g_\infty})^\sharp=O_1(R^{-\alpha-3})$ while
$\widetilde g_\infty-g_{\mathrm E}=O_2(R^{-2})$.  Strict approximants transform in
the same way.  The finite-charge theorem for the resulting $j>\alpha$
branch gives
$\widetilde W=1-a_\infty R^{-\alpha}+o(R^{-\alpha})$; transforming back and
expanding $F=W^{-\alpha/2}$ gives
\eqref{hd:eq:critical-kappa-zero-final}, with $A=\alpha a_\infty/2$.
\end{proof}

\begin{proposition}\label{hd:prop:critical-model}
Assume $\kappa_q>0$ and put
\begin{equation}\label{hd:eq:Lambda-critical}
 \Lambda_q:=\frac{2\kappa_q}{(n-1)\alpha^2}.
\end{equation}
There are canonically normalized smooth functions $\psi_0,\psi_1,\psi_2$
such that, for
\[
 f_\pm(s)=\Lambda_qs+\lambda_\pm\pm K\log s,\qquad
 a_\pm=\sqrt{f_\pm},
\]
the profiles
\begin{equation}\label{hd:eq:critical-upm-refined}
 u_\pm=a_\pm+\psi_0+\frac{\psi_1}{a_\pm}+\frac{\psi_2}{a_\pm^2}
\end{equation}
have opposite strict cone signs for all sufficiently large $s$, once $K$ is
fixed large.  Moreover every subcritical decay improves to every
$\beta<\alpha$, and
\begin{equation}\label{hd:eq:critical-coarse-bound}
 -C\le R^\alpha(1-W(R\theta))\le C\sqrt{\log R}.
\end{equation}
\end{proposition}

\begin{proof}
For $q_b=-r^{-\alpha}b(\log r)$,
\[
 r^{2\alpha+4}\sigma_2(D^2q_b)
 =(n-1)(\alpha b-b')(\alpha b'-b'').
\]
The zero-mean part of $\sigma_2(B^\circ)$ determines $\psi_1$ through the
critical Laplacian; a constant shift of $\psi_0^\circ$ and one further
Poisson solve determine $\psi_0,\psi_2$.  Exact quadratic polarization gives,
for $f\sim\Lambda_qs$,
\[
 \sigma_2\bigl(\widetilde{\mathcal K}_\alpha[u_f]+C_q\bigr)
 =\frac{(n-1)\alpha^2}{2}(f'-\Lambda_q)+O(f^{-1})+O(|f''|).
\]
Hence $f_\pm$ give opposite signs of size $K/s$, while the curved-background
error is $o(s^{-1})$.  Standard power barriers yield every subcritical
exponent; the radial choice $b^2=B_0^2+\Lambda_+(s-s_0)$ with
$\Lambda_+>\Lambda_q$, together with a constant lower barrier, gives
\eqref{hd:eq:critical-coarse-bound}.
\end{proof}

Write
\[
 W(Rx)=1+R^{-\alpha}P_R^{(0)}(x),
 \qquad L_R:=\log R,
 \qquad \widehat P_R:=P_R^{(0)}/\sqrt{L_R}.
\]
The angular term in the critical model is separated by
\[
 p_0(x):=-|x|^{-\alpha}\psi_0(x/|x|),
 \qquad
 B_\infty:=D^2p_0+\widehat C_q.
\]
The leading Bianchi identity gives
$\operatorname{div}T_1(B_\infty)=0$.  Hence, for
$\widetilde P_R:=P_R^{(0)}-p_0$, the natural modified current and its averaged
flux are
\begin{align*}
 \mathcal J_R^B
 &:=T_1(D^2\widetilde P_R)\nabla\widetilde P_R
   +2T_1(B_\infty)\nabla\widetilde P_R,\\
 \mathcal Q_B(R)
 &:=\int_{A_{1,2}}\eta(\varrho)
   \langle\mathcal J_R^B,\nabla\varrho\rangle\,\dd x.
\end{align*}

\begin{lemma}
\label{hd:lem:critical-modified-current-div}
On a fixed slightly larger annulus, let $\varepsilon_R=R^{-\alpha}$, $W_R=W(Rx)$, and set
\[
 D_R:=\frac{W_R}{\varepsilon_R}(A_{g_R})^\sharp,
 \qquad
 q_R:=\frac{\varepsilon_R}{2W_R}|\nabla P_R^{(0)}|_{g_R}^2,
 \qquad
 \mathbb B_R:=\Hess_{g_R}p_0+D_R.
\]
Then the exact equation for the Green end becomes
\begin{equation}\label{hd:eq:critical-shifted-exact-equation}
 \sigma_2^{g_R}\bigl(\Hess_{g_R}\widetilde P_R+\mathbb B_R-q_RI\bigr)=0.
\end{equation}
Define the modified covariant current
\[
 \mathbb J_R
 :=T_1^{g_R}(\Hess_{g_R}\widetilde P_R)\nabla^{g_R}\widetilde P_R
 +2T_1^{g_R}(\mathbb B_R)\nabla^{g_R}\widetilde P_R.
\]
Then
\begin{align}
 \operatorname{div}_{g_R}\mathbb J_R
 ={}&-2\sigma_2^{g_R}(\mathbb B_R)
 +2(n-1)q_R\tr_{g_R}(\Hess_{g_R}\widetilde P_R+\mathbb B_R)
 \notag\\
 &-n(n-1)q_R^2
 -\Ric_{g_R}(\nabla\widetilde P_R,\nabla\widetilde P_R)
 +2\langle\operatorname{div}_{g_R}T_1^{g_R}(\mathbb B_R),\nabla\widetilde P_R\rangle.
 \label{hd:eq:critical-modified-cov-div}
\end{align}
Moreover,
\begin{align}
 \|\mathbb B_R-B_\infty\|_{C^1}
 &\le C\bigl(R^{-1}+R^{-\alpha}\sqrt{L_R}\bigr),
 \label{hd:eq:critical-BR-approx}\\
 \|\mathbb J_R-\mathcal J_R^B\|_{L^1}
 &\le C\bigl(R^{-1}\sqrt{L_R}+R^{-2}L_R+R^{-\alpha}L_R\bigr).
 \label{hd:eq:critical-current-comparison}
\end{align}
\end{lemma}

\begin{proof}
Substituting $P_R^{(0)}=\widetilde P_R+p_0$ into the normalized
Schouten matrix gives \eqref{hd:eq:critical-shifted-exact-equation}
without discarding any quadratic gradient term.  The covariant Hessian
null-Lagrangian and the product rule give
\[
 \operatorname{div}\mathbb J_R
 =2\bigl[\sigma_2(\Hess\widetilde P_R+\mathbb B_R)-\sigma_2(\mathbb B_R)\bigr]
 -\Ric(\nabla\widetilde P_R,\nabla\widetilde P_R)
 +2\langle\operatorname{div} T_1(\mathbb B_R),\nabla\widetilde P_R\rangle.
\]
Identity
$\sigma_2(A-qI)=\sigma_2(A)-(n-1)q\tr A+\binom n2q^2$, together with
\eqref{hd:eq:critical-shifted-exact-equation}, proves
\eqref{hd:eq:critical-modified-cov-div}.

The finite background expansion gives
$(A_{g_R})^\sharp=\varepsilon_R(B_\infty-D^2p_0)
+O_1(\varepsilon_RR^{-1})$.  Combining it with
$W_R-1=O(\varepsilon_R\sqrt{L_R})$ and
$g_R-g_{\mathrm E}=O_{C^2}(R^{-2})$ proves
\eqref{hd:eq:critical-BR-approx}.  Finally the modified current is
linear in the Hessian.  The bounds
$\|\nabla P_R^{(0)}\|_\infty+
\|\Hess P_R^{(0)}\|_{L^1}=O(\sqrt{L_R})$ control the metric/Hessian
difference, the replacement of $\mathbb B_R$ by $B_\infty$, and the
volume and gradient conversions, yielding
\eqref{hd:eq:critical-current-comparison}.  No $L^2$ Hessian estimate is
used.
\end{proof}

\begin{lemma}
\label{hd:lem:critical-modified-charge-drift}
As $R\to\infty$, uniformly for $1\le t\le2$,
\begin{equation}\label{hd:eq:critical-modified-charge-drift}
 \mathcal Q_B(tR)-\mathcal Q_B(R)
 =2|\mathbb S^{n-1}|\kappa_q\log t+\mathcal E_B(R,t),
\end{equation}
where
\begin{equation}\label{hd:eq:critical-modified-charge-error}
 \sup_{1\le t\le2}|\mathcal E_B(R,t)|
 \le C\left(
 R^{-1}\sqrt{L_R}
 +R^{-2}L_R
 +R^{-\alpha}L_R^{3/2}
 \right).
\end{equation}
Consequently,
\begin{equation}\label{hd:eq:critical-charge-growth}
 \frac{\mathcal Q_B(R)}{\log R}
 \longrightarrow2|\mathbb S^{n-1}|\kappa_q.
\end{equation}

\end{lemma}

\begin{proof}
The scaling relation is
\[
 \widetilde P_{tR}(x)=t^\alpha\widetilde P_R(tx),
\]
and the cutoff defining the charge converts the difference of two
scales into
\[
 \mathcal Q_B(tR)-\mathcal Q_B(R)
 =-\int\phi_t\,\operatorname{div}\mathcal J_R^B\,\dd x.
\]
The leading source term in
\eqref{hd:eq:critical-modified-cov-div} is evaluated from
\[
 \sigma_2(B_\infty)=r^{-n}\sigma_2(B(\theta)),
 \qquad
 \avg_{\mathbb S^{n-1}}\sigma_2(B)=-\kappa_q,
 \qquad
 \int_0^\infty\phi_t(r)\frac{\dd r}{r}=-\log t
\]
and equals $2|\mathbb S^{n-1}|\kappa_q\log t$ with the displayed
orientation.  The remaining terms are estimated using
$\|\nabla P_R^{(0)}\|_\infty=O(\sqrt{L_R})$,
$\|\Hess P_R^{(0)}\|_{L^1}=O(\sqrt{L_R})$, and
$q_R=O(R^{-\alpha}L_R)$.  Together with
\eqref{hd:eq:critical-BR-approx}--\eqref{hd:eq:critical-current-comparison},
the standard product estimates give exactly the three terms on the
right-hand side of \eqref{hd:eq:critical-modified-charge-error}; their
individual tensor contractions are routine and are omitted.  Summing the
$t=2$ identity over dyadic scales gives a convergent accumulated error
after division by $\log R$; the residual interval is covered by the
uniform estimate for $1\le t\le2$.  This proves
\eqref{hd:eq:critical-charge-growth}.
\end{proof}

\begin{theorem}
\label{hd:thm:critical-moving-endpoint}
Let $n>8$ be even and assume \eqref{hd:eq:first-nonzero-Cq}, with
$j=q+2=\alpha$.  If $\kappa_q>0$, then in the fixed conformal-normal gauge
\begin{align}
 W(R,\theta)
 &=1-R^{-\alpha}\left(
 \sqrt{\Lambda_q\log R}+\psi_0(\theta)+o(1)
 \right),
 \label{hd:eq:critical-moving-final-W}\\
 U_{\rm cn}(\rho,\theta)
 &=\rho^{-\alpha}
 +\sqrt{\frac{\kappa_q}{2(n-1)}}
  \sqrt{\log\frac1\rho}
 +\frac\alpha2\psi_0(\theta)
 +o(1).
 \label{hd:eq:critical-moving-final-U}
\end{align}
\end{theorem}

\begin{proof}
The bound \eqref{hd:eq:critical-coarse-bound} and the standard Bernstein/$L^1$
Hessian estimates give
local compactness of $\widehat P_R$.  Every subsequential limit solves the
flat boundary equation, is nonpositive, and has growth $O(|x|^{-\alpha})$;
Trudinger--Wang therefore gives
$\widehat P_{R_j}\to-c|x|^{-\alpha}$ for some $c\ge0$.  Divide the modified
current by $L_{R_j}$.  The terms involving the fixed field $B_\infty$
vanish, while the Hessian null-Lagrangian current of the radial limit
converges by the uniform gradient and $L^1$ Hessian bounds.  Its flux is
\[
 \lim\frac{\mathcal Q_B(R_j)}{L_{R_j}}
 =(n-1)|\mathbb S^{n-1}|\alpha^2c^2.
\]
Comparing this identity with
\eqref{hd:eq:critical-charge-growth} gives
\[
 c^2=\frac{2\kappa_q}{(n-1)\alpha^2}=\Lambda_q.
\]
Every subsequential tangent therefore has the same coefficient
$c=\sqrt{\Lambda_q}$.  Compactness upgrades this to the full convergence
\begin{equation}\label{hd:eq:critical-hat-full-convergence}
 \widehat P_R(x)
 \longrightarrow-\sqrt{\Lambda_q}\,|x|^{-\alpha}
\end{equation}
locally uniformly on $\mathbb R^n\setminus\{0\}$.  Restricting to $|x|=1$
and using $u(s,\theta)=-P_{e^s}^{(0)}(\theta)$ yields
\begin{equation}\label{hd:eq:critical-u-leading-lock}
 \frac{u(s,\theta)}{\sqrt s}
 \longrightarrow\sqrt{\Lambda_q}
\end{equation}
uniformly in $\theta$.

With this leading tangent identified, choose the constants
$\lambda_\pm$ in the strict profiles \eqref{hd:eq:critical-upm-refined} so
that they bracket $W$ on one large sphere.  Exterior comparison then gives
\[
 u(s,\theta)=\sqrt{\Lambda_qs}+\psi_0(\theta)+o(1),
\]
which is \eqref{hd:eq:critical-moving-final-W}.  Since $F=W^{-\alpha/2}$ and $R^{-2\alpha}\log R=o(R^{-\alpha})$, Taylor expansion gives

\[
 F=1+\frac\alpha2R^{-\alpha}
 \left(\sqrt{\Lambda_q\log R}+\psi_0+o(1)\right)
 +o(R^{-\alpha}).
\]
Using $U_{\rm cn}=\rho^{-\alpha}F$, $R=\rho^{-1}$, and $(\alpha/2)\sqrt{\Lambda_q}=\sqrt{\kappa_q/(2(n-1))}$ now gives \eqref{hd:eq:critical-moving-final-U}.

\end{proof}

\subsection{The finite-charge case \texorpdfstring{$j>\alpha$}{j greater than alpha}}

\label{hd:subsec:weylflat-finite-charge}

\begin{equation}\label{hd:eq:finite-charge-jalpha}
  j>\alpha,
  \qquad
  |(A_{g_\infty})^\sharp|
  +R|\nabla(A_{g_\infty})^\sharp|
  \le C R^{-j-2}.
\end{equation}

\begin{theorem}
\label{hd:thm:finite-charge-weylflat}
Under \eqref{hd:eq:finite-charge-jalpha}, there is a unique $A\ge0$ such that
\begin{equation}\label{hd:eq:finite-charge-final-U}
  U_{\rm cn}(\rho,\theta)
  =\rho^{-\alpha}+A+o(1)
\end{equation}
uniformly in $\theta$ as $\rho\downarrow0$.
The same conclusion holds if the end is infinitely flat, in the sense that
for every $N>0$,
\begin{equation}\label{hd:eq:infinite-flat-background}
 |(A_{g_\infty})^\sharp|
 +R|\nabla(A_{g_\infty})^\sharp|
 \le C_N R^{-N}
\end{equation}
for all sufficiently large $R$.
\end{theorem}

\begin{proof}
Assume first \eqref{hd:eq:finite-charge-jalpha}.  We establish the
critical-scale bound.  The power and mixed barriers
from Proposition~\ref{prop:low-barrier-package}, with the $j$-th curvature
term absorbed into the faster error, give, for every $\beta<\alpha$ and every
$\alpha<\delta<\min\{j,n-2\}$,
\[
 |W-1|\le C_\beta R^{-\beta},
 \qquad
 (W-1)_+\le C_\delta R^{-\delta}.
\]
If
\[
 \mathcal M(R):=\max_{\theta\in\mathbb S^{n-1}}R^\alpha(1-W(R,\theta)),
 \qquad
 \mathcal M_+(R):=\max\{\mathcal M(R),0\},
\]
the same barriers give the critical recurrence
\begin{equation}\label{hd:eq:finite-charge-dyadic}
 \mathcal M_+(2R)
 \le \mathcal M_+(R)+CR^{\alpha-j}.
\end{equation}
Since $j>\alpha$, the increment is summable on dyadic scales.  Together with
the faster positive-part estimate, this yields
\begin{equation}\label{hd:eq:finite-charge-critical-bound}
 |W(R,\theta)-1|\le CR^{-\alpha}
\end{equation}
for all sufficiently large $R$, uniformly in $\theta$.  We may therefore set
\[
  W(Rx)=1+R^{-\alpha}P_R(x).
\]
On each fixed annulus the normalized curvature term is
$O(R^{\alpha-j})$ and $g_R\to g_{\mathrm E}$ in $C^3$.  The compactness and
Hessian-measure argument of Sections~\ref{sec:low-compactness}--\ref{sec:low-tangents}
therefore applies.  The faster positive-part estimate makes every limit
nonpositive, so each tangent is $-c|x|^{-\alpha}$ for some $c\ge0$.

The calculation of Proposition~\ref{prop:low-charge-drift} also carries over;
the only new slow term is the background polarization
$O(R^{\alpha-j})$.  Consequently, uniformly for $1\le t\le2$,
\begin{equation}\label{hd:eq:finite-charge-drift}
  |Q(tR)-Q(R)|
  \le C\bigl(R^{-\alpha}+R^{-2}+R^{\alpha-j}\bigr).
\end{equation}
The drift is summable.  Passing the limiting charge to any tangent gives
\[
  Q[-c|x|^{-\alpha}]
  =(n-1)|\mathbb S^{n-1}|\alpha^2c^2.
\]
Thus all tangents have the same $c$ and
$W=1-cR^{-\alpha}+o(R^{-\alpha})$.  Expanding
$F=W^{-\alpha/2}$ and using $U_{\rm cn}=\rho^{-\alpha}F$ proves
\eqref{hd:eq:finite-charge-final-U} with the unique coefficient
$A=\alpha c/2\ge0$.
If \eqref{hd:eq:infinite-flat-background} holds, choose
$\alpha<j<n-2$ and take $N=j+2$; then
\eqref{hd:eq:finite-charge-jalpha} holds, so the preceding argument applies.
\end{proof}

\subsection{Completion of the proofs}

Theorem~\ref{thm:main-working}\textup{(iii)} follows from
Theorem~\ref{hd:thm:actual-to-n6}, Proposition~\ref{hd:prop:formal-single-log},
and Theorem~\ref{hd:thm:post-resonance-locking}.  The four cases of
Theorem~\ref{thm:main-high-flatness} are respectively
Theorems~\ref{hd:thm:subcritical-first-cell-trapping},
\ref{hd:thm:critical-moving-endpoint}, \ref{hd:thm:critical-kappa-zero}, and
\ref{hd:thm:finite-charge-weylflat}.
Since the pole $p\in S$ was arbitrary, the same conclusions hold at every
pole.  We denote the resulting pole-dependent free constants by $A_p$ in the
statements of the main theorems.

\phantomsection\label{sec:references-start}

\end{document}